\documentclass[11pt]{amsart}
\usepackage[T1]{fontenc}
\usepackage{lmodern,microtype}
\usepackage{amsmath,amssymb,amsthm,mathtools}
\usepackage[margin=1in]{geometry}
\usepackage{booktabs,enumitem}
\usepackage{xcolor}
\usepackage[colorlinks=true,linkcolor=blue!45!black,citecolor=blue!45!black,urlcolor=blue!45!black]{hyperref}

\newtheorem{theorem}{Theorem}[section]
\newtheorem{proposition}[theorem]{Proposition}
\newtheorem{lemma}[theorem]{Lemma}
\newtheorem{corollary}[theorem]{Corollary}
\newtheorem{hypothesis}[theorem]{Hypothesis}

\newtheorem*{maintheorem}{Main Theorem}
\theoremstyle{definition}

\theoremstyle{remark}
\newtheorem{remark}[theorem]{Remark}
\newtheorem*{unnumberedremark}{Remark}

\newcommand{\Hh}{\mathbb H}
\newcommand{\C}{\mathbb C}
\newcommand{\A}{\C[[t]]}
\newcommand{\K}{\C((t))}
\newcommand{\IC}{\operatorname{IC}}
\newcommand{\Gr}{\operatorname{gr}}

\newcommand{\Ext}{\operatorname{Ext}}

\title{Jantzen filtrations for graded affine Hecke algebras}
\author{Dan Ciubotaru}
\address{Mathematical Institute, University of Oxford, OX2 6GG, UK}
\email{dan.ciubotaru@maths.ox.ac.uk}
\date{August 2026}
\hypersetup{
  pdftitle={Jantzen filtrations for graded affine Hecke algebras},
  pdfauthor={Dan Ciubotaru}
}

\begin{document}
\raggedbottom
\begin{abstract}
We prove the Jantzen conjecture for standard modules of equal-parameter
graded affine Hecke algebras with real central character in the dominant
deformation direction.  The Jantzen levels are semisimple, and the
multiplicities are given by the full local intersection cohomology of the
corresponding orbit closures.  The proof identifies the algebraic
standard-to-contragredient map with the normal costalk-to-stalk map
on a contracting transverse slice.  Hard Lefschetz determines its
elementary divisors, and a residual grading on the specialized convolution
algebra proves semisimplicity.  An example in type $A$ is included to show that dominance is necessary.
\end{abstract}

\maketitle

\section{Introduction}

Let $G$ be a connected complex reductive group and $T\subset G$ a maximal
torus.  Put $\mathfrak g=\operatorname{Lie}G$,
$\mathfrak t=\operatorname{Lie}T$, and $R=R(G,T)$.  
Fix a positive system $R^+\subset R$, with simple roots $\Delta$, and put
$W=N_G(T)/T$ and
$\mathfrak t_{\mathbb R}=X_*(T)\otimes_{\mathbb Z}\mathbb R$.  Let
$\mathbf c:R\to\mathbb C$ be a $W$-invariant parameter function,
$\alpha\mapsto c_\alpha$, let $r\in\mathbb C$, and put
$k_\alpha=c_\alpha r$.  In the Bernstein presentation, the graded affine
Hecke algebra with parameter function $k$ is
\[
 \Hh(R,k)=
 \mathbb C[W]\otimes_{\mathbb C}\operatorname{Sym}(\mathfrak t^*)
 \qquad\text{as a vector space}.
\]
For $w\in W$, let $N_w$ denote the corresponding element of the group-algebra
subalgebra $\mathbb C[W]\subset\Hh(R,k)$.  Thus
$N_wN_{w'}=N_{ww'}$, and multiplication between the group and polynomial
subalgebras is determined by
\[
 \xi N_{s_\alpha}-N_{s_\alpha}s_\alpha(\xi)
 =c_\alpha r\,\langle\xi,\alpha^\vee\rangle
 \qquad(\alpha\in\Delta,\ \xi\in\mathfrak t^*),
\]
where
$s_\alpha(\xi)=\xi-\langle\xi,\alpha^\vee\rangle\alpha$.
This is the convention of
\cite[Section~2, equation~(4)]{SolleveldInduced}.  For the standard principal
Springer datum $(T,\{0\},\mathbf 1)$ one has $c_\alpha=2$ for every root.
Set $r=\tfrac12$, so that $c_\alpha r=1$, and write
$\Hh=\Hh(R,1)$ for the resulting equal-parameter algebra; its cross relation
has right-hand side $\langle\xi,\alpha^\vee\rangle$.  This parameter-one
algebra is used throughout.

A central character is called \emph{real} when its $W$-orbit meets
$\mathfrak t_{\mathbb R}$.  Fix such a character, choose a 
representative $\sigma\in\mathfrak t_{\mathbb R}$ of its $W$-orbit, and use
the parameter $(\sigma,r)$.  (In the
specialization just fixed, $r=\tfrac12$ and hence $2r=1$.)
Put
\[
 \mathfrak X_{\sigma,r}=\mathfrak g_{\sigma,2r}
 =\{v\in\mathfrak g:[\sigma,v]=2rv\},
 \qquad G_\sigma=Z_G(\sigma).
\]

At this central character, the relevant parameters are the enhanced
$G_\sigma$-orbits in $\mathfrak X_{\sigma,r}$ whose intersection complexes
occur in the localized principal Springer complex, denoted by
$K_{\sigma,r}$.  Equivalently, they are the enhanced orbits
satisfying the conditions of
\cite[Proposition~4.1]{Solleveld}.  Label the enhanced orbits by
\[
 \gamma=(\mathcal O_\gamma,\mathcal L_\gamma).
\]
The geometric construction of the Hecke algebra modules from cuspidal local systems goes
back to Lusztig's foundational series
\cite{LusztigCuspidalI,LusztigCuspidalII,LusztigCuspidalIII}.  In the
enhanced-orbit and normalization conventions used here,
\cite[Theorem~3.2(a),(c) and Theorem~4.2(a)]{Solleveld} says that these labels
parametrize the geometric standard modules $X_\gamma$ and their unique
irreducible quotients $L_\gamma$.
\begin{samepage}
For a label $\gamma$, define its dual label
\[
 \gamma^\dagger=(\mathcal O_\gamma,\mathcal L_\gamma^\vee).
\]
For a module or lattice $V$, the notation $V^\vee$ means its contragredient.
Thus $X_{\gamma^\dagger}^\vee$ involves two distinct operations: dualizing the
orbit local system and then taking the contragredient of the resulting
standard module.
In the principal Springer realization used in the main theorem, the
cuspidal datum $(T,\{0\},\mathbf1)$ is self-dual.  This identifies the
principal block with its dual block, but it does not require each individual
enhancement to be self-dual: duality sends the label $\gamma$ to
$\gamma^\dagger$.  Accordingly, the canonical map below has the
dual-labelled contragredient as its target.  When
$\mathcal L_\gamma^\vee\simeq\mathcal L_\gamma$, one has
$\gamma^\dagger=\gamma$ and this is the usual map from a standard module to
its own contragredient.  The dagger notation is retained throughout so that
the general statement does not depend on this additional self-duality.
\end{samepage}

The geometric realization and the analytic Langlands classification use
opposite signs for the central parameter.  Section~\ref{sec:fixed-slice-comparison} gives
their precise comparison by the Iwahori--Matsumoto involution, including its
effect on the inducing module, the parabolic, and the deformation direction.
Under this correspondence, the same
symbols $X_\gamma$ and $L_\gamma$ denote the geometric modules and their
analytic Iwahori--Matsumoto transforms, according to the convention in
force.  Lusztig's reduction
theorem transfers the result to affine Hecke algebras at real central
character; see~\cite[Theorem~9.3 and Section~10.9]{LusztigReduction}.

The deformation ring is $\mathbb C[[t]]$, with fraction field
$\mathbb C((t))$.  Write the analytic Langlands datum for $X_\gamma$ as
\[
 X_\gamma=X(M,\delta_{\rm an},\nu),
\]
where $M$ is the Langlands Levi, $\delta_{\rm an}$ is the irreducible
tempered inducing
module for the derived Levi factor, and
$\nu\in\mathfrak z(\mathfrak m)$ is the real unramified central twist; here
$\mathfrak m=\operatorname{Lie}M$.  Thus the fixed noncentral infinitesimal
character is carried by $\delta_{\rm an}$, while the full central variable is
carried by $\nu$.  Choose
\[
 \eta_{\rm an}\in X_*(Z(M)^\circ)\otimes_{\mathbb Z}\mathbb Q
\]
in the interior of its dominant Langlands chamber.  Rescale $\eta_{\rm an}$ by a
positive integer so that it is integral.  Positive rescaling replaces the deformation coordinate
by a positive scalar multiple and does not change the $t$-adic filtration.
For a Levi subgroup $M$, a direction
$\eta\in X_*(Z(M)^\circ)\otimes_{\mathbb Z}\mathbb Q$ is called
\emph{regular relative to $M$} when
\[
 \langle\alpha,\eta\rangle\ne0\quad(\alpha\in R\setminus R_M),
 \qquad\text{equivalently}\qquad
 \operatorname{Stab}_W(\eta)=W_M.
\]
Every direction in the interior of a Langlands chamber is regular in this
relative sense.  This is the meaning of ``regular'' for every deformation
direction in the paper; it never means that the full $W$-stabilizer is
trivial.

Put
\[
 S_M=\operatorname{Sym}(\mathfrak z(\mathfrak m)^*).
\]
Let $\Hh_M\subset\Hh$ be the corresponding parabolic graded affine Hecke
subalgebra and let $\Hh_{M,\mathrm{der}}$ be its derived-root-system factor.
In the principal untwisted setting used here,
\[
 \Hh_M\simeq\Hh_{M,\mathrm{der}}\otimes_{\mathbb C}S_M.
\]
Let
$\delta_{{\rm an},S_M}=\delta_{\rm an}\otimes_{\mathbb C}S_M$
be the universal unramified-twist family: the derived graded-Hecke factor
acts on $\delta_{\rm an}$, while
$\zeta\in\mathfrak z(\mathfrak m)^*$ acts by
\[
 \zeta\cdot(v\otimes f)=v\otimes(\zeta f).
\]
Form the universal standard family
\[
 \mathcal X_{M,S_M}
   =\Hh\otimes_{\Hh_M}\delta_{{\rm an},S_M}.
\]
The formal arc defines a homomorphism
\[
 \operatorname{ev}_{\nu,\eta_{\rm an}}:S_M\longrightarrow\mathbb C[[t]],
 \qquad f\longmapsto f(\nu+t\eta_{\rm an}).
\]
The standard lattice is, by definition, its pullback
\begin{equation}\label{eq:standard-lattice-definition}
 X_{\gamma,\A}
  =\mathbb C[[t]]
    \otimes_{S_M,\operatorname{ev}_{\nu,\eta_{\rm an}}}\mathcal X_{M,S_M}.
\end{equation}
Consequently, on the pulled-back inducing lattice, a central generator
$\zeta$ acts by the scalar
$\zeta(\nu+t\eta_{\rm an})=
\zeta(\nu)+t\langle\zeta,\eta_{\rm an}\rangle$.
It is finite free over $\mathbb C[[t]]$, and base change gives
\[
 X_{\gamma,\A}/tX_{\gamma,\A}\cong X(M,\delta_{\rm an},\nu)=X_\gamma.
\]
Applying the same construction to the dual-enhancement standard module,
with its corresponding dual inducing datum and the same oriented central
arc, defines the lattice $X_{\gamma^\dagger,\A}$ used below.
Since $\K=\A[t^{-1}]$, localization commutes with parabolic induction:
\[
 X_{\gamma,\A}\otimes_{\A}\K
 \simeq
 (\K\otimes_{\mathbb C}\Hh)
 \otimes_{\K\otimes_{\mathbb C}\Hh_M}
 (\delta_{\rm an}\otimes_{\mathbb C}\K)_{\nu+t\eta_{\rm an}},
\]
where the subscript means that $\zeta\in\mathfrak z(\mathfrak m)^*$ acts by
$\zeta(\nu)+t\langle\zeta,\eta_{\rm an}\rangle\in\K$.  Thus the generic fibre is a
standard module for the scalar-extended algebra
$\K\otimes_{\mathbb C}\Hh$, with $\K$-valued twist
$\nu+t\eta_{\rm an}$.

If $V$ is a finite free $\A$-lattice,
its contragredient is
\[
 V^\vee=\operatorname{Hom}_{\A}(V,\A),
\]
with the Hecke action defined by the linear transpose anti-involution
\[
 N_w^*=N_{w^{-1}},\qquad \xi^*=\xi
 \qquad(w\in W,\ \xi\in\mathfrak t^*)
\]
through $(h\cdot f)(v)=f(h^*v)$.  Every later occurrence of $h\mapsto h^*$
refers to this convention.

Over $\K$, the standard-to-contragredient intertwiner is determined only up
to a scalar in $\K^\times$.  An \emph{integral normalization} is a scalar
multiple that restricts to an $\A$-linear morphism
\[
 I_t:X_{\gamma,\A}\longrightarrow X_{\gamma^\dagger,\A}^{\vee},
 \qquad I_t\otimes_{\A}\K\text{ an isomorphism}.
\]
It is \emph{primitive} when its reduction modulo $t$ is nonzero, equivalently
when $I_t(X_{\gamma,\A})$ is not contained in
$tX_{\gamma^\dagger,\A}^{\vee}$.  In any $\A$-bases, this says that its
matrix has entries in $\A$ and at least one entry is a unit.  Removing the
common power of $t$ from any integral matrix produces the primitive
normalization, uniquely up to $\A^\times$.  Multiplication by an $\A$-unit
does not affect the filtration below.
Primitivity does not mean that $I_t$ is an isomorphism over $\A$.  Since it
becomes an isomorphism over $\K$, its cokernel is $t$-power torsion, and Smith
normal form gives diagonal entries
\[
 t^{a_1},\ldots,t^{a_N}\qquad(a_i\geq0,\ \min_i a_i=0).
\]
The map is an $\A$-isomorphism exactly when every $a_i$ is zero.  Positive
exponents occur when the specialization at $t=0$ loses rank; these exponents
are the elementary divisors measured by the Jantzen filtration.
Proposition~\ref{prop:general-coefficient-arc} compares this normalization
with the geometric canonical map.

The Jantzen filtration on $X_\gamma$ is
\begin{equation}\label{eq:jantzen}
 J^nX_\gamma=
 \frac{\{x\in X_{\gamma,\A}: I_t(x)\in
 t^nX_{\gamma^\dagger,\A}^{\vee}\}+tX_{\gamma,\A}}
 {tX_{\gamma,\A}}\qquad(n\geq0).
\end{equation}
Set
$\operatorname{gr}_J^nX_\gamma=J^nX_\gamma/J^{n+1}X_\gamma$.

For each enhanced orbit $\gamma$, choose a point
$x_\gamma\in\mathcal O_\gamma$.  Let
$i_\gamma:\{x_\gamma\}\hookrightarrow\mathfrak X_{\sigma,r}$ be its
inclusion, and put
\[
 A_\gamma=\pi_0\bigl(Z_{G_\sigma}(x_\gamma)\bigr).
\]
Let $\rho_\gamma$ be the irreducible $A_\gamma$-representation
corresponding to $\mathcal L_\gamma$.  For another enhanced orbit $\delta$, set
\[
 \IC_\delta=\IC(\overline{\mathcal O_\delta},\mathcal L_\delta),
\]
where the intersection complex is normalized to be perverse.  For every
enhanced orbit $\gamma$, put
\[
 d_\gamma=\dim_{\mathbb C}\mathcal O_\gamma.
\]
The costalk-to-stalk comparison introduces the dual label
$\delta^\dagger=(\mathcal O_\delta,\mathcal L_\delta^\vee)$.  For a formal
variable $u$, define
\begin{equation}\label{eq:full-ic-polynomial}
 \widehat P_{\gamma\delta}(u)=
 \sum_{j\geq0}
 \dim\operatorname{Hom}_{A_\gamma}\!\left(
   \rho_\gamma^\vee,
   \mathcal H^{-d_\delta+j}i_\gamma^*\IC_{\delta^\dagger}
 \right)u^j.
\end{equation}
The lower IC-vanishing bound gives
$\mathcal H^k i_\gamma^*\IC_{\delta^\dagger}=0$ for $k<-d_\delta$, so no
negative value of $j$ contributes.  The perverse support condition gives
$\mathcal H^k i_\gamma^*\IC_{\delta^\dagger}=0$ for $k>-d_\gamma$
\cite[Section~2.2]{BBD}.  Thus every nonzero summand above satisfies
$j\leq d_\delta-d_\gamma$, and in particular
\[
 v^{d_\delta-d_\gamma}\widehat P_{\gamma\delta}(v^{-1})
 \in\mathbb Z_{\geq0}[v].
\]
The label $\delta^\dagger$ occurs because Verdier duality gives
$\mathbb D\IC_\delta=\IC_{\delta^\dagger}$; the representation
$\rho_\gamma^\vee$ extracts the corresponding dual isotypic summand.  These
two dualizations match the realization of a standard module by a costalk and
of its contragredient by a stalk.
Lusztig's geometric multiplicity theorem
\cite[Proposition~10.5 and Corollary~10.7]{LusztigCuspidalII}, in the
normalization and enhanced-orbit notation of
\cite[Proposition~5.1(a)]{Solleveld}, is
\[
 [X_\gamma:L_\delta]=\widehat P_{\gamma\delta}(1).
\]
The Jantzen conjecture refines this multiplicity identity by placing each
copy of $L_\delta$ in a semisimple
layer; compare~\cite[Conjecture~6.2.2]{BC}.

\begin{maintheorem}
Let $\Hh$ be an equal-parameter graded affine Hecke algebra as above, and use
the standard principal Springer parametrization just fixed.  Let $X_\gamma$ have
real central character.  For every rational cocharacter deformation
direction in the interior of the dominant Langlands chamber, the
successive quotients of the Jantzen filtration of $X_\gamma$ are semisimple.
With the Langlands quotient placed in layer zero, their multiplicities satisfy,
for an indeterminate $v$,
\begin{equation}\label{eq:layer-polynomial}
 \sum_{n\geq0}[\operatorname{gr}_J^nX_\gamma:L_\delta]v^n
 =v^{d_\delta-d_\gamma}\widehat P_{\gamma\delta}(v^{-1}).
\end{equation}
\end{maintheorem}

This is Theorem~\ref{thm:equal-parameter-jantzen}.  A stalk term indexed by
$j$ contributes to layer
\[
 n=d_\delta-d_\gamma-j.
\]
The support bound above guarantees that this integer is nonnegative.  Thus
\eqref{eq:layer-polynomial} records the complete local IC grading,
including the orbit-dimension translation that places the Langlands quotient
in layer zero.

\subsection*{The geometric comparison}

The geometric direction corresponding to the analytic arc is
$\eta_{\rm g}=-\eta_{\rm an}$, positive for the opposite attractive
parabolic.  The parameter $t$ is unchanged.  Let $x=x_\gamma$ and choose
a homogeneous triple $(x,h,f)$ with $\sigma-rh$ centralizing the triple.
The graded Slodowy slice is
\[
 S_x=x+\bigl(\ker(\operatorname{ad}f)\cap\mathfrak g_{\sigma,2r}\bigr).
\]
Proposition~\ref{prop:general-graded-slodowy} proves its transversality
and positivity for the chosen cocharacter.

There are two different restriction maps at $x$.  The ordinary point
adjunction includes the self-intersection factor of the orbit direction;
it is not the map used in this paper.  Instead, restrict along
$\mathcal O_\gamma$ and take the fibre at $x$.  With $j:\mathcal O_\gamma
\hookrightarrow\mathfrak X_{\sigma,r}$ and $l:\{x\}\hookrightarrow
\mathcal O_\gamma$, the normalized fibres of $K=K_{\sigma,r}$ are
\[
 \mathsf N_x^!K=l^*j^!K[-d_\gamma](-d_\gamma),\qquad
 \mathsf N_x^*K=l^*j^*K[-d_\gamma](-d_\gamma).
\]
The orbit adjunction induces $c_x^N:\mathsf N_x^!K\to\mathsf N_x^*K$.
On a transverse product neighbourhood this is exactly the vertex
costalk-to-stalk map of the normalized slice object.  This comparison,
proved in \eqref{eq:normal-slice-arrow}, removes no Euler class by
division.

Put $C_\chi=Z_G(\sigma)\times\mathbb C^\times$ and
$H=C_{\chi,x}^\circ$.  The fibres are equivariant for the stabilizer,
not for $C_\chi$.  Complete their $H$-equivariant hypercohomology at
$\chi=(\sigma,r)$ and pull back along the arc.  The coefficient homomorphism
is component-invariant: Proposition~\ref{prop:two-component-groups}
proves this using a reductive Levi factor of the stabilizer.  No
pointwise invariance of the Lie-algebra arc under its unipotent radical
is required.  After this base change,
the component action is linear, and its $\rho_\gamma$-multiplicity
spaces define the lattices
$\mathsf L_{!,\gamma}^{\eta_{\rm g}}$ and
$\mathsf L_{*,\gamma}^{\eta_{\rm g}}$.
The comparison with algebraic induction is
\begin{equation}\label{eq:master-comparison}
\begin{array}{ccc}
 X_{\gamma,\A}&\xrightarrow{\ I_t\ }&X_{\gamma^\dagger,\A}^{\vee}\\[2mm]
 {\scriptstyle\Phi_!}\downarrow&&\downarrow{\scriptstyle\Phi_*}\\[2mm]
 \mathsf L_{!,\gamma}^{\eta_{\rm g}}&
 \xrightarrow{\ c_\gamma^N\ }&
 \mathsf L_{*,\gamma}^{\eta_{\rm g}}.
\end{array}
\end{equation}
The vertical maps are isomorphisms over $\A$; after multiplying one by
an $\A$-unit the diagram commutes.  The lower map is primitive because
the head summand is point-supported on the slice.
Proposition~\ref{prop:general-coefficient-arc} proves the integral
comparison.

For a positive-dimensional slice, let $A_{\eta_{\rm g}}$ be the effective
contracting torus, and choose the
positive character $\xi_+$.  Its equivariant class
\[
 z=c_1(\xi_+)\in H^2_{A_{\eta_{\rm g}}}(\mathrm{pt}),\qquad
 \deg z=2,
\]
is ample on the projectivized punctured slice.  Under the coefficient arc,
\[
 z\longmapsto c_{\eta_{\rm g}}t,\qquad
 c_{\eta_{\rm g}}=\langle\xi_+,\eta_{\rm g}\rangle>0.
\]
For a perverse IC summand $I$ on the slice, the cone calculation gives
\[
 \operatorname{coker}(c_I)
 \simeq\mathbb H^\bullet
 \bigl([(S_x\setminus\{x\})/A_{\eta_{\rm g}}],\overline I\bigr)[1].
\]
Here $\overline I$ is the descended perverse coefficient and $[1]$ is
the relative-dimension shift of the torsor.
Hard Lefschetz decomposes this cokernel into strings
$\mathbb C[z]/(z^a)$; their lengths are the positive Smith exponents.
The head supplies the exponent-zero part.
Lemma~\ref{lem:normal-purity-base-change} justifies the change from
stabilizer equivariance to this one-variable calculation.

For semisimplicity, retain the full $C_\chi$-equivariant Ext algebra of
$K$ as the acting algebra, even though its module is obtained by
stabilizer restriction.  Quotienting the trivially acting parameter
torus $T_{\sigma,r}$ moves the residual parameter to the origin.
Perverse regrading then gives a finite-dimensional algebra
$\mathcal E_\chi^{\rm res}$ with
\[
 (\mathcal E_\chi^{\rm res})_0
   =\prod_\delta\operatorname{End}(V_\delta),\qquad
 \operatorname{rad}\mathcal E_\chi^{\rm res}
   =(\mathcal E_\chi^{\rm res})_{>0}.
\]
The Smith--Lefschetz calculation identifies the filtration with the
degree filtration on its normal-fibre module.  Positive-degree
elements strictly raise the filtration and annihilate each layer.
Each layer is therefore a module over the displayed semisimple
degree-zero quotient.

\subsection*{Relation with earlier work}

The geometric representation-theoretic foundations used in this paper were
established by Lusztig in the series \emph{Cuspidal local systems and graded
Hecke algebras}~\cite{LusztigCuspidalI,LusztigCuspidalII,LusztigCuspidalIII}.
Lusztig constructs graded Hecke algebras from cuspidal
local systems, realizes standard and irreducible modules in equivariant
intersection cohomology, and proves the induction theorems.  In particular,
\cite[Proposition~10.5 and Corollary~10.7]{LusztigCuspidalII} gives the
geometric formula for the multiplicity of an irreducible module in a standard
module.  In this paper, the argument is formulated using Solleveld's constructible-sheaf
model~\cite{Solleveld}, which is especially
convenient for the integral comparison below.  

The algebraic construction underlying this paper goes back to
Jantzen's original work on highest-weight modules
\cite{JantzenForms,Jantzen}.  Jantzen deforms a highest-weight module over
the local principal ideal domain $\mathbb C[t]_{(t)}$.  Its contravariant
form, or equivalently the associated map from the deformed standard module
to its contragredient dual, defines the $n$-th lattice by divisibility of the
image by $t^n$.  Completion from $\mathbb C[t]_{(t)}$ to $\mathbb C[[t]]$
does not change the induced filtration on the special fibre.  In Smith
normal form the map is diagonal with entries $t^{a_1},\ldots,t^{a_N}$;
the vectors with $a_i=n$ form a basis of the $n$-th associated graded
quotient.  The elementary-divisor formulation used here is the discrete
valuation-ring version of Jantzen's construction. 
Rogawski introduced the analogous
deformation filtration for standard modules of affine Hecke algebras
\cite{Rogawski}.

Beilinson and Bernstein recast the construction geometrically for regular
integral blocks of category $\mathcal O$~\cite{BeilinsonBernstein}.  The
deformed standard-to-costandard morphism corresponds to the canonical map
$i_!\to i_*$, and they identify its Jantzen filtration with the shifted
weight filtration on a nearby-cycle object; semisimplicity then follows from
Gabber's monodromy--weight theorem.  The local Smith--Lefschetz calculation here
plays an analogous role: primitive Lefschetz strings determine the
elementary divisors and hence the layer indices.  A new difficulty in our setting is that
exact specialization at a  graded Hecke parameter is
inhomogeneous.  The remedy is to quotient the parameter torus that acts trivially, and use the resulting residual grading on the specialized
convolution algebra. Its positive-degree ideal is the Jacobson radical.

For affine Hecke algebras of type $A$, Ginzburg
stated the semisimplicity and multiplicity theorem, without proof, in
\cite[Theorem~2.6.2]{GinzburgICM}; Suzuki later transported the category
$\mathcal O$ Jantzen filtration through the Arakawa--Suzuki functor and
proved that, for each fixed deformation direction, the resulting filtration
agrees with the Jantzen filtration on the corresponding degenerate affine
Hecke algebra standard module~\cite[Theorem~4.3.5]{Suzuki}.  Consequently, he obtained the multiplicity formula in that setting.
Related constructions of Jantzen filtrations for graded and affine Hecke
algebras appear in~\cite{Chan} and~\cite{FujitaHernandez}.
Suzuki further asserted independence from every direction satisfying his
nondegeneracy condition~\cite[Proposition~5.3.2]{Suzuki}.  For this he invokes
direction-independence of the Verma-module filtration from~\cite{Barbasch}, but
the cited result concerns the conventional Jantzen filtration and does not
establish independence for arbitrary regular non-dominant directions.
In fact, Gabber and Joseph had explicitly identified that stronger
independence statement as an open question~\cite[Section~4.2]{GabberJoseph}.
Williamson's later $\mathfrak{sl}_4$ example shows that it is false in that
generality~\cite[Sections~1.4 and~8.4]{WilliamsonLocal}.  
Appendix~\ref{app:counterexample} transfers the same phenomenon to a type-$A$
graded Hecke algebra. 

Grojnowski's unpublished note~\cite{GrojnowskiJantzen} contains the local
contracting-cone calculation at the heart of the present proof.  Its
application to affine Hecke modules presupposes the integral identifications
represented by the vertical arrows in~\eqref{eq:master-comparison}.  The main
new result is the extension to the principal Springer realization in an arbitrary 
Lie type.  It
requires parabolic induction, a graded Slodowy
contraction, local coefficients on orbit strata, and exact descent to
component-group summands.
All uses of Grojnowski's calculation are proved in
Section~\ref{sec:local-smith}; the note is
cited for historical attribution, and no result in this paper depends on the
reader having access to it.

\smallskip

As mentioned, the dominance hypothesis cannot simply be omitted.  Williamson showed in
category $\mathcal O$ that a regular non-dominant deformation can change the
Jantzen layers and destroy their
semisimplicity~\cite[Sections~1.6 and~8.4]{WilliamsonLocal}.
Appendix~\ref{app:counterexample} transports the equivariant $4231$
intersection matrix to multisegments and gives an $\Hh(\mathrm{GL}_{16})$-example in which two regular directions produce distinct
filtrations.

\smallskip

Very recently, Droschl \cite{Dro} has announced an independent geometric proof of Jantzen's (equivalently Rogawski's) conjecture for $GL(N)$.

\subsection*{Conventions and organization}

All intersection complexes are normalized to be perverse.  Cohomological
degrees are ordinary degrees unless a regrading is explicitly displayed; the
equivariant parameter $z$ has degree two.  The filtration is normalized by
placing the Langlands quotient in layer zero.  Verdier duality changes an
enhancement $\rho$ to $\rho^\vee$.  The opposite parabolics used in geometric
and Langlands induction are denoted $Q^+$ and $Q^-$, and their common Levi is
denoted $Q$.  The symbols $\eta_{\rm an}$ and $\eta_{\rm g}$ refer both to cocharacters and,
in linear expressions such as $\nu+t\eta_{\rm an}$, for their differentials.
Thus $\mathfrak q=\operatorname{Lie}Q$ denotes the Levi Lie algebra.  After conjugating a semisimple
parameter into
$\mathfrak t=\mathfrak t_{\mathbb R}\oplus i\mathfrak t_{\mathbb R}$,
$\operatorname{Re}$ denotes projection onto
$\mathfrak t_{\mathbb R}$.

The letter $\delta$ in an enhanced-orbit pair $(\gamma,\delta)$ is a label,
whereas $\delta_{\rm an}$ and $\delta_{\rm g}$ always denote analytic and
geometric inducing modules.  The complex $K$ denotes the relevant Springer
direct image after it has been fixed locally, and $K_{\sigma,r}$ denotes its
fixed-parameter specialization.  Residual Ext algebras are denoted by
calligraphic $\mathcal B$ below, to distinguish them from the Borels
$B^\pm$ and from the matrices in the appendices.

The term \emph{Grojnowski grading} refers to the compatible regrading of the
equivariant costalk and stalk modules constructed in Proposition
\ref{prop:ext-equivariance}.  It is obtained from the residual equivariant
grading of Lemma~\ref{lem:residual-grading}, followed by the perverse
regrading associated with a decomposition-theorem splitting; the costalk is
additionally shifted by the degree $\Delta_c$ of the canonical
costalk-to-stalk map, so that this map is homogeneous of degree zero.  The
resulting decreasing degree filtration on the completed costalk, after
pullback to the chosen one-variable arc and reduction modulo $t$, is
precisely its Smith filtration.  In the local cone coordinate used below,
the positive class $z$ pulls back to a nonzero scalar multiple of $t$, so
the two filtrations agree.  This is not the original equivariant
cohomological grading after specialization at the nonzero central character.
Its historical and algebraic provenance is recorded immediately after
Lemma~\ref{lem:residual-grading}.

For the following table put $S=\mathbb C[z]$.  The symbols $u$ and $v$ denote
homogeneous costalk and stalk elements, respectively, and $\deg$ denotes the
ordinary equivariant cohomological degree.  The subscript $G$ denotes the
Grojnowski regrading.  The table suppresses the common additive
translation that is later fixed by placing the Langlands quotient in degree
zero.  The degree conventions used in the local calculation are:
\begin{center}
\begin{tabular}{@{}lll@{}}
\toprule
object & notation & degree convention\\
\midrule
equivariant parameter & $z$ & $\deg z=2$\\
costalk-to-stalk map & $c$ & $\deg c=\Delta_c$\\
cone cokernel & $\mathbb H^\bullet(\mathcal Y,\overline I)[1]$
  & $\mathbb H^k$ occurs in degree $k-1$\\
costalk regrading & $|u|_{G,!}$ & $|u|_{G,!}=\deg(u)+\Delta_c$\\
stalk regrading & $|v|_{G,*}$ & $|v|_{G,*}=\deg(v)$\\
cyclic cokernel summand & $S/(z^a)$ & length and Smith exponent $a\geq1$\\
\bottomrule
\end{tabular}
\end{center}
These conventions make $c$ homogeneous of Grojnowski degree zero.  Thus, in
a Smith block $c(u)=z^av$ with $a\geq1$, one has
\[
 |u|_{G,!}=|z^av|_{G,*}=2a+|v|_{G,*};
\]
the assertion is not that $u$ and $v$ have the same degree.  Hard Lefschetz
places the bottom vector of a length-$a$ string in cokernel degree $-a$.
Consequently $|v|_{G,*}=-a$ and $|u|_{G,!}=a$, which is why the Smith
exponent agrees with the Grojnowski degree of the corresponding lattice
generator even though $\deg z=2$.  Proposition~\ref{prop:ext-equivariance}
proves this identification and the equality of the resulting degree and
Smith filtrations; it also places the exponent-zero quotient in Grojnowski
degree zero.
Corollary~\ref{cor:absolute-layer-normalization} applies the orbit-dimension
translation that places the Langlands quotient in degree zero.

Section~2 records the elementary-divisor formalism.
Section~\ref{sec:local-smith} proves the type-independent local
Smith--Lefschetz and semisimplicity statements.  Section
\ref{sec:fixed-slice-comparison} supplies the fixed-slice and integral
coefficient comparisons and proves the main theorem.  Section
\ref{sec:dominance} isolates the remaining inputs for geometric
unequal-parameter algebras and then explains the role of the dominant
chamber.  Appendix~\ref{app:counterexample} gives the direction-dependence
example.  Appendix~\ref{app:rank-one} works through the complete principal
$A_1$ block, illustrating in one calculation the algebraic intertwiner, the
fixed Springer slice, the Smith--Lefschetz comparison, and the residual
radical.

\section{Elementary divisors and graded filtrations}

\subsection{The valuation filtration}

The following elementary-divisor calculation will be used throughout.

\begin{lemma}[Elementary divisors]\label{lem:ed}
Let $M$ and $N$ be finite free $\A$-modules of equal rank, and let
$u:M\to N$ become an isomorphism over $\K$.  Put
$m=\operatorname{rank}_{\A}M$ and set
\[
 F^n(M/tM)=\bigl(u^{-1}(t^nN)+tM\bigr)/tM.
\]
There are bases in which $u=\operatorname{diag}(t^{a_1},\ldots,t^{a_m})$
with $a_i\geq0$.  If
$\Gr_F^n=F^n/F^{n+1}$, then
\[
 \dim\Gr_F^n(M/tM)=\#\{i:a_i=n\}.
\]
\end{lemma}

\begin{proof}
Choose arbitrary bases of $M$ and $N$, and let $U\in\operatorname{Mat}_m(\A)$
be the matrix of $u$.  Since $u$ becomes an isomorphism over $\K$, the matrix
$U$ has full rank and $\det U\ne0$.  The Smith normal-form theorem over the
principal ideal domain $\A$ gives $P,Q\in\operatorname{GL}_m(\A)$ such that
\[
 PUQ=\operatorname{diag}(d_1,\ldots,d_m),
 \qquad d_i\ne0.
\]
The invertible row and column operations represented by $P$ and $Q$ are
precisely changes of target and source bases.  Every nonzero element of the discrete valuation ring
$\A=\mathbb C[[t]]$ has the form $d_i=\varepsilon_i t^{a_i}$ with
$\varepsilon_i\in\A^\times$ and $a_i\geq0$.  Absorbing the units
$\varepsilon_i$ into the target basis gives bases $(e_i)$ and $(f_i)$ for
which $u(e_i)=t^{a_i}f_i$.  The condition
$t^{a_i}f_i\in t^nN$ holds precisely when
$a_i\geq n$.  Thus $F^n$ is spanned by the reductions of those $e_i$ for
which $a_i\geq n$.  Taking successive quotients proves the formula.
\end{proof}

The filtration records the orders of vanishing of the map.  The geometric
argument identifies them with the lengths of Lefschetz strings.

\subsection{The graded comparison}

For conical local geometry the filtration retains a grading.  Give
$S=\mathbb C[z]$ the grading $\deg z=2$.  For a graded module $V$, the shift
convention is $V(d)^k=V^{k+d}$.

\begin{lemma}[Graded elementary divisors]\label{lem:graded-smith}
Let $M_!$ and $M_*$ be finite free graded $S$-modules of the same rank, and
let $c:M_!\to M_*$ be homogeneous of degree $\Delta_c$ and become an
isomorphism after inverting $z$.  Regrade the source by
\[
 |u|_!=\deg(u)+\Delta_c
 \qquad(u\in M_!\text{ homogeneous}),
\]
and retain the original grading on $M_*$.  Then $c$ has degree zero for these
two gradings.  There are homogeneous bases $(e_i)$ and $(f_i)$ and integers
$a_i\geq0$ for which
\[
 c(e_i)=z^{a_i}f_i,
 \qquad
 \deg(e_i)+\Delta_c=\deg(f_i)+2a_i.
\]
Equivalently,
\[
 \operatorname{coker}(c)\cong
 \bigoplus_i S/(z^{a_i})(s_i),
 \qquad s_i=-\deg(f_i),
\]
where summands with $a_i=0$ are zero and may be omitted.
\end{lemma}

\begin{proof}
\smallskip
\noindent\emph{Injectivity and torsion of the cokernel.}
Because $c$ becomes an isomorphism after $z$ is inverted, every element of
$\ker c$ is killed by a power of $z$.  The free $S$-module $M_!$ has no
$z$-torsion, so $\ker c=0$.  Similarly,
$\operatorname{coker}(c)[z^{-1}]=0$.  Since the cokernel is finitely
generated, one power of $z$ kills it.  Consequently there is an exact sequence of
graded modules
\[
 0\longrightarrow M_!\xrightarrow{c}M_*
 \longrightarrow\operatorname{coker}(c)\longrightarrow0
\]
whose last term is finite-length and $z$-power torsion.

\smallskip
\noindent\emph{The graded Smith form.}
With the source regrading in the statement, $c$ is a degree-zero map.  Choose
homogeneous bases initially.  Every entry of the matrix of $c$ is then
homogeneous; because $S=\mathbb C[z]$ and $\deg z=2$, each nonzero entry is a
scalar multiple of a power of $z$.  Choose an entry having the smallest
$z$-exponent, move it to the first diagonal position, and rescale it to
$z^{a_1}$.  This entry divides every other matrix entry.  Subtracting suitable
homogeneous multiples of the first row and column therefore clears the rest
of that row and column.  The quotients used in these operations are powers of
$z$ of exactly the degree required to keep every new basis vector
homogeneous.  Repeating the argument on the remaining submatrix gives
homogeneous bases $(e_i)$ and $(f_i)$ in which the matrix is diagonal with
entries $z^{a_i}$.  This is the graded Smith algorithm in the present
one-variable setting.

Returning to the original source grading, homogeneity of
$c(e_i)=z^{a_i}f_i$ gives
\[
 \deg(e_i)+\Delta_c=\deg(f_i)+2a_i.
\]

\smallskip
\noindent\emph{The graded cokernel.}
The cokernel of the $i$-th diagonal block is generated by the image of
$f_i$, has the single relation $z^{a_i}f_i=0$, and places that generator in
degree $\deg(f_i)$.  With the shift convention
$V(d)^k=V^{k+d}$, this block is
$S/(z^{a_i})(-\deg(f_i))$.  Taking the direct sum proves the displayed
decomposition.  An exponent-zero block is an isomorphism and
contributes nothing to the cokernel.

\smallskip
\noindent\emph{Completion and the $t$-adic filtration.}
Flat completion extends the homogeneous bases to bases over
$\mathbb C[[z]]$ and leaves the diagonal entries $z^{a_i}$ unchanged.  Under
the base change $z\mapsto c_\eta t$, the $i$-th entry becomes
\[
 (c_\eta t)^{a_i}=c_\eta^{a_i}t^{a_i}.
\]
The scalar $c_\eta^{a_i}$ is a unit and can be absorbed into a target basis
vector.  Thus the elementary divisors over $\A$ are exactly $t^{a_i}$.
Lemma~\ref{lem:ed} then says that the $n$-th associated-graded quotient is
spanned by the reductions of the $e_i$ with $a_i=n$.  For $n>0$, their
degrees, and hence their graded multiplicities, are recorded by the shifts in
the graded cokernel decomposition.  When $n=0$, the diagonal block is already
an isomorphism and disappears from the cokernel; reducing $c$ modulo $z$
identifies this layer instead with $\operatorname{im}(c\bmod z)$.
\end{proof}

Thus, after completion and base change $z\mapsto c_\eta t$ with
$c_\eta>0$, the elementary divisors are $t^{a_i}$ up to units.  The graded
cokernel determines the positive Smith exponents and their graded
multiplicities in the $t$-adic filtration.  It does not see the
exponent-zero blocks, which contribute nothing to the cokernel; these are
recovered from $\operatorname{im}(c\bmod z)$.  Proposition
\ref{prop:ext-equivariance} gives a basis-free, functorial identification of
both kinds of layer.

\begin{remark}\label{rem:completed-smith}
The same conclusion holds if finite freeness is known only after localization
at $(z)$ or after $z$-adic completion.  The localized or completed source and
target are then free over the discrete valuation ring $S_{(z)}$ or
$\mathbb C[[z]]$, so ordinary Smith theory applies.  Flatness identifies the
resulting torsion cokernel with the localization or completion of the graded
cokernel.  The Smith exponents $a_i$ are therefore unchanged.
\end{remark}

\section{Local Smith--Lefschetz theory and semisimplicity}
\label{sec:local-smith}

\subsection{The equivariant cone calculation}

The local filtration problem reduces to a positively contracted affine
cone.  The calculation is stated with the nonconstant coefficient
systems needed in generalized Springer blocks.

Let $X$ be a pure $d$-dimensional affine variety with a positive
$\mathbb C^\times$-action and unique fixed point $x$.  Put
$U=X\setminus\{x\}$, let $i:\{x\}\hookrightarrow X$ be the inclusion, and
let $\mathcal Y=[U/\mathbb C^\times]$ denote the quotient stack.  Choose the
equivariant generator
$z\in H^2_{\mathbb C^\times}(\mathrm{pt})$ so that its associated orbifold
line bundle on $\mathcal Y$ is ample.  Let
$\mathcal E$ be a pure polarizable equivariant local system of geometric
origin on a smooth invariant
stratum whose closure is $X$, and normalize
$I=\IC(X,\mathcal E)$ to be perverse.  Write $\overline I$ for the descended
pure IC object on $\mathcal Y$.  Write
\[
 c:\mathsf L_!=H^\bullet_{\mathbb C^\times}(i^!I)
 \longrightarrow
 \mathsf L_*=H^\bullet_{\mathbb C^\times}(i^*I)
\]
for the canonical map.  Assume that its source and target
are finite free over $\mathbb C[z]$.  Alternatively, assume that their
$z$-adic completions are finite free over $\mathbb C[[z]]$; in that case all
module statements below are read after completion.

\begin{lemma}[Equivariant cone calculation with local coefficients]\label{lem:cone-calculation}
Under these assumptions, $c$ is injective and there is a canonical graded
isomorphism
\[
 \operatorname{coker}(c:\mathsf L_!\to\mathsf L_*)
 \cong \mathbb H^\bullet(\mathcal Y,\overline I)[1]
\]
as graded $\mathbb C[z]$-modules, where the hypercohomology is placed in
ordinary cohomological degrees (and is ordinary intersection cohomology when
$\mathcal E$ is constant), and $z$ acts
by the Chern class of the ample orbifold line bundle on $\mathcal Y$.
If a finite group $\Gamma$ acts equivariantly on all the data, the
identification is $\Gamma$-equivariant.  For every $k\geq0$, hard Lefschetz
gives
\[
 z^k:\mathbb H^{-k}(\mathcal Y,\overline I)
 \xrightarrow{\sim}
 \mathbb H^k(\mathcal Y,\overline I)(k).
\]
\end{lemma}

\begin{proof}
\smallskip
\noindent\emph{The localization sequence.}
For $j:U\hookrightarrow X$, the localization distinguished triangle is
\[
 i_*i^!I\longrightarrow I
 \longrightarrow Rj_*j^*I
 \xrightarrow{\partial}i_*i^!I[1].
\]
Thus the last arrow returns to the first object, shifted cohomologically by
$[1]$.  Applying $i^*$ and using $i^*i_*\simeq\mathrm{id}$ gives the
distinguished triangle on the fixed point
\[
 i^!I\xrightarrow{c}i^*I\longrightarrow i^*Rj_*j^*I
 \xrightarrow{i^*\partial}i^!I[1].
\]
The equivariant hypercohomology of the third term is identified as follows.  Since
$Rj_*j^*I$ is equivariant and the $\mathbb C^\times$-action contracts $X$ to
$x$, the contraction principle gives
\[
 R\Gamma_{\mathbb C^\times}
   \bigl(\{x\},i^*Rj_*j^*I\bigr)
 \simeq
 R\Gamma_{\mathbb C^\times}
   \bigl(X,Rj_*j^*I\bigr).
\]
By the defining property of the derived direct image, the latter complex is
\[
 R\Gamma_{\mathbb C^\times}(U,j^*I).
\]
Thus the long exact sequence of the displayed triangle contains, in every
degree, the costalk-to-stalk map $c$ followed by the equivariant cohomology
of the punctured cone.

After $z$ is inverted, equivariant localization
\cite[Theorem~6.2(2)]{GKM} leaves only the fixed
point, and the first map becomes an
isomorphism.  Its kernel is therefore $z$-torsion; as a submodule of the free
module $\mathsf L_!$, it vanishes.  The connecting homomorphism following the
punctured-cone term has image in the kernel of $c$ in the next degree, and is
therefore zero.  Hence the long exact sequence breaks, degree by degree, into
short exact sequences
\[
 0\longrightarrow H^k_{\mathbb C^\times}(i^!I)
 \xrightarrow{c}H^k_{\mathbb C^\times}(i^*I)
 \longrightarrow H^k_{\mathbb C^\times}(i^*Rj_*j^*I)
 \longrightarrow0.
\]
In the completed alternative, apply the same argument after $z$-adic
completion.  The modules are finite over the noetherian coefficient ring, so
completion is exact; localization also shows that the cokernel is killed by
a power of $z$.

\smallskip
\noindent\emph{Passage to the projective quotient.}
The action on $U$ has finite stabilizers, so its Borel construction computes
the rational cohomology of the quotient stack
$\mathcal Y=[U/\mathbb C^\times]$.  Let
$q:\mathcal Y\to Y_c$ be the coarse-moduli morphism; positivity of the
grading makes $Y_c$ projective.  A stack coefficient cannot be replaced by an
ordinary local system on $Y_c$ without losing a nontrivial stabilizer
representation.  The relevant coefficient on the coarse quotient is instead
the derived pushforward $Rq_*\overline I$.  In characteristic zero this pushforward is
exact for finite inertia and hence agrees with $q_*\overline I$ here.

At a geometric point
$\bar y\to Y_c$, proper base change identifies the higher stalks of
$Rq_*\overline I$ with the group cohomology of the finite stabilizer
$\Gamma_{\bar y}$ acting on the corresponding stalk of $\overline I$.
Because $\mathbb C[\Gamma_{\bar y}]$ is semisimple, invariants are exact and
this group cohomology vanishes in positive degrees.  Equivalently, on a
finite quotient chart $[V_0/\Gamma_0]$, coarse pushforward is the Reynolds
summand $(r_*-)^{\Gamma_0}$ of the finite pushforward along
$r:V_0\to V_0/\Gamma_0$.  The averaging idempotent is self-adjoint, so the
summand commutes with Verdier duality and preserves purity, semisimplicity,
and polarizability.  Here self-adjointness is taken with respect to the
Verdier pairing in the category of polarizable pure Hodge modules, after
averaging the polarization over the finite group.  Consequently
\[
 R\Gamma(\mathcal Y,\overline I)
 \simeq R\Gamma(Y_c,q_*\overline I).
\]
This description retains arbitrary finite-stabilizer coefficient
representations with their correct invariant multiplicities.

By the definition of the quotient stack, equivariant derived objects on
$U$ are the same as derived objects on $\mathcal Y$.  Let
$\pi:U\to\mathcal Y$ be the quotient atlas.  It is smooth of relative
complex dimension one, so $\pi^*[1]$ is perverse $t$-exact.  Since
$\overline I$ is normalized to be perverse on $\mathcal Y$, its pullback is
therefore
\[
 j^*I=\pi^*\overline I[1]
\]
in the underlying constructible category.  Equivariant derived global
sections consequently give
\[
 R\Gamma_{\mathbb C^\times}(U,j^*I)
 \simeq R\Gamma(\mathcal Y,\overline I[1])
 \simeq R\Gamma(\mathcal Y,\overline I)[1].
\]
Combining this with the contraction and direct-image identifications above
proves
\[
 H^\bullet_{\mathbb C^\times}(i^*Rj_*j^*I)
 \cong \mathbb H^\bullet(\mathcal Y,\overline I)[1],
\]
which is the asserted description of the cokernel of $c$.  The shift is
independent of $d$: it is the relative dimension of the punctured-cone
torsor.  For
example, when $X=\mathbb A^d$ and $I=\mathbb C_X[d]$, the bottom class of
$H^\bullet(\mathbb P^{d-1},\mathbb C[d-1])[1]$ lies in degree $-d$, as does
the generator of the equivariant stalk.
For constant coefficients, if $IH^\bullet(Y_c)$ denotes ordinary
topologically normalized intersection cohomology, the same convention reads
\[
 \mathbb H^\bullet(\mathcal Y,\overline I)[1]
 \cong IH^\bullet(Y_c)[d].
\]

\smallskip
\noindent\emph{Hard Lefschetz and equivariance.}
With the sign convention fixed above, $z$ is the first Chern class of the positive
orbifold line bundle associated with $U\to\mathcal Y$.  Let $N$ be a common
multiple of the orders of the finite stabilizers.  Its $N$-th tensor power
descends to $\mathcal O_{Y_c}(N)$ on the coarse quotient
$Y_c=\operatorname{Proj}\mathbb C[X]$, and this descended line bundle is
ample.  Thus $z$ is a positive rational multiple of its first Chern class.
The object $\overline I$ is pure:
in the generalized Springer setting the coefficient system has finite
monodromy and is of geometric origin, and purity is preserved by IC
extension and finite-stack descent.  Apply projective hard Lefschetz on
$Y_c$ to the pure polarizable object $q_*\overline I$.  In the Hodge-module
realization this is Saito's hard Lefschetz theorem
\cite[Th\'eor\`eme~5.3.1]{SaitoPolarizable}; equivalently, after spreading
out, it follows from the pure $\ell$-adic theorem
\cite[Th\'eor\`eme~5.4.10]{BBD}.  It gives
\[
 z^k:\mathbb H^{-k}(\mathcal Y,\overline I)
 \xrightarrow{\sim}
 \mathbb H^k(\mathcal Y,\overline I)(k).
\]
The argument allows $\mathcal E$ to be nonconstant.  Naturality of the
localization triangle and of descent shows that an additional finite
component action $\Gamma$ commutes with
both the cokernel identification and the Lefschetz operators.
\end{proof}

We next construct the graded module on which the local filtration is
formed.  Restriction along an orbit is not restriction to a point, and
equivariance on a point is equivariance for its stabilizer.

\subsection{Normal restriction and change of equivariance}

Let $C$ act on a smooth variety $X$, let $\mathcal O=Cx$ have complex
dimension $d$, and write $j:\mathcal O\hookrightarrow X$ and
$l:\{x\}\hookrightarrow\mathcal O$.  Work in the invariant open complement
of $\overline{\mathcal O}\setminus\mathcal O$, so that $j$ is closed.
For a $C$-equivariant complex $P$, define its \emph{normal fibres}
\[
 \mathsf N_x^!P=l^*j^!P[-d](-d),\qquad
 \mathsf N_x^*P=l^*j^*P[-d](-d).
\]
The adjunction along the orbit gives a natural morphism
\begin{equation}\label{eq:normal-fibre-map}
 c_x^N(P):\mathsf N_x^!P\longrightarrow\mathsf N_x^*P.
\end{equation}
These are objects equivariant for $C_x$, not for $C$.  Smooth purity
on the orbit gives
\begin{equation}\label{eq:normal-point-objects}
 \mathsf N_x^!P\simeq i_x^!P[d],\qquad
 \mathsf N_x^*P=i_x^*P[-d](-d).
\end{equation}
These formulas identify objects; they do not identify
\eqref{eq:normal-fibre-map} with the ordinary point adjunction.

For a transverse slice $s:V\hookrightarrow X$, with vertex
$k:\{x\}\hookrightarrow V$, put $I=s^*P[-d](-d)$.
A transverse product neighbourhood identifies the orbit adjunction with
the identity in the orbit direction tensored with the vertex adjunction:
\begin{equation}\label{eq:normal-slice-arrow}
 \bigl(\mathsf N_x^!P\xrightarrow{c_x^N(P)}\mathsf N_x^*P\bigr)
 \simeq\bigl(k^!I\longrightarrow k^*I\bigr).
\end{equation}
The comparison can be made algebraically.  After shrinking $V$ about
$x$, transversality makes the action map
$a:C\times V\to X$ smooth and gives $a^{-1}(\mathcal O)=C\times\{x\}$.
Smooth base change for the orbit adjunction, together with
$a^*P\simeq\operatorname{pr}_V^*s^*P$ from equivariance, identifies its
pullback with the vertex adjunction on $V$.  Restricting to
$\{1\}\times\{x\}$ gives \eqref{eq:normal-slice-arrow}.
If a torus $T\subset C_x$ preserves $V$, the construction is
$T$-equivariant for $t\cdot(c,v)=(tct^{-1},tv)$; the neighbourhood can
be chosen $T$-stable.  No invariance of $V$ under the full stabilizer
$C_x$ is required.  The comparison is natural in $P$ and commutes
with every morphism $P\to P'[m]$.  This is the stratumwise
intersection-form construction of
\cite[Lemma~3.9 and Remark~3.12]{JMWParity}.
The graded Slodowy construction below supplies the transversality used here.

The point adjunction includes the self-intersection factor of the
orbit-tangent space.  That Euler class may vanish.  No division by it is
made in \eqref{eq:normal-slice-arrow}.

\begin{lemma}[Purity at the vertex and equivariant base change]
\label{lem:normal-purity-base-change}
Let $I$ be a pure perverse IC complex of weight $w$ on a positively
contracted affine cone with vertex $x$, equivariant for the contracting
action and with polarizable coefficients.
Then $H^q(i_x^*I)$ is pure of weight $w+q$, and the costalk has the
corresponding purity by duality.
If a connected linear algebraic group $H$ supplies an equivariant
structure on either vertex object in the mixed category, its
equivariant hypercohomology is finite free
over $H_H^\bullet(\mathrm{pt})$.  Change of equivariance to a subtorus
is computed by ordinary tensor product.  These assertions also hold for
finite direct sums and equivariant direct summands.
\end{lemma}

\begin{proof}
Suppose first that the strict support has positive dimension.  The
descent $\overline I$ to the projective quotient stack has weight $w-1$,
with $I|_U=\pi^*\overline I[1]$.  The descent and projective hard
Lefschetz argument of Lemma~\ref{lem:cone-calculation} applies to this
quotient independently of freeness at the vertex.  Put $L=c_1(\pi)$;
its sign does not matter here.  The torsor Gysin sequence contains
\[
 H^{q-1}(\mathcal Y,\overline I)(-1)
 \xrightarrow{L}H^{q+1}(\mathcal Y,\overline I)
 \longrightarrow H^q(U,I)
 \longrightarrow H^q(\mathcal Y,\overline I)(-1)
 \xrightarrow{L}H^{q+2}(\mathcal Y,\overline I).
\]
For $q<0$ the last arrow is injective by hard Lefschetz.  Thus $H^q(U,I)$
is the cokernel of the first arrow, whose terms are pure of weight
$w+q$.  IC extension at the vertex is the truncation in degrees $q<0$
of the punctured-neighbourhood cohomology.  Contraction identifies this
with the displayed calculation on $U$.  Strict support gives zero
stalk cohomology in degrees $q\geq0$.  A point-supported summand is
immediate.  Duality gives costalk purity.

We use algebraic approximations to the Borel construction in the given
mixed equivariant realization, so the spectral sequence below respects
weights.  For connected $H$ the action on ordinary vertex cohomology is trivial,
and the Borel spectral sequence has
\[
 E_2^{p,q}=H_H^p(\mathrm{pt})\otimes H^q(i_x^{!,*}I).
\]
The coefficient ring is polynomial in even degrees and pure of weight
$p$ in degree $p$; a unipotent radical does not change it.
Terms of total degree $p+q$ have weight $w+p+q$.  Every differential
raises that degree by one and preserves weights, so vanishes.
Homogeneous lifts of a basis of ordinary cohomology give a free
equivariant basis: the resulting map from a free module is an
isomorphism on the associated graded of this spectral sequence.

The change-of-groups spectral sequence computes derived tensor product
with the coefficient ring of the subtorus.  Freeness kills higher Tor,
so its edge map is the ordinary tensor-product isomorphism, naturally
in the vertex object.  Direct sums and direct summands preserve these
properties.
\end{proof}

\subsection{The residual algebra and its normal-fibre module}

Fix $\chi=(\sigma,r)$ with $r\ne0$ and set
\[
 C_\chi=Z_G(\sigma)\times\mathbb C^\times,\qquad
 X_\chi=\mathfrak g_{\sigma,2r}.
\]
The action is $(g,a)v=a^{-2}\operatorname{Ad}(g)v$.
Let $K=K_{\sigma,r}$ be the localized Springer direct image of
\cite[equations~(2.11)--(2.14)]{Solleveld}.  Let $A=T_{\sigma,r}$ be the
smallest algebraic torus whose Lie algebra contains $\chi$ and put
$D=C_\chi/A$.  Write $M_{\rm cusp}$ for the cuspidal Levi of the block,
distinct from the Langlands Levi.

Fix $x\in X_\chi$ and put
\[
 H=C_{\chi,x}^{\circ},\qquad D_x^\circ=H/A,\qquad
 \Gamma=\pi_0(C_{\chi,x}).
\]
The full stabilizer acts on the normal fibres; before specialization
its component action can be semilinear.
The global algebra and point modules use different coefficient rings:
\[
 R_\chi=H_{C_\chi}^\bullet(\mathrm{pt}),\qquad
 R_x=H_H^\bullet(\mathrm{pt}),\qquad R_\chi\longrightarrow R_x.
\]
For a cocharacter $\eta$ of $H$, completion and evaluation on
$\chi+t(\eta,0)$ define compatible maps
\[
 \phi_\eta:\widehat R_\chi\longrightarrow\A,\qquad
 \psi_{\eta,x}:\widehat R_x\longrightarrow\A.
\]
When applying component projectors we require that $\psi_{\eta,x}$
be $\Gamma$-invariant.  This is a condition on the homomorphism of
equivariant coefficient rings, not on a chosen Lie-algebra representative
of the arc.  For the Langlands arcs, Proposition
\ref{prop:two-component-groups} verifies it by restricting to a reductive
Levi factor of the point stabilizer.  Restriction removes the connected
unipotent radical without changing equivariant cohomology or component
groups; the central direction is fixed in that reductive model.

Choose a decomposition-theorem splitting.  After descent to $D$, set
\[
 \overline K_D=\bigoplus_a{}^pH^aK_D,\quad
 \mathcal B_D=\bigoplus_m\Ext_D^m(\overline K_D,\overline K_D),\quad
 \mathcal E_\chi^{\rm res}
 =\mathbb C_0\otimes_{H_D^\bullet(\mathrm{pt})}\mathcal B_D.
\]
Also set $\overline K=\bigoplus_a{}^pH^aK$,
$\mathcal B=\Ext^\bullet_{C_\chi}(\overline K,\overline K)$ and
\[
 \mathcal B_\eta=\A\otimes_{\widehat R_\chi,\phi_\eta}
       (\widehat R_\chi\otimes_{R_\chi}\mathcal B).
\]
The normal lattices and their map are
\begin{equation}\label{eq:normal-arc-lattices}
 \mathsf L_{!,*}^{\eta}
 =\A\otimes_{\widehat R_x,\psi_{\eta,x}}
   H_H^\bullet(\mathsf N_x^{!,*}K)^\wedge_\chi,\qquad
 c^\eta:\mathsf L_!^\eta\longrightarrow\mathsf L_*^\eta.
\end{equation}
Degree statements use the chosen perverse regrading.  There is no
assertion of $C_\chi$-equivariance on the point.

\begin{lemma}[Residual grading at a nonzero central character]
\label{lem:residual-grading}
The torus $A$ is central in $C_\chi$ and fixes the localized incidence
diagram and its coefficient, so $K$ descends to a $D$-equivariant complex.
There is an ungraded algebra isomorphism
\[
 \mathcal H_\chi:=
 \mathbb H(G,M_{\rm cusp},q\mathcal E)/\mathfrak m_{(\sigma,r)}
 \xrightarrow{\sim}\mathcal E_\chi^{\rm res}.
\]
If $\overline K_D=\bigoplus_\delta V_\delta\otimes\IC_\delta$, then
\[
 (\mathcal E_\chi^{\rm res})_0
   =\prod_\delta\operatorname{End}(V_\delta),\qquad
 \operatorname{rad}\mathcal E_\chi^{\rm res}
   =(\mathcal E_\chi^{\rm res})_{>0}.
\]
The algebra $\mathcal B_\eta$ acts on the normal lattices and their map.
Its special fibre is $\mathcal E_\chi^{\rm res}$.
If the normal vertex objects satisfy
Lemma~\ref{lem:normal-purity-base-change}, the lattice special fibres
carry compatible residual gradings: an algebra element of degree $m$
raises module degree by $m$.
\end{lemma}

\begin{proof}
\emph{The global algebra.}
The Zariski closure of $\exp(\mathbb C\chi)$ is central in $C_\chi$.
The fixed-incidence description in \cite[Section~2.1]{Solleveld} shows
that it fixes the localized diagram.  It acts trivially on the
constructible coefficient by connectedness, as in
\cite[Lemma~2.7, equations~(2.23)--(2.24)]{Solleveld}.
A complementary reductive factor up to finite central isogeny gives
\[
 \mathcal B\simeq H_A^\bullet(\mathrm{pt})\otimes\mathcal B_D.
\]
Finite kernels and component invariants cause no higher cohomology
with complex coefficients.

The splitting regrades $\Ext^n_{C_\chi}(K,K)$: the component between
perverse cohomology indices $a$ and $b$ has degree $n+a-b$.
The underlying algebra is unchanged.
The completed comparison in
\cite[Proposition~2.3 and Theorem~2.5]{Solleveld}, followed by quotient
by the maximal ideal, identifies the exact central fibre with
$\mathbb C_\chi\otimes_{R_\chi}\mathcal B$.  In the displayed
factorization $\chi=(\chi_A,0)$; evaluating the first factor and
augmenting the second gives $\mathcal E_\chi^{\rm res}$.  It also gives
$\mathbb C\otimes_\A\mathcal B_\eta=\mathcal E_\chi^{\rm res}$.

Negative Ext between perverse sheaves vanishes, so the residual grading
is nonnegative.  The algebra is finite dimensional because the Hecke
algebra is finite over its centre.  Schur's lemma gives its degree-zero
part.  The positive-degree ideal is nilpotent with semisimple quotient,
and hence is exactly the Jacobson radical.

\smallskip
\emph{The point module.}
Restriction along the orbit followed by its fibre functor takes a
morphism $K\to K[m]$ to a morphism of normal fibres.  Naturality of
\eqref{eq:normal-fibre-map} makes the normal map commute with this
action.  The target equivariance is $C_{\chi,x}$; on connected
equivariant cohomology the action is linear over $R_\chi\to R_x$, as in
\cite[equations~(3.15)--(3.16)]{Solleveld}.  Thus
\eqref{eq:normal-arc-lattices} is $\mathcal B_\eta$-linear.

Since $A\subset H$ acts trivially on the normal objects, in the
perverse regrading one has
\[
 H_H^\bullet(\mathsf N_x^{!,*}K)
 \simeq H_A^\bullet(\mathrm{pt})\otimes M_{!,*},\qquad
 M_{!,*}=\bigoplus_a
 H_{D_x^\circ}^\bullet(\mathsf N_x^{!,*}{}^pH^aK_D).
\]
The fibre functor here includes restriction from $D$ to $D_x^\circ$.
The residual arc in the latter group is $t\bar\eta$, and maps to the
same arc in $D$ used for the algebra.  Consequently
\[
 \mathsf L_{!,*}^\eta/t\mathsf L_{!,*}^\eta
 \simeq\mathbb C_0\otimes_{H_{D_x^\circ}^\bullet(\mathrm{pt})}M_{!,*}.
\]
It is the \emph{stabilizer} coefficient ring that is augmented.
The global augmentation ideal restricts into this ideal, so the
action factors through $\mathcal E_\chi^{\rm res}$.
Lemma~\ref{lem:normal-purity-base-change} supplies freeness and base
change.  The arc is homogeneous for $\deg t=2$, and the perverse Ext
degree gives the claimed degree of the action.

If $c^\eta(v)=t^p w$ and $b\in\mathcal B_\eta$, then
$c^\eta(bv)=t^pbw$.  Each inverse image defining the Smith filtration is
therefore stable, and its reduction and quotients carry the specialized
algebra action.  If $\psi_{\eta,x}$ is component-invariant, the component
action becomes $\A$-linear and commutes with this construction.
\end{proof}

\begin{unnumberedremark}[Provenance of the residual-grading argument]
The formal principle used here is classical: if a finite-dimensional algebra
$E=\bigoplus_{m\geq0}E_m$ is nonnegatively graded and $E_0$ is semisimple,
then $\operatorname{rad}E=E_{>0}$; consequently, a filtration raised by
$E_{>0}$ has semisimple associated graded pieces.  This is the familiar
grading--radical mechanism underlying Koszul and mixed-category arguments;
compare \cite[Section~2.1]{BGS}.  In the present setting, Solleveld's
completed geometric realization supplies the Ext algebra, while the local
Smith--Lefschetz calculation is the one appearing in Grojnowski's model
\cite[Section~2.3]{GrojnowskiJantzen}.  The additional step specific to this
argument is the residual quotient by $T_{\sigma,r}$.  It is needed because
evaluation at the nonzero parameter $(\sigma,r)$ is inhomogeneous and does
not preserve the original equivariant grading; quotienting by
$T_{\sigma,r}$ moves the remaining parameter to the origin and produces the
grading on $\mathcal E_\chi^{\rm res}$ used below.
\end{unnumberedremark}

Retain the normal lattices \eqref{eq:normal-arc-lattices} and put
$\mathcal E=\mathcal E_\chi^{\rm res}$.  Assume that $\eta$ contracts a
transverse slice at $x$ and that its normalized IC restrictions are pure
and polarizable.  Equation~\eqref{eq:normal-slice-arrow} and
Lemma~\ref{lem:normal-purity-base-change} identify the normal map, after
residual base change, with the vertex map for the contracting action.
The next hypothesis records what the filtration argument uses.

\begin{hypothesis}[Local cone compatibility]
\label{hyp:local-cone-compatibility}
Summand by summand, restriction to the chosen normal slice and pullback
along $\chi+t(\eta,0)$ identify $c^\eta$ with the completed cone map of
Lemma~\ref{lem:cone-calculation}.
\end{hypothesis}

Under the purity assumptions above, Lemma~\ref{lem:normal-purity-base-change}
already makes $\mathsf L_!^\eta$ and $\mathsf L_*^\eta$ finite free over
$\A$.  Their ranks agree because the cone map is generically invertible.
These are consequences, not additional hypotheses.

Under this hypothesis, put
\[
 \overline{\mathsf L}_!^\eta=
 \mathsf L_!^\eta/t\mathsf L_!^\eta
\]
and, for $p\geq0$, define
\[
 F^p\overline{\mathsf L}_!^\eta=
 \frac{(c^\eta)^{-1}(t^p\mathsf L_*^\eta)+
       t\mathsf L_!^\eta}{t\mathsf L_!^\eta}.
\]

Choose a decomposition-theorem splitting
$K\simeq\bigoplus_k{}^pH^k(K)[-k]$.  The resulting perverse regrading gives
compatible gradings on $\mathcal E$ and on the special fibres of the
costalk and stalk.  The transported module grading is called the residual
Grojnowski grading.  Write $H_G^j$ for degree $j$ in this grading.

\begin{proposition}[Ext/convolution equivariance and semisimplicity]
\label{prop:ext-equivariance}
Assume Hypothesis~\ref{hyp:local-cone-compatibility}.
The algebra $\mathcal E$ acts on
$\overline{\mathsf L}_!^\eta$, and every term of its Smith filtration is
$\mathcal E$-stable.  Moreover,
\[
 F^p\overline{\mathsf L}_!^\eta=
 \bigoplus_{j\geq p}H_G^j(\overline{\mathsf L}_!^\eta).
\]
Consequently every associated-graded layer is annihilated by the radical
$\mathcal E_{>0}$ and is semisimple for the specialized graded affine Hecke
algebra, with multiplicities given by IC costalk coefficients.
\end{proposition}

\begin{proof}
Retain the notation $\mathsf L_!^\eta,\mathsf L_*^\eta$, and $c^\eta$ from
the statement.  To reduce notation in the proof, write
\[
 \mathsf L_!=\mathsf L_!^\eta,\qquad
 \mathsf L_*=\mathsf L_*^\eta,\qquad
 c=c^\eta,\qquad z=t.
\]
If the slice is a point, the normal map is the identity and the assertion
is immediate, with all exponents zero.  Assume henceforth that the slice
has positive dimension.  Here the last equality is a choice of coordinate, not an identification of
the multivariable coefficient ring with $\mathbb C[[t]]$.  More explicitly,
let $A_\eta\cong\mathbb C^\times$ be the effective one-dimensional torus
through which the contracting action factors, and let
$z\in H_{A_\eta}^2(\mathrm{pt})$ be the first Chern class of its positive
character.  Pullback along the arc sends $z$ to a nonzero scalar multiple of
$t$.  Rescaling the formal coordinate by that scalar does not change any
Smith exponent or filtration.  Normalize the coordinate so that $z\mapsto t$,
and assign $t$ residual degree two.

\smallskip
\noindent\emph{Convolution equivariance.}
Use the equivariant Ext algebra of $K$ already defined in
Lemma~\ref{lem:residual-grading}.  Its multiplication is composition
of derived morphisms.  Applying either normal-fibre functor therefore
gives an action, and naturality of the orbit adjunction makes the
normal map equivariant for this action.  This formulation retains the
cuspidal coefficient system and its duality; it does not replace a
coefficient-dependent convolution kernel by ordinary Borel--Moore
homology of the underlying correspondence.  After completion at $\chi$,
\cite[Theorem~2.5]{Solleveld} identifies this action with the action of the
completed graded affine Hecke block.  A chosen
splitting of $K$ identifies the component from ${}^pH^a(K)$ to
${}^pH^b(K)$ inside $\operatorname{Ext}^m(K,K)$ with
\[
 \operatorname{Ext}^{m+a-b}({}^pH^a(K),{}^pH^b(K)).
\]

Keep the full $C_\chi$-equivariance until the exact central fibre is taken.
Lemma~\ref{lem:residual-grading} then identifies that fibre with
$\mathcal E$ and supplies its nonnegative residual grading.  Reindexing by
the perverse degrees is performed on this residual algebra; the resulting
grading is not the original Borel--Moore grading specialized at $\chi$.

The notation $H_G^j$ denotes the degree-$j$ part of the special
fibre of the regraded normal module.  Its connected-stabilizer model is
the sum of the equivariant normal fibres of the perverse summands,
with the stabilizer coefficient ring augmented as in
Lemma~\ref{lem:residual-grading}.  The stalk is treated in the same way.  Naturality of the orbit adjunction gives, for
$\varphi:K\to K[m]$, the square
\[
\begin{array}{ccc}
 \mathsf N_x^!K&\xrightarrow{c_x^N}&\mathsf N_x^*K\\
 \downarrow{\scriptstyle\mathsf N_x^!\varphi}&&
 \downarrow{\scriptstyle\mathsf N_x^*\varphi}\\
 \mathsf N_x^!K[m]&\xrightarrow{c_x^N}&\mathsf N_x^*K[m].
\end{array}
\]
This is a square in the stabilizer-equivariant category.  Thus the
normal arrow is linear before regrading.  Completion and pullback along
$\phi_\eta$ make it a $\mathcal B_\eta$-linear arrow over $\A$.  Since the
one-variable parameter $z=t$ is central, $\mathcal B_\eta$ preserves every inverse
image $c^{-1}(z^p\mathsf L_*)$.  After reduction modulo $z$, the acting
algebra becomes
$\mathbb C\otimes_\A \mathcal B_\eta\cong\mathcal E$; consequently the reductions
of those inverse images and all Smith quotients are graded
$\mathcal E$-modules.  The grading is homogeneous because the family
projects to the arc $t\bar\eta$ in the residual group, as in
Lemma~\ref{lem:residual-grading}.

\smallskip
\noindent\emph{Smith layers and primitive strings.}
Apply Lemma~\ref{lem:cone-calculation} to each IC summand on the positive
normal slice through $x$; write $I$ for that summand,
$\mathcal Y$ for its projectivized punctured slice, and $\overline I$ for its
descent to $\mathcal Y$.  Suppress the summand index.  The
degree-to-exponent correspondence is obtained by constructing the primitive
quotient, identifying it with the Smith layer, and comparing their degrees.
Restriction to the effective contracting torus identifies the summand of
the completed map with the completed map in
Lemma~\ref{lem:cone-calculation}.  Under pullback to the formal arc its cone
coordinate is a nonzero scalar multiple of $t$; with the normalization at
the start of the proof it is $z=t$.  Consequently the $t$-adic Smith
filtration in the statement is exactly the $z$-adic filtration in the cone
calculation.  The splitting decomposes $c$ and its cokernel accordingly, so
it is enough to work on this fixed summand.
Put $\mathcal C=\operatorname{coker}(c)$.
Retain ordinary equivariant cohomological degrees on $\mathsf L_!$,
$\mathsf L_*$, and
$\mathcal C$, so $\deg z=2$, and write $\Delta_c$ for the homogeneous degree
of $c$.
Thus, whenever $c(u)=z^av$ with $u,v$ homogeneous,
\begin{equation}\label{eq:degree-balance}
 \deg(u)+\Delta_c=2a+\deg(v).
\end{equation}
The shift in the cone calculation is part of this convention:
$\mathcal C\cong\mathbb H^\bullet(\mathcal Y,\overline I)[1]$, so a class in
$\mathbb H^k(\mathcal Y,\overline I)$ has degree $k-1$ in $\mathcal C$.
Up to the translation fixed by placing the smooth point-supported summand
at $x$ in degree zero, the compatible Grojnowski degrees on source and
target are
\begin{equation}\label{eq:grojnowski-source-degree}
 |u|_{G,!}=\deg(u)+\Delta_c,
 \qquad |v|_{G,*}=\deg(v).
\end{equation}
Thus $c$ has degree zero for the transported grading; the same
translation is used on every perverse summand.

\smallskip
\noindent\emph{The primitive quotient.}
For $a\geq1$ define the functorial quotient
\begin{equation}\label{eq:primitive-string-quotient}
 \operatorname{Prim}_a(\mathcal C)=
 \frac{\ker(z^a:\mathcal C\to\mathcal C)}
 {\ker(z^{a-1}:\mathcal C\to\mathcal C)+
  z\ker(z^{a+1}:\mathcal C\to\mathcal C)}.
\end{equation}
Before completion put $S=\mathbb C[z]$; after completion one may replace it
by $\mathbb C[[z]]$.  The finite-length quotients below are unchanged, and
$(s)$ denotes a grading shift.  On a cyclic summand $S/(z^b)(s)$ this quotient
vanishes unless $b=a$; for
$b=a$, it is the one-dimensional space generated by the bottom vector of the
string.  Thus $\operatorname{Prim}_a(\mathcal C)$ records the Smith summands
of exponent $a$, together with their grading and coefficient action.

\smallskip
\noindent\emph{The Smith-layer map.}
There is a direct map from the $a$-th Smith layer to
\eqref{eq:primitive-string-quotient}.  If
$\bar u\in F^a(\mathsf L_!/z\mathsf L_!)$, choose a lift
$u\in\mathsf L_!$ and write
\[
 c(u)=z^a v,\qquad v\in\mathsf L_*.
\]
Send the class of $\bar u$ in $F^a/F^{a+1}$ to the class of
$\bar v=v+c(\mathsf L_!)$ in $\operatorname{Prim}_a(\mathcal C)$.  To verify
well-definedness, first replace the lift by $u+zh$.  If
$c(u+zh)=z^av'$, then, since $\mathsf L_*$ is $z$-torsion-free,
\[
 c(h)=z^{a-1}(v'-v).
\]
Thus $[v'-v]\in\ker(z^{a-1}:\mathcal C\to\mathcal C)$, the first term in
the denominator of~\eqref{eq:primitive-string-quotient}.  If instead one
changes the class by an element represented by $u'$ with
$c(u')=z^{a+1}v''$, its image changes by $z[v'']$.  Since
$[v'']\in\ker(z^{a+1}:\mathcal C\to\mathcal C)$, this is the second term in
that denominator.  Hence the map is well defined on $F^a/F^{a+1}$.  On a
Smith block
\[
 S e\xrightarrow{\ z^b\ }S f
\]
it is zero for $b\ne a$ and sends $e\bmod z$ to the bottom class of $f$ for
$b=a$.  It is therefore an isomorphism
\begin{equation}\label{eq:smith-primitive-isomorphism}
 \operatorname{gr}_F^a(\mathsf L_!/z\mathsf L_!)
 \xrightarrow{\sim}\operatorname{Prim}_a(\mathcal C).
\end{equation}
For exponent zero, reduction of $c$ gives the canonical isomorphism
$F^0/F^1\xrightarrow{\sim}\operatorname{im}(c\bmod z)$.  Its grading is
determined as follows.  For a semisimple perverse summand $P$, the image of
$H^\bullet(i_x^!P)\to H^\bullet(i_x^*P)$ is the fibre of the maximal direct
summand of $P$ supported at $x$.  To see this, split off the point-supported
summand.  The map is the identity on that summand.  On an IC summand with
larger support, a nonzero image would produce a point-supported subobject and
quotient, contrary to simplicity.  A point-supported perverse sheaf has
cohomology only in perverse degree zero.  Under
\eqref{eq:grojnowski-source-degree}, this is Grojnowski degree zero.  Hence
the exponent-zero Smith quotient is concentrated in degree zero.

\smallskip
\noindent\emph{The degree comparison.}
Apply hard Lefschetz to
$\mathcal C\cong\mathbb H^\bullet(\mathcal Y,\overline I)[1]$.  A primitive
vector $p$ in hypercohomology degree $-\ell$ generates the string
\[
 p,zp,\ldots,z^\ell p,
 \qquad \deg_{\mathbb H}(z^kp)=-\ell+2k,
\]
of length $a=\ell+1$.  Formula~\eqref{eq:primitive-string-quotient} selects
the bottom vector $p$, which has degree $-\ell-1=-a$ in $\mathcal C$, and
\eqref{eq:smith-primitive-isomorphism} puts its corresponding lattice
generator $u$ in Smith layer $a$.  Choose a homogeneous lift
$v\in\mathsf L_*$ of $p$.  Then~\eqref{eq:degree-balance} gives
\[
 \deg(u)+\Delta_c=2a-a=a,
 \qquad\text{hence}\qquad |u|_{G,!}=a.
\]
Thus the degree of $c$ and the shift in the cone calculation together give
the asserted index.

\smallskip
\noindent\emph{Equality of filtrations.}
It suffices to pass from concentration of the quotients to equality of the
filtrations without assuming a global Smith basis.  The filtration $F$ is
graded, and the preceding calculation identifies
$\operatorname{gr}_F^a$ with a primitive quotient concentrated in
Grojnowski degree $a$; the exponent-zero quotient is concentrated in degree
zero.  For each fixed degree $j$, the induced finite filtration on
$H_G^j(\mathsf L_!/z\mathsf L_!)$ is exhaustive and separated and can
therefore have a nonzero quotient only at $a=j$.  It follows degree by
degree that
\[
 F^p\cap H_G^j=
 \begin{cases}
  H_G^j,&p\leq j,\\
  0,&p>j.
 \end{cases}
\]
Indeed, before the unique possible jump at $p=j$ the filtration is still
exhaustive on $H_G^j$, and after that jump separatedness forces it to be
zero.  Summing over all degrees gives
\[
 F^p(\mathsf L_!/z\mathsf L_!)=
 \bigoplus_{j\geq p}H_G^j(\mathsf L_!/z\mathsf L_!).
\]
Both
\eqref{eq:primitive-string-quotient} and
\eqref{eq:smith-primitive-isomorphism} are expressed using only kernels,
images, and multiplication by the central element $z$, so this identification
is canonical.

\smallskip
\noindent\emph{The radical and semisimplicity.}
For $e\in\mathcal E_m$, choose, as in
Lemma~\ref{lem:residual-grading}, a homogeneous lift
$\widetilde e\in(\mathcal B_D)_m$.  Its image in the one-variable algebra
$\mathcal B_\eta$ commutes with $c$ and $z$, so it acts
on~\eqref{eq:primitive-string-quotient}.  Changing $\widetilde e$ changes
the operator by a multiple of $t=z$ and hence has no effect on the special
fibre.  Thus the construction of~\eqref{eq:smith-primitive-isomorphism} is
canonically $\mathcal E$-linear.  The equality of filtrations just proved,
together with the fact that $\mathcal E_m$ raises residual Grojnowski degree
by $m$, gives
\[
 \mathcal E_mF^p\subset F^{p+m}.
\]
Consequently $\mathcal E_{>0}F^p\subset F^{p+1}$, so the positive-degree
ideal kills every associated-graded Smith layer.

Lemma~\ref{lem:residual-grading} gives
\[
 \mathcal E_0\cong\prod_\delta\operatorname{End}(V_\delta),
 \qquad \operatorname{rad}\mathcal E=\mathcal E_{>0}.
\]
By the string calculation above, this ideal acts trivially on every
$\operatorname{gr}^p_F$.  Each layer is therefore an
$\mathcal E_0$-module and is semisimple.  The splitting identifies the
multiplicity space of the simple factor indexed by $\delta$ with
$H_G^p(\mathsf N_x^!\IC_\delta)$, with the normal shifts and twists.
Through the central-fibre identification in
Lemma~\ref{lem:residual-grading}, this is the required
semisimplicity and IC multiplicity statement for the specialized graded
affine Hecke action.
\end{proof}

\section{The fixed-slice comparison}
\label{sec:fixed-slice-comparison}

This section applies Section~\ref{sec:local-smith} to the graded Slodowy
slice attached to an enhanced parameter.  It compares the analytic and
geometric conventions, constructs the positive fixed slice, retains the
component-group enhancement, and identifies the completed geometric family
with Langlands induction.  The final subsection proves the main theorem.

\subsection*{Analytic versus geometric standard-module conventions}
Let $y\in\mathfrak g$ be nilpotent, let $r>0$, and let
$\sigma_{\mathrm{an}}\in\mathfrak g$ be semisimple with
$[\sigma_{\mathrm{an}},y]=-2ry$.  Let $\rho$ be an irreducible
representation of $\pi_0Z_G(\sigma_{\mathrm{an}},y)$ whose associated local
system occurs in the chosen generalized-Springer block.  Denote by
$X^{\rm an}_{y,\sigma_{\rm an},r,\rho}$ the standard module in the
analytic Langlands convention and
$E_{y,\sigma_{\mathrm g},r,\rho}$ for Solleveld's geometric standard module.
The sheaf-theoretic construction below naturally produces the latter.  Put
$\sigma_{\mathrm g}=-\sigma_{\mathrm{an}}$; then
$[\sigma_{\mathrm g},y]=2ry$.
Let $\operatorname{sgn}^*$ denote Solleveld's sign-twist functor, let
$\operatorname{IM}$ be the Iwahori--Matsumoto involution, and let $\mathbf r$
be the central parameter generator in his universal graded affine Hecke
algebra; $\mathbf r$ specializes to the scalar $r$.

\begin{lemma}[Analytic/geometric transport]
\label{lem:analytic-geometric-transport}
There is a natural identification
\[
 X^{\rm an}_{y,\sigma_{\rm an},r,\rho}
 \simeq \operatorname{IM}^{*}E_{y,\sigma_{\rm g},r,\rho}.
\]
If $X(M,\delta,\nu+t\eta)$ is an analytic Langlands family
and $Q=M$, then
\[
 X(M,\delta,\nu+t\eta)\simeq
 \operatorname{IM}^*\!\left(
  \Hh\otimes_{\Hh_Q}
  \bigl((\operatorname{IM}_Q)^*\delta\bigr)_{-\nu-t\eta}
 \right).
\]
After the coefficient rings of the two arcs are identified by the sign
pullback $f(\lambda)\mapsto f(-\lambda)$, the functor
$\operatorname{IM}^*$ is $\mathbb C[[t]]$-linear.  It transports the
standard-to-contragredient map and its Jantzen filtration exactly; in
particular, it preserves elementary divisors, layers, semisimplicity, and
composition multiplicities.
\end{lemma}

\begin{proof}
Solleveld's identity in Section~3.1, immediately before equation~(3.11), is
\[
 X^{\rm an}_{y,\sigma_{\rm an},r,\rho}
 =\operatorname{sgn}^{*}E_{y,\sigma_{\rm an},-r,\rho}
 \simeq \operatorname{IM}^{*}E_{y,\sigma_{\rm g},r,\rho}.
\]
On the Bernstein generators,
\[
 \operatorname{IM}(N_w)=\operatorname{sgn}(w)N_w,\qquad
 \operatorname{IM}(\xi)=-\xi,\qquad
 \operatorname{IM}(\mathbf r)=\mathbf r.
\]
Thus $\operatorname{IM}$ restricts to $\operatorname{IM}_Q$, commutes with
parabolic induction, and sends the central arc $-\nu-t\eta$ to
$\nu+t\eta$, which proves the family identity.  The same generator check
shows that $\operatorname{IM}$ commutes with the linear transpose
anti-involution.  Hence it transports both contragredients and the canonical
intertwining map.  The parameter $t$ is unchanged under the signed
base-ring identification, so inverse images of $t^n$ correspond for every
$n$.  Exactness of this fixed equivalence proves the remaining assertions.
\end{proof}

Thus an analytic central arc $\nu+t\eta$ is represented geometrically by
$-\nu-t\eta$.  Relative to the same positive roots this lies in the opposite
chamber.  The fixed-slice construction uses the opposite parabolic and
attractive nilradical; for that choice the geometric direction $-\eta$
is positive.  The same condition $t>0$ orients both arcs.  The enhanced-orbit
labels $\gamma,\delta$ below are understood through this bijection.  Whenever
a Langlands family with Levi $Q=M$ is fixed, put
$Q_{\mathrm{der}}=[Q,Q]$, and write $\Hh_{Q,\mathrm{der}}$ for the graded
affine Hecke algebra of the derived root system of $Q$.  If $\tau$ is an
$\Hh_{Q,\mathrm{der}}$-module, write $\tau_\lambda$ for its unramified
central twist by $\lambda\in\mathfrak z(\mathfrak q)$.  The Levi module used
in the geometric coefficient comparison below is therefore the
Iwahori--Matsumoto transform of the analytic tempered inducing module.

\subsection{The graded Slodowy slice}
\label{sec:graded-slodowy-slice}

The principal realization is constructed in arbitrary type.  The geometric
comparison is first formulated for a generalized Springer cuspidal datum and
its linear-dual datum; the main theorem uses the self-dual principal
specialization.  This subsection constructs a transverse slice on which the
deformation cocharacter has positive weights and compares the generalized
Springer diagram with its restriction to that slice.  The following
subsections add local systems and enhancements and identify the geometric
family with the algebraic standard module.

Let $G$ be a connected reductive group with Lie algebra $\mathfrak g$, let
$r>0$, and let $\sigma\in\mathfrak g$ be real semisimple.  Put
\[
 \mathfrak g_{\sigma,2r}=\{v\in\mathfrak g:[\sigma,v]=2rv\},
 \qquad G_\sigma=Z_G(\sigma),\qquad
 \mathfrak g_\sigma=\operatorname{Lie}G_\sigma=\mathfrak g_{\sigma,0}.
\]
$\mathfrak g_{\sigma,2r}$ is the localized parameter variety used in the geometric construction
of standard modules; compare \cite[Sections~2--4]{Solleveld}.  Let
$y\in\mathfrak g_{\sigma,2r}$ and choose an $\mathfrak{sl}_2$-triple
$(y,h,f)$ such that
\[
 \sigma_0=\sigma-rh
\]
centralizes the triple.  Such a homogeneous triple exists by the graded
Jacobson--Morozov theorem~\cite[Theorem~1]{Vinberg}; in the present
geometric-parameter convention, compare
\cite[equation~(26) and Proposition~3.5(c)]{AMS} and
\cite[equation~(3.6) and Lemma~B.3]{Solleveld}.  Set
$G_{\sigma_0}=Z_G(\sigma_0)$, and choose the attractive
parabolic $P_{\rm att}=G_{\sigma_0} U_{\rm att}$.  Write
$\mathfrak p_{\rm att}=\operatorname{Lie}P_{\rm att}$ and
$\mathfrak u_{\rm att}=\operatorname{Lie}U_{\rm att}$, and choose
$P_{\rm att}$ so that the $\sigma_0$-weights on
$\mathfrak u_{\rm att}$ are positive.
Let $\eta:\mathbb C^\times\to Z(G_{\sigma_0})^\circ$ be an integral cocharacter in the
same open chamber: whenever a nonzero $Z(G_{\sigma_0})^\circ$-weight $\alpha$ satisfies
$\alpha(\sigma_0)>0$, require $\langle\alpha,\eta\rangle>0$.

\begin{proposition}[Graded Slodowy contraction]
\label{prop:general-graded-slodowy}
The vector space
\[
N_{\sigma,r}(y)=\mathfrak g^f\cap\mathfrak g_{\sigma,2r},
\qquad \mathfrak g^f=\ker(\operatorname{ad}f),
\]
is an $\eta$-stable normal complement to the $G_\sigma$-orbit of $y$ in
$\mathfrak g_{\sigma,2r}$:
\begin{equation}\label{eq:general-slodowy-splitting}
 \mathfrak g_{\sigma,2r}
 =[\mathfrak g_\sigma,y]\oplus N_{\sigma,r}(y).
\end{equation}
Every $\eta$-weight on $N_{\sigma,r}(y)$ is strictly positive.  Consequently
\[
 V_{\sigma,r}(y)=y+N_{\sigma,r}(y)
\]
is a transverse affine slice contracted to $y$ by $\eta$; its intersection
with every $G_\sigma$-orbit closure is an attractive affine cone.
\end{proposition}

\begin{proof}
Decompose $\mathfrak g$ simultaneously under $\sigma_0$ and the adjoint
$\mathfrak{sl}_2$.  On an irreducible $\mathfrak{sl}_2$-summand of highest
weight $m$, on which $\sigma_0$ has eigenvalue $a$, the $h$-weights are
$m,m-2,\ldots,-m$.  Since $\sigma=\sigma_0+rh$, its eigenvalue on the
$h$-weight-$j$ line is $a+rj$.

The map $\operatorname{ad}y$ raises $h$-weight by two.  Hence a
$2r$-eigenvector of $\sigma$ lies in
$[\mathfrak g_\sigma,y]$ unless it is the lowest-weight vector of its
$\mathfrak{sl}_2$-summand.  Indeed, if $v$ has $h$-weight $j$ and
$\sigma$-eigenvalue $a+rj=2r$, and if $v$ is not a lowest-weight vector,
then $v=yv'$ for a vector $v'$ of $h$-weight $j-2$.  The $\sigma$-eigenvalue
of $v'$ is
\[
 a+r(j-2)=(a+rj)-2r=0,
\]
so $v'\in\mathfrak g_\sigma$ and $v\in[\mathfrak g_\sigma,y]$.  The vectors
not obtained in this way are exactly the lowest-weight vectors, namely those
in $\mathfrak g^f$.  This proves the direct sum
\eqref{eq:general-slodowy-splitting}; it is the graded form of the usual
Slodowy decomposition.

For a lowest-weight vector $v$ in a highest-weight-$m$ summand, membership in
$\mathfrak g_{\sigma,2r}$ says
\[
 2r=a-rm,
 \qquad\text{hence}\qquad a=r(m+2)>0.
\]
Thus $v$ belongs to a positive $\sigma_0$-weight space, and in particular to
$\mathfrak u_{\rm att}$.  The chamber condition gives
$\operatorname{wt}_\eta(v)>0$.  This proves positivity on the whole normal
space.  Since $\eta$ centralizes the triple, it fixes $y$ and preserves the
slice.  The differential at $(1,y)$ of the action map
\[
 G_\sigma\times V_{\sigma,r}(y)\longrightarrow\mathfrak g_{\sigma,2r}
\]
is surjective, with kernel
$\mathfrak z_{\mathfrak g_\sigma}(y)\oplus0$.  After quotienting this
kernel it is the isomorphism
$[\mathfrak g_\sigma,y]\oplus N_{\sigma,r}(y)\to
\mathfrak g_{\sigma,2r}$ of \eqref{eq:general-slodowy-splitting}.  Thus the slice is transverse to the
orbit at $y$.  Finally, in the affine coordinates
$V_{\sigma,r}(y)=y+N_{\sigma,r}(y)$, the action of $\eta$ is linear about
$y$ and has strictly positive weights.  It contracts every point to $y$;
the intersection with a $G_\sigma$-stable orbit closure is consequently a
closed, positively graded affine cone with vertex $y$.
\end{proof}

The same $\mathfrak{sl}_2$ calculation appears in the proof that Solleveld's
determinant class is nonzero; see \cite[Appendix~B, Lemma~B.3]{Solleveld}.
Here it identifies the full normal representation and its contracting
chamber.

Use Lusztig's auxiliary induction slice
\[
 \mathfrak a=y+\mathfrak z_{\mathfrak p_{\rm att}}(f)
\]
where
$\mathfrak z_{\mathfrak p_{\rm att}}(f)=
\ker(\operatorname{ad}f)\cap\mathfrak p_{\rm att}$, and denote its proper
incidence variety by $\dot{\mathfrak a}$.  Equip it
with any generalized Springer cuspidal local system.  Let $T_{\sigma,r}$ be
the smallest torus in $G\times\mathbb C^\times$
whose Lie algebra contains $(\sigma,r)$, acting by
$(g,\zeta)v=\zeta^{-2}\operatorname{Ad}(g)v$.
Let $\mathcal P$ be the $G$-conjugacy class of parabolics in the cuspidal
datum, choose $P_0\in\mathcal P$, and identify $\mathcal P\simeq G/P_0$.

\begin{proposition}[Completed fixed-slice comparison]
\label{prop:general-fixed-slice}
The fixed part of the auxiliary slice is
\begin{equation}\label{eq:general-fixed-slice}
 \mathfrak a^{T_{\sigma,r}}=V_{\sigma,r}(y).
\end{equation}
The fixed incidence
variety over this slice is
\[
 \dot{\mathfrak a}^{T_{\sigma,r}}
 =\dot{\mathfrak a}\cap
 \bigl(V_{\sigma,r}(y)\times(G/P_0)^{\exp(\mathbb C\sigma)}\bigr).
\]
With the restricted cuspidal local system, completed localization identifies
the full equivariant diagram
with this fixed diagram, compatibly with proper pushforward and the
Euler-normalized convolution action.  At each fixed locus occurring in this diagram, let
$N_{\rm mov}$ denote the sum of the nonzero $T_{\sigma,r}$-weight summands in
its normal bundle.  Along the arc $(\sigma+t\eta,r)$ every corresponding
equivariant Euler class is a unit after base change to $\mathbb C[[t]]$, so
this comparison preserves $t$-adic orders.
\end{proposition}

The fixed incidence variety is the restriction to the slice of Solleveld's
localized generalized Springer diagram.  It is generally a proper closed
subvariety of the ordinary Cartesian inverse image of
$V_{\sigma,r}(y)$.

\begin{proof}
The $T_{\sigma,r}$-fixed subspace of $\mathfrak g$ is exactly
$\mathfrak g_{\sigma,2r}$; this is the fixed-point description underlying
the localized complex $K_{\sigma,r}$ in \cite[Section~2.1]{Solleveld}.
Proposition~\ref{prop:general-graded-slodowy} shows that
$N_{\sigma,r}(y)\subset\mathfrak u_{\rm att}\subset\mathfrak p_{\rm att}$.
Hence
\[
 \mathfrak z_{\mathfrak p_{\rm att}}(f)\cap\mathfrak g_{\sigma,2r}
 =\mathfrak g^f\cap\mathfrak g_{\sigma,2r}
 =N_{\sigma,r}(y),
\]
which proves \eqref{eq:general-fixed-slice}.

In the incidence variety a fixed point has Lie-algebra coordinate in
$V_{\sigma,r}(y)$ and parabolic coordinate fixed by
$\exp(\mathbb C\sigma)$.  Only the $G$-factor acts on the parabolic
coordinate; the additional $\mathbb C^\times$-factor acts on the
Lie-algebra coordinate and hence imposes no further condition on $P$.
This proves the displayed description of the closed subvariety.  In
Solleveld's notation it is the restriction of
$\dot{\mathfrak g}_{\sigma,r}$; the square defining that variety is in
general non-Cartesian~\cite[Section~2.1, equations (2.11)--(2.14)]{Solleveld}.
The cuspidal coefficient system restricts canonically by equivariance.
Completed support localization identifies the equivariant (co)homology of
the full diagram with that of this fixed locus
\cite[Theorem~6.2(1)--(3)]{GKM}.  The two compatibilities needed later are
recorded next.

For convolution, use localization normalized by the moving Euler
classes in the fixed-locus Gysin maps.  In this normalization proper
pushforward, the refined pullbacks in convolution, and the action on
equivariant fibre homology commute with localization.  Equivalently,
the comparison identifies composition of the corresponding derived
morphisms; see \cite[Proposition~2.4 and Theorem~2.5]{Solleveld}.
All coefficient and component structures are retained.

This concerns the fixed-incidence comparison and its algebra action.
The normal map on the fixed parameter variety is defined separately
by \eqref{eq:normal-fibre-map}; its slice comparison is
\eqref{eq:normal-slice-arrow}.  No identification of ambient point
adjunctions under smooth restriction is asserted.

Put $T=T_{\sigma,r}$, fix one of the moving normal bundles just described,
and let
$\lambda\in X^*(T)$ be a weight occurring in $N_{\rm mov}$.  By definition,
such a weight is \emph{moving} when it is nontrivial; its weight space is
then transverse to, rather than tangent to, the fixed locus.  The analytic
one-parameter subgroup $\exp(\mathbb C(\sigma,r))$ is Zariski dense in the
smallest algebraic torus $T$ whose Lie algebra contains $(\sigma,r)$.  Hence
a nontrivial character $\lambda$ cannot be trivial on this subgroup, and its
differential satisfies
\[
 d\lambda(\sigma,r)\ne0.
\]
Along the deformation arc, the corresponding equivariant Euler factor is
\[
 d\lambda(\sigma,r)+t\,d\lambda(\eta,0),
\]
which has nonzero constant term and is therefore a unit of
$\mathbb C[[t]]$.  After applying the splitting principle, every factor of
$e_T(N_{\rm mov})$ has this form, up to positive-degree Chern classes; its
nonzero scalar term makes it invertible in the completed equivariant
coefficient ring (the positive-degree part is nilpotent on each
finite-dimensional fixed component).  Thus the full moving Euler class is a
unit.  Completed
localization therefore preserves $t$-adic orders and the open normalization
up to a unit.
\end{proof}

\begin{remark}[Localization input]
\label{rem:solleveld-version}
References to~\cite{Solleveld} use arXiv:2106.03196v3 (January 2025).  The
numbering cited here agrees with the final published version in
\emph{Journal of Algebra}.  The
fixed-point homology localization used in
Proposition~\ref{prop:general-fixed-slice} is Proposition~A.2 of that
version.  Although stated there for an affine variety, its proof extends by
d\'evissage to a finite-type torus variety with a finite stable affine
stratification; the incidence varieties here have such stratifications.
Proposition~2.4 compares the localized and global endomorphism
algebras, and Theorem~2.5 identifies their completions with the completed
Hecke algebra.  Appendix~A of that paper contains the required localization
argument.
\end{remark}

\begin{lemma}[Clean graded Slodowy pullback]
\label{lem:general-clean-slodowy}
Let $K_{\sigma,r}$ be the localized generalized Springer complex and let
$d=\dim G_\sigma y$.  On the graded Slodowy slice
$i_V:V_{\sigma,r}(y)\hookrightarrow\mathfrak g_{\sigma,2r}$ there is a
canonical clean-pullback isomorphism
\[
 i_V^!K_{\sigma,r}[2d](d)\xrightarrow{\sim}i_V^*K_{\sigma,r}.
\]
The same assertion holds for every pure IC summand of $K_{\sigma,r}$.
\end{lemma}

\begin{proof}
Write $V=V_{\sigma,r}(y)$.
\smallskip
\noindent\emph{A finite Whitney stratification.}
Choose a Cartan subalgebra containing $\sigma$.  The sets of roots satisfying
$\alpha(\sigma)=0$ and $\alpha(\sigma)=2r$, together with the signs of all
remaining numbers $\alpha(\sigma)$ and $\alpha(\sigma)-2r$, are constant on
the relative interior of a rational polyhedral face in the real parameter
space.  Choose a rational point $(\sigma',r')$ in that same face and multiply
it by a positive integer.  Then
\[
 Z_G(\sigma')=Z_G(\sigma),\qquad
 \mathfrak g_{\sigma',2r'}=\mathfrak g_{\sigma,2r},
\]
and the attractive parabolic is unchanged.  The finite-orbit theorem for
integral graded Lie algebras therefore applies: $G_\sigma$ has only
finitely many orbits on the nilpotent graded piece
\cite[Section~2.1]{LusztigGraded}.  The orbit
partition is Whitney.  Indeed, Whitney regularity holds on a nonempty open
part of each incident pair of algebraic strata
\cite[Lemma~19.3]{Whitney}, and equivariance plus
transitivity propagates that open condition over the smaller orbit.
Solleveld's description of $K_{\sigma,r}$ as a proper direct image from the
localized incidence diagram shows that it is constructible for this
partition~\cite[Section~2.1 and Section~4]{Solleveld}.

\smallskip
\noindent\emph{Clean pullback near the base point.}
By~\eqref{eq:general-slodowy-splitting}, $V$ is transverse at $y$ to the
$G_\sigma$-orbit through $y$.  A standard consequence of Whitney
condition~(a) is that this transversality persists after shrinking the
slice: there is an open neighbourhood $V^\circ\subset V$ of $y$ such that
$V^\circ$ is transverse to every orbit stratum that it meets.

Put $g_0=\dim G_\sigma$, and consider the action and projection maps
\[
 a:G_\sigma\times V^\circ\longrightarrow\mathfrak g_{\sigma,2r},
 \qquad p:G_\sigma\times V^\circ\longrightarrow V^\circ
\]
given by $a(g,v)=\operatorname{Ad}(g)v$ and $p(g,v)=v$.  At a point
$(1,v)$, the image of the differential of $a$ is
\[
 T_v(G_\sigma v)+T_vV^\circ.
\]
This is the whole tangent space of $\mathfrak g_{\sigma,2r}$ by the
transversality just established.  Hence $a$ is a submersion at every
$(1,v)$, and $G_\sigma$-translation gives the same conclusion at every
$(g,v)$.  Thus $a$ is smooth.  Since
$\dim V=\dim\mathfrak g_{\sigma,2r}-d$, its relative complex dimension is
\[
 \dim(G_\sigma\times V)-\dim\mathfrak g_{\sigma,2r}=g_0-d.
\]
The projection $p$ is smooth of relative complex dimension $g_0$.

\smallskip
\noindent\emph{Propagation of transversality.}
Every $\eta$-orbit in $V$ meets $V^\circ$: the orbit map extends to
$\mathbb A^1$ with value $y$ at zero, and $V^\circ$ is an open
neighbourhood of $y$.  Moreover, $\eta$ preserves $V$ and carries
$G_\sigma$-orbits to themselves.  Its differential therefore transports
the equality
$T_v(G_\sigma v)+T_vV=T_v\mathfrak g_{\sigma,2r}$
between points of the same $\eta$-orbit.  Since that equality holds on
$V^\circ$, it holds on all of $V$.
Consequently the action map
$a:G_\sigma\times V\to\mathfrak g_{\sigma,2r}$ is smooth everywhere,
of relative dimension $g_0-d$.  In what follows $a$ and $p$ denote
the maps on $G_\sigma\times V$.

\smallskip
\noindent\emph{The global pullback comparison.}

Consider the two pullbacks of $K_{\sigma,r}$.  The
$G_\sigma$-equivariant structure on $K_{\sigma,r}$ gives
\[
 a^*K_{\sigma,r}\simeq p^*i_V^*K_{\sigma,r}.
\]
The same equivariant structure, together with smooth purity for the action
and projection maps, gives the extraordinary-pullback comparison
\[
 a^!K_{\sigma,r}\simeq p^!i_V^!K_{\sigma,r}.
\]
Recall that if $f$ is smooth of relative complex dimension $e$, then
\[
 f^!\simeq f^*[2e](e).
\]
Applying this formula first to $p$ and then to $a$ yields the chain
\[
\begin{aligned}
 p^*i_V^!K_{\sigma,r}[2g_0](g_0)
 &\simeq p^!i_V^!K_{\sigma,r}\\
 &\simeq a^!K_{\sigma,r}\\
 &\simeq a^*K_{\sigma,r}[2(g_0-d)](g_0-d)\\
 &\simeq p^*i_V^*K_{\sigma,r}
       [2(g_0-d)](g_0-d).
\end{aligned}
\]
Cancelling the common shift $[2(g_0-d)](g_0-d)$ gives
\[
 p^*\bigl(i_V^!K_{\sigma,r}[2d](d)\bigr)
 \simeq p^*i_V^*K_{\sigma,r}.
\]
The projection $p$ has the section $s(v)=(1,v)$.  Applying $s^*$ therefore
removes $p^*$ and gives
\[
 i_V^!K_{\sigma,r}[2d](d)
 \xrightarrow{\sim}i_V^*K_{\sigma,r}.
\]
This constructs the asserted isomorphism globally, after global
smoothness has been established.  In particular, it does not assume
an inverse purity morphism before proving that the pullback is clean.

Finally, a decomposition-theorem summand is the image of an idempotent
endomorphism of $K_{\sigma,r}$.  Naturality of this isomorphism makes it commute with
every such idempotent.  The isomorphism therefore restricts to each pure IC
summand, as asserted.
\end{proof}

\begin{corollary}[Absolute layer normalization]
\label{cor:absolute-layer-normalization}
Let $i_\gamma:\{y\}\hookrightarrow\mathfrak g_{\sigma,2r}$, let
$s=i_V:V_{\sigma,r}(y)\hookrightarrow\mathfrak g_{\sigma,2r}$, and put
$d_\gamma=\dim G_\sigma y$.  For a pure perverse IC summand $\IC_\delta$ set
\[
 I_{\gamma\delta}=s^!\IC_\delta[d_\gamma]
 \simeq s^*\IC_\delta[-d_\gamma](-d_\gamma).
\]
Then $I_{\gamma\delta}$ is perverse on the transverse slice.  Assume that
the completed costalk and stalk pulled back to the chosen one-variable arc
are finite free.  Equip the geometric standard-module special fibre
$X_\gamma$ with the Smith filtration $F^\bullet$ induced by the pulled-back
normal map \eqref{eq:normal-fibre-map}.  With the smooth summand $\delta=\gamma$ placed in
layer zero, the absolute layer index is
\[
 [\operatorname{gr}_F^nX_\gamma:L_\delta]
 =\dim\operatorname{Hom}_{A_\gamma}\!\left(
   \rho_\gamma,
   \mathcal H^{n+d_\gamma}i_\gamma^!\IC_\delta
 \right).
\]
\end{corollary}

\begin{proof}
Let $k:\{y\}\hookrightarrow V_{\sigma,r}(y)$.  The clean-pullback
normalization gives
\[
k^!I_{\gamma\delta}=i_\gamma^!\IC_\delta[d_\gamma].
\]
Consequently
\[
 \mathcal H^n(k^!I_{\gamma\delta})
 \cong
 \mathcal H^{n+d_\gamma}(i_\gamma^!\IC_\delta).
\]
With the Grojnowski regrading fixed in
Proposition~\ref{prop:ext-equivariance}, the left-hand side is the
degree-$n$ part of the regraded costalk on the transverse slice.  The
Smith--primitive isomorphism~\eqref{eq:smith-primitive-isomorphism}
identifies this degree-$n$ part with the contribution of $\IC_\delta$ to
the $n$-th Smith layer, once the common translation of the grading has
been fixed.  For
$\delta=\gamma$ the smooth point costalk satisfies
\[
 \mathcal H^{d_\gamma}i_\gamma^!\IC_\gamma\cong\rho_\gamma.
\]
The point-supported head summand therefore occurs in Smith layer zero and
fixes the common translation.  This proves the asserted absolute layer
formula.
\end{proof}

\subsection{The local theorem}

\begin{theorem}[Local Smith theorem]
\label{thm:general-local}
Let $y\in\mathfrak g_{\sigma,2r}$ and let $\eta$ be a contracting
cocharacter as in Proposition~\ref{prop:general-graded-slodowy}.
Let $K_{\sigma,r}$ be a localized Springer complex whose perverse IC
summands are pure and polarizable.  Form the completed normal map
\eqref{eq:normal-arc-lattices} using the connected point stabilizer.
Its Smith filtration is convolution-equivariant and has semisimple
layers.  If the coefficient homomorphism is component-invariant, the same holds after
projection to an enhancement $\rho_\gamma$.  With the head in layer
zero,
\[
 [\operatorname{gr}_F^nX_\gamma:L_\delta]
 =\dim\operatorname{Hom}_{A_\gamma}\!\left(
 \rho_\gamma,\mathcal H^{n+d_\gamma}i_\gamma^!\IC_\delta\right).
\]
Here $X_\gamma$ denotes the geometric standard-module special fibre;
the algebraic family is identified in
Proposition~\ref{prop:general-coefficient-arc}.
\end{theorem}

\begin{proof}
\emph{The normal arrow.}
Put $V=V_{\sigma,r}(y)$, $s:V\hookrightarrow X_\chi$, and
$k:\{y\}\hookrightarrow V$.  Proposition~\ref{prop:general-graded-slodowy}
makes $V$ a positive transverse slice.
Lemma~\ref{lem:general-clean-slodowy} gives the pure perverse object
\[
 I_{\gamma\delta}=s^!\IC_\delta[d_\gamma]
 \simeq s^*\IC_\delta[-d_\gamma](-d_\gamma).
\]
Transverse base change \eqref{eq:normal-slice-arrow} identifies the
orbit-restriction map with $k^!I_{\gamma\delta}\to k^*I_{\gamma\delta}$.
The ordinary ambient point map is not used.
If $V$ is a point, this map is the identity.  The normal module is
concentrated in residual degree zero, so the positive-degree ideal acts
by zero; the conclusion follows with only layer zero.  Henceforth
assume $\dim V>0$.

\smallskip
\emph{Coefficients and freeness.}
The normal fibres are equivariant for $H=C_{\chi,y}^\circ$.
Their ordinary cohomology is pure by
Lemma~\ref{lem:normal-purity-base-change}, which also gives finite
freeness over the stabilizer coefficient ring and ordinary base change
for restriction to a subtorus.  Quotient $A=T_{\sigma,r}$.
The arc becomes $t\bar\eta$ in $H/A$.  Restriction to its one-dimensional
image identifies the pulled-back normal map with the cone map over
$\mathbb C[[z]]$, followed by
\[
 z\longmapsto c_\eta t,\qquad c_\eta>0.
\]
Both lattices are therefore finite free of equal rank and Hypothesis
\ref{hyp:local-cone-compatibility} holds.  This coefficient comparison
is needed in addition to transversality.

\smallskip
\emph{Smith layers and the algebra action.}
Lemma~\ref{lem:cone-calculation} identifies the torsion cokernel on each
positive-dimensional support with the projectivized punctured-slice
cohomology, shifted by $[1]$.  A length-$a$ Lefschetz string gives a
Smith block $z^a$, hence $t^a$ up to a unit.  A point-supported summand
gives an identity block.  Proposition~\ref{prop:ext-equivariance}
identifies this filtration with the residual degree filtration.

The acting algebra is still the full specialized convolution algebra
$\mathcal E_\chi^{\rm res}$; only its fibre representation passes
through the stabilizer.  Its degree-$m$ part raises normal degree by
$m$, and its radical is the positive-degree ideal
(Lemma~\ref{lem:residual-grading}).  The radical kills each Smith
layer, proving semisimplicity.  For a component-invariant arc, exact
isotypic projection commutes with this calculation.

Finally $k^!I_{\gamma\delta}=i_\gamma^!\IC_\delta[d_\gamma]$.
Normal degree $n$ is therefore ambient costalk degree $n+d_\gamma$.
For $\delta=\gamma$ the normal restriction is point-supported; its
$\rho_\gamma$-fibre puts the head in degree zero.  This proves the
formula, with the normalization of
Corollary~\ref{cor:absolute-layer-normalization}.
\end{proof}

\subsection{Enhancements and the integral family}

Two component groups enter the construction: the group attached to the
fixed geometric Levi datum and the group attached to the final orbit.  The
comparison proceeds in three steps.  First, the groups are identified and
the common isotypic summand is selected.  Next, the universal Levi-centre
coefficient ring is split from the fixed Levi module.  Finally, completed
Langlands induction is compared with the universal equivariant costalk
family.

For reference, the external inputs from Solleveld used in this subsection
are the following.  Equation~(B.1) is the change-of-equivariance
isomorphism.  Lemma~B.3 proves that the determinant class $\epsilon_-$ is
nonzero at the Langlands parameter, and Theorem~B.2(a) says that the
induction map is injective with image containing $\epsilon_-$ times the
target.  Proposition~3.5(a), including equation~(3.16), identifies the
completed source with the equivariant costalk; Proposition~3.5(b), equations
(3.17)--(3.18), gives the dual stalk statement.  Theorem~2.5 is used only to
identify the completed convolution action with the completed graded affine
Hecke action.  These statements are applied with all coefficient and
component-group structures retained.

Retain $y$ and $r$ from the fixed-slice construction, and let $Q$ be the
Langlands Levi.  Let $\delta_{\rm an}$ be the analytic tempered
$\Hh_Q$-module represented by a semisimple point
$s_{\delta,{\rm an}}\in\mathfrak q$ satisfying
$[s_{\delta,{\rm an}},y]=-2ry$, together with an enhancement $\rho_Q$.
Put
$s_{\delta,{\rm g}}=-s_{\delta,{\rm an}}$, and let
$\delta_{\rm g}=\operatorname{IM}_Q^*\delta_{\rm an}$ be the corresponding
geometric Levi module.  Let $\rho$ be the enhancement of the final orbit.
Write the full
geometric parameter as
\[
 \sigma_{\rm g}=s_{\delta,{\rm g}}+\nu_{\rm g},
 \qquad \nu_{\rm g}\in\mathfrak z(\mathfrak q).
\]

Choose the homogeneous triple $(y,h,f)$ so that
$\sigma_0=\sigma_{\rm g}-rh$ centralizes it, and take
$Q=Z_G(\sigma_0)$, as in the Langlands normalization above.  Set
\[
 \mathcal R_y=Z_G(y,h,f,\sigma_0),\qquad
 \mathcal U_y=Z_{R_uZ_G(y)}(\sigma_{\rm g}),
\]
where $R_u$ denotes the connected unipotent radical.  The group
$\mathcal R_y$ can be disconnected.  It is its component group,
not equality of the full stabilizers, that controls enhancements.

\begin{proposition}[Stabilizers, coefficients, and enhancements]
\label{prop:two-component-groups}
There is a Levi decomposition
\begin{equation}\label{eq:component-stabilizers}
 Z_G(\sigma_{\rm g},y)=\mathcal U_y\rtimes\mathcal R_y,
 \qquad
 \mathcal R_y=Z_Q(\sigma_{\rm g},y)=Z_Q(s_{\delta,{\rm g}},y).
\end{equation}
$\mathcal U_y$ is connected and unipotent.  Inclusion therefore identifies
\[
 A_y=\pi_0Z_G(\sigma_{\rm g},y)
 \xleftarrow{\ \sim\ }
 \pi_0Z_Q(s_{\delta,{\rm g}},y)=A_Q.
\]
Restriction to the corresponding Levi factor of the scaling stabilizer
identifies the connected-equivariant coefficient rings and normal-fibre
cohomology, compatibly with the normal map and its Hecke action.
The universal central-twist coefficient homomorphism is $A_Q$-invariant.
After this base change the identification $\rho_Q=\rho$ gives the same
isotypic projector on both sides of the normal map.
\end{proposition}

\begin{proof}
\emph{The reductive factor.}
The homogeneous triple gives
$Z_G(y)=R_uZ_G(y)\rtimes Z_G(y,h,f)$;
see \cite[(4.1.1)(b)--(d)]{McNinch} for the associated-cocharacter
form of this decomposition.
One way to see the decomposition is to use the $h$-grading: on
$\mathfrak g^y$ the weights are nonnegative, its weight-zero part is
$\mathfrak g^{y,h,f}$, and its positive-weight part is the Lie algebra
of the unipotent radical.  Conjugation by the triple cocharacter
contracts that radical and retracts the centralizer onto the
triple centralizer; this also accounts for disconnected components.
Conjugation by $\exp(\mathbb C\sigma_{\rm g})$ preserves both factors.
Taking its fixed points in the unique product decomposition gives
\eqref{eq:component-stabilizers}, with reductive factor
$Z_G(y,h,f,\sigma_0)$.  The exponential isomorphism for complex
unipotent groups identifies $\mathcal U_y$ with the fixed subspace
of its Lie algebra, so $\mathcal U_y$ is connected.

An element of $Q$ centralizing $\sigma_{\rm g}$ centralizes
$h=(\sigma_{\rm g}-\sigma_0)/r$.  If it also centralizes $y$, it
centralizes the triple, so $Z_Q(\sigma_{\rm g},y)=\mathcal R_y$.
Finally $\nu_{\rm g}$ is central in $Q$, giving the equality with
$Z_Q(s_{\delta,{\rm g}},y)$.  The resulting component-group
identification is the one used in Langlands induction
\cite[Theorem~B.2(c) and Proposition~B.4]{Solleveld}; it does not
require equality of the full stabilizers.

\smallskip
\emph{The scaling stabilizer and its coefficients.}
Let $\lambda_y:\mathbb C^\times\to G$ be the cocharacter of the
triple, so $\operatorname{Ad}(\lambda_y(a))y=a^2y$.  Put
\[
 \widetilde H_y=Z_{Z_G(\sigma_{\rm g})\times\mathbb C^\times}(y),
 \qquad
 \widetilde{\mathcal R}_y
 =\{(l\lambda_y(a),a):l\in\mathcal R_y,\ a\in\mathbb C^\times\}.
\]
Since $\lambda_y$ commutes with $\sigma_{\rm g}$ and $\mathcal R_y$,
\[
 \widetilde H_y=\mathcal U_y\rtimes\widetilde{\mathcal R}_y,
 \qquad
 \widetilde{\mathcal R}_y
 =Z_{Z_Q(\sigma_{\rm g})\times\mathbb C^\times}(y)
 \simeq\mathcal R_y\times\mathbb C^\times.
\]
In the last expression the identification is the displayed embedding
$(l,a)\mapsto(l\lambda_y(a),a)$.
In particular $A=T_{\sigma,r}$ lies in this reductive factor.
Write $H=\widetilde H_y^\circ$ and
$H_Q=\widetilde{\mathcal R}_y^\circ$.
For either normal vertex object $\mathsf N$, restriction gives
\[
 H_H^\bullet(\mathrm{pt})\xrightarrow{\sim}
 H_{H_Q}^\bullet(\mathrm{pt}),\qquad
 H_H^\bullet(\mathsf N)\xrightarrow{\sim}
 H_{H_Q}^\bullet(\mathsf N).
\]
Indeed, the map of Borel constructions for $H_Q\subset H$ has
fibre $H/H_Q\simeq\mathcal U_y$, which is contractible and has
ordinary cohomology $\mathbb C$ in degree zero.  The sheaf on the
smaller Borel construction is pulled back from the larger one;
the projection formula gives the asserted comparison.  This
argument applies equally to the coefficient objects in the Springer
fibre diagram and in their mixed realizations.  The comparisons are
natural in morphisms and are linear over the global equivariant
coefficient ring.  They therefore preserve the normal arrow,
the Hecke action, completion, and subsequent coefficient base change.
They are also equivariant for
$\pi_0(\widetilde{\mathcal R}_y)=\pi_0(\widetilde H_y)=A_y$.

\smallskip
\emph{The invariant coefficient homomorphism.}
In the reductive model, $\sigma_{\rm g}$ and the universal central
twist in $\mathfrak z(\mathfrak q)$ are fixed by
$\widetilde{\mathcal R}_y$.  Evaluation on this family is therefore
$A_Q$-invariant.  Via the restriction isomorphism it defines an
$A_y$-invariant homomorphism from the full connected-stabilizer
coefficient ring.  The same is true after completion and on every
central one-parameter arc.  The component action becomes linear
after this base change, so its isotypic projector commutes with the
normal map.  Verdier duality replaces $\rho$ by $\rho^\vee$ only
when rewriting the costalk coefficient as a dual stalk coefficient.
\end{proof}

The passage to the reductive factor is essential.  For example, in
$GL_3$ take $y=E_{12}$, $r=\tfrac12$,
$\sigma=\operatorname{diag}(\tfrac12,-\tfrac12,-\tfrac12)$ and
$\sigma_0=\operatorname{diag}(0,0,-\tfrac12)$.  Then
$Q=GL_2\times GL_1$, but $1+bE_{32}$ centralizes $(\sigma,y)$
without belonging to $Q$ when $b\ne0$.  It need not fix the
Lie-algebra arc pointwise.  What is invariant, and what the argument
uses, is the induced homomorphism on equivariant coefficient rings.

\begin{proposition}[Component-group descent]
\label{prop:component-descent}
Let
\[
 A_y=\pi_0\bigl(Z_{G_\sigma}(y)\bigr)
\]
and let $\rho\in\operatorname{Irr}(A_y)$ be an enhancement occurring in the
generalized Springer block.  After completion at $\chi$ and pullback along a
formal arc whose coefficient homomorphism $\psi_{\eta,y}$ is
$A_y$-invariant, the
one-variable normal map and its Smith filtration are $A_y$-equivariant.
Proposition~\ref{prop:two-component-groups} verifies this condition for
the central Langlands arcs.  Applying
\[
 \operatorname{Hom}_{A_y}(\rho,-)
\]
commutes with completion and every step of the Smith filtration
after the invariant coefficient base change, as well as with the
specialized convolution action.  The resulting costalk module is the
geometric standard module $E_{y,\sigma,r,\rho}$, while the stalk module, with
the dual coefficient convention, is the contragredient of the standard
module labelled by $\rho^\vee$ for the dual cuspidal datum.  Consequently
Theorem~\ref{thm:general-local} holds componentwise for every enhancement
$\rho$ whenever its stated hypotheses hold.
\end{proposition}

\begin{proof}
The stabilizer acts on the fibre diagram, the coefficient system, and all
adjunction morphisms.  For a Langlands family, restrict to
$\widetilde{\mathcal R}_y$ as in Proposition
\ref{prop:two-component-groups}; this does not change the connected
equivariant cohomology or its component action.  After the invariant
coefficient base change, the resulting
$A_y$-action on the two lattices is linear and commutes
with the costalk-to-stalk map, the formal parameter $t$, and
convolution.  Since $A_y$ is finite and the coefficient field has
characteristic zero, $\mathbb C[A_y]$ is semisimple.  Hence
$\operatorname{Hom}_{A_y}(\rho,-)$ is exact and is represented by a direct
summand idempotent.  For the completed canonical map
$c^\eta:\mathsf L_!^\eta\to\mathsf L_*^\eta$, write
\[
 c_\rho^\eta=\operatorname{Hom}_{A_y}(\rho,c^\eta):
 \operatorname{Hom}_{A_y}(\rho,\mathsf L_!^\eta)
 \longrightarrow
 \operatorname{Hom}_{A_y}(\rho,\mathsf L_*^\eta).
\]
Then
\[
 \operatorname{Hom}_{A_y}\!\left(
  \rho,(c^\eta)^{-1}(t^n\mathsf L_*^\eta)\right)
 =(c_\rho^\eta)^{-1}\!\left(t^n
   \operatorname{Hom}_{A_y}(\rho,\mathsf L_*^\eta)\right),
\]
and the same equality holds after reduction modulo $t$.
Thus isotypic projection commutes with the Smith filtration and its
associated graded.

The standard geometric-module identifications realize
$E_{y,\sigma,r,\rho}$ by
$\operatorname{Hom}_{A_y}(\rho,\mathsf L_!^\eta/t\mathsf L_!^\eta)$.  Apply
\cite[Theorem~C(b)]{Solleveld} with parameter $\rho$ and dualize its
finite-dimensional identity.  This identifies
$\operatorname{Hom}_{A_y}(\rho,
 \mathsf L_*^\eta/t\mathsf L_*^\eta)$ with the linear contragredient of the
standard module with label $\rho^\vee$ for the dual cuspidal coefficient
datum.
Thus the same $\rho$-projector is used on the canonical arrow.  When the
cuspidal datum is self-dual---in particular in the principal realization
used in the main theorem---both labels lie in the same Hecke block, and this
is exactly the dual-labelled contragredient target in
\eqref{eq:jantzen}.  If $\rho\simeq\rho^\vee$, it is literally the usual
self-contragredient standard module.
Since convolution commutes with the component-group action, every
identification in Theorem~\ref{thm:general-local} restricts to the indicated
summand.
\end{proof}

Let $Q$ be a Levi subgroup in a generalized Springer block and put
$S_Q=\operatorname{Sym}(\mathfrak z(\mathfrak q)^*)$.  Let $\delta_{\rm an}$
be an analytic tempered Levi module with fixed semisimple point
$s_{\delta,{\rm an}}$, and put
$s_{\delta,{\rm g}}=-s_{\delta,{\rm an}}$.  After the noncentral equivariant
parameters are specialized at $s_{\delta,{\rm g}}$, let
$\delta_{\rm g}=\operatorname{IM}_Q^*\delta_{\rm an}$ be the corresponding
geometric Levi module.  Let $\rho_Q$ be the enhancement defining
$\delta_{\rm g}$.  Denote by $\mathcal P_y^Q$ the generalized Springer fibre
for the Levi $Q$ at $y$.  Throughout the remainder of the paper,
$\dot{\mathcal L}$ denotes the restriction of the cuspidal coefficient
system $\mathcal L$ to the relevant Springer fibre (here $\mathcal P_y^Q$).
Let $\Hh_{Q,\mathrm{der}}$ denote the graded affine Hecke algebra of the
derived root system of $Q$.  Assume that multiplication gives the untwisted
central tensor decomposition
\[
 \Hh_Q\simeq \Hh_{Q,\mathrm{der}}\otimes_{\mathbb C}S_Q.
\]
This assumption holds in
the principal datum used in Theorem~\ref{thm:equal-parameter-jantzen}.

\begin{lemma}[Splitting of the Levi coefficient lattice]
\label{lem:levi-coefficient-splitting}
Under this tensor-decomposition assumption, the
$\rho_Q$-isotypic equivariant homology of the Levi Springer fibre is
canonically
\[
 \delta_{\rm g}\otimes_{\mathbb C}S_Q
\]
as a module for $\Hh_Q$.
Here $\Hh_{Q,\mathrm{der}}$ acts on the first factor and $S_Q$ acts on the
second by multiplication.
\end{lemma}

\begin{proof}
Let $H_Q=Z_{Q\times\mathbb C^\times}(\sigma,y)^\circ$.
The connected centre of $Q$ acts trivially on the Springer fibre and its
coefficient system.  Up to a finite central isogeny, which does not
change equivariant cohomology over $\mathbb C$, its Borel construction
splits off the central factor.  Thus equivariant homology is the tensor
product of the derived-parameter model and $S_Q$.

First specialize the derived parameters at the fixed Levi point
$s_{\delta,{\rm g}}$.  The full component group fixes this point and
acts trivially on $\mathfrak z(\mathfrak q)$, so it now acts
$S_Q$-linearly.  Only at this stage apply
$\operatorname{Hom}_{A_Q}(\rho_Q,-)$.
The first factor becomes $\delta_{\rm g}$ by the change-of-groups
isomorphism (B.1), Proposition~B.4(a), and Theorem~C of
\cite{Solleveld}.  This gives the displayed tensor product without
applying an isotypic projector to a semilinear module over a
maximal-torus coefficient ring.

The convolution operators for $\Hh_{Q,\mathrm{der}}$ act on the first
factor.  A central polynomial $\xi$ acts on the second factor by
multiplication by its equivariant Chern class.  Hence the factorization
is $\Hh_Q$-linear.  It is canonical after the geometric realization of
$\delta_{\rm g}$ and its enhancement have been fixed.
\end{proof}

Choose the homogeneous $\mathfrak{sl}_2$-triple $(y,h,f)$ from the graded
Slodowy setup and put
$\sigma_0=\sigma-rh$, which centralizes the triple.  Let $Q$ be the Langlands
Levi and choose the Langlands parabolic for which
$\operatorname{Re}\sigma_0$ is negative on its nilradical.  After
$Q^\circ$-conjugating the parameter if necessary, let
$S\subset Z_{Q\times\mathbb C^\times}(y)^\circ$ be a maximal torus containing
$T_{\sigma,r}$, and let $\epsilon_-$ be Solleveld's
determinant class for this Langlands parabolic.

\begin{lemma}[Completed Langlands-induction comparison]
\label{lem:two-induction-square}
The equivariant induction
map
\begin{equation}\label{eq:solleveld-completed-induction}
 \Hh_G\otimes_{\Hh_Q}
 H^S_\bullet(\mathcal P_y^Q,\dot{\mathcal L})
 \longrightarrow H^S_\bullet(\mathcal P_y,\dot{\mathcal L})
\end{equation}
is Hecke- and $H_S^*(\mathrm{pt})$-linear, injective, and has cokernel
annihilated by $\epsilon_-$.  Moreover $\epsilon_-(\sigma,r)\ne0$.  Hence
\eqref{eq:solleveld-completed-induction} becomes an isomorphism after
completion at $(\sigma,r)$, compatibly with the transpose correspondence,
Verdier duality, and the standard Hecke anti-involution.

The class $\epsilon_-$ is not the Euler class of the attractive normal slice
defined by the opposite parabolic.  The latter vanishes on every nonzero
normal direction whose equivariant character vanishes at $(\sigma,r)$.  This
vanishing is precisely part of the Jantzen degeneration, and the local Euler
class is never inverted in this comparison.
\end{lemma}

\begin{proof}
The Langlands-chamber sign condition is exactly the first case of
\cite[Lemma~B.3]{Solleveld}, which first proves
$\epsilon_-(\sigma,r)\ne0$.  The hypotheses of
\cite[Theorem~B.2(a)]{Solleveld} are therefore satisfied; that theorem gives
the injection and says that its image contains $\epsilon_-$ times the target.
Thus $\epsilon_-$ is a unit in the completed local
coefficient ring, and the injection becomes an isomorphism.

The map is induced by a proper Gysin correspondence.  Transposing that
correspondence and applying Verdier duality replaces
$\dot{\mathcal L}$ by $\dot{\mathcal L}^{\vee}$ and gives the standard Hecke
anti-involution; the projection formula makes the two maps adjoint.  The
present lemma is deliberately proved over the connected stabilizer, before
any isotypic projection.  In Step~1 of
Proposition~\ref{prop:general-coefficient-arc}, the same correspondence is
used with the full reductive scaling stabilizer
$C_{Q,y}=\widetilde{\mathcal R}_y$ of Proposition
\ref{prop:two-component-groups}.  Restriction identifies its target
cohomology with that for the full point stabilizer, without asserting
that the latter preserves the induction correspondence.  Step~2 makes the component
action linear over the completed central coefficient ring and applies the
isotypic projector.  No component-group compatibility is needed in the
argument above.

On an irreducible $\mathfrak{sl}_2$-summand of highest weight
$m$ and $\sigma_0$-weight $a$, its lowest-weight line contributes the
equivariant Euler factor $a-r(m+2)$ to the attractive normal slice.  A vector in
$\mathfrak g^f\cap\mathfrak g_{\sigma,2r}$ satisfies
$a=r(m+2)$, so this factor is zero.  By contrast, Solleveld's
$\epsilon_-$ is computed with the Langlands parabolic and is nonzero by the
preceding chamber argument.  Thus no vanishing local Euler factor is divided
out in the completed induction comparison.
\end{proof}

Retain $y$ and $r$ from the fixed-slice construction.  Let $\Hh_G$ be the
graded affine Hecke algebra in the geometric realization attached to a generalized
Springer cuspidal datum, let $Q$ be the Levi subgroup used for parabolic
induction, and let $\delta_{\rm an}$ be the analytic tempered inducing
module with fixed semisimple point $s_{\delta,{\rm an}}$.  Put
\[
 \delta=\delta_{\rm g}=\operatorname{IM}_Q^*\delta_{\rm an},
 \qquad s_{\delta,{\rm g}}=-s_{\delta,{\rm an}}.
\]
Let $\nu=\nu_{\rm g}\in\mathfrak z(\mathfrak q)$ be a real geometric
central twist and put $\sigma=\sigma_{\rm g}=s_{\delta,{\rm g}}+\nu$.
Assume the untwisted central tensor splitting of
Lemma~\ref{lem:levi-coefficient-splitting}, and put
\[
 S_Q=\operatorname{Sym}(\mathfrak z(\mathfrak q)^*)
\]
and let $\widehat S_{Q,\nu}$ be its completion at the maximal ideal defined by
$\nu$.  Put $\delta_{S_Q}=\delta\otimes_{\mathbb C}S_Q$; this is the
universal unramified-twist family of $\delta$.  Choose a maximal torus
$S\subset Z_{Q\times\mathbb C^\times}(y)^\circ$ containing
$T_{\sigma,r}$, and let
\[
 \kappa_\delta:H_S^*(\mathrm{pt})\longrightarrow\widehat S_{Q,\nu}
\]
fix the remaining semisimple coordinates at $s_{\delta,{\rm g}}$ and retain
the universal $\mathfrak z(\mathfrak q)$-coordinate.  Let
$A_Q=\pi_0(Z_Q(s_{\delta,{\rm g}},y))$ and let $\rho_Q$ be the enhancement
defining $\delta$.  Central twists do not change this group.
To make the completions explicit, set
\[
 \mathcal X_{Q,S_Q}=\Hh_G\otimes_{\Hh_Q}\delta_{S_Q},\qquad
 \widehat{\mathcal X}_{Q,\nu}
 =\widehat S_{Q,\nu}\otimes_{S_Q}\mathcal X_{Q,S_Q},
\]
and define the geometric lattice in the following order.
Put $H_Q=C_{Q,y}^\circ$, the connected stabilizer of $y$ in
$Z_Q(\sigma)\times\mathbb C^\times$, and let
$\kappa_\delta^{H_Q}:H_{H_Q}^\bullet(\mathrm{pt})\to\widehat S_{Q,\nu}$
be evaluation on the universal central twist.  Set
\[
 \widehat{\mathcal C}_{y,\rho_Q,\nu}
 =\operatorname{Hom}_{A_Q}\!\left(
 \rho_Q,\,
 \widehat S_{Q,\nu}\otimes_{H_{H_Q}^\bullet(\mathrm{pt}),
                            \kappa_\delta^{H_Q}}
 H_\bullet^{H_Q}(\mathcal P_y,\dot{\mathcal L})\right).
\]
The component action is linear after base change, since it fixes the
central coordinate.  The maximal torus $S$ is used only for faithful
flatness in the connected-equivariant comparison; no projector is
applied to an arbitrary $S$-equivariant module before base change.
Let $\eta\in X_*(Z(Q)^\circ)\otimes_{\mathbb Z}\mathbb Q$ be regular and
rescale it positively to be integral.  If $\eta$ lies in the contracting
chamber and the graded Slodowy slice has positive dimension, let $A_\eta$ be the effective
one-dimensional torus that it generates.  Choose a character $\xi_+$ on the
positive ray of $A_\eta$, and put
$z_\eta=c_1(\xi_+)$ and
$c_\eta=\langle\xi_+,\eta\rangle>0$.
Here and below, a primitive integral normalization of a map over $\A$ means
that its matrix entries have no common factor $t$, or equivalently that its
image is not contained in $t$ times the target lattice.  For a self-dual
cuspidal datum the maps below lie in one block; in general they pair a block
with its linear-dual block.

\begin{proposition}[Coefficient arc and completed costalk]
\label{prop:general-coefficient-arc}
The universal comparison has the following three consequences.
For the comparison of maps in \textup{(2)}, assume that the perverse IC
summands of the localized complex have pure polarizable realizations.
This hypothesis is verified for the principal datum in the proof of
Theorem~\ref{thm:equal-parameter-jantzen}.
\begin{enumerate}
\item There is an isomorphism of $\widehat S_{Q,\nu}$-lattices
\begin{equation}\label{eq:universal-coefficient-family}
 \Phi_!: \widehat{\mathcal X}_{Q,\nu}
 \xrightarrow{\ \sim\ }
 \widehat{\mathcal C}_{y,\rho_Q,\nu}.
\end{equation}
The target is the $\rho_Q$-isotypic completed normal costalk of
$K_{\sigma,r}$, with the regrading \eqref{eq:normal-point-objects}.
There is a corresponding isomorphism $\Phi_*$ with the completed
normal stalk for the dual cuspidal coefficient datum.

\item Pullback by
\[
 f\longmapsto f(\nu+t\eta)
\]
gives the algebraic Jantzen family.  Under $\Phi_!$ and $\Phi_*$, its
primitive standard-to-contragredient map and the normal map
\eqref{eq:normal-fibre-map} differ by a unit of $\A$.

\item Suppose in addition that $\eta$ lies in the positive contracting
chamber and the slice has positive dimension.  If $t>0$ is oriented in the direction $\eta$, then
\begin{equation}\label{eq:positive-coefficient-arc}
 z_\eta\longmapsto c_\eta t,
 \qquad c_\eta=\langle\xi_+,\eta\rangle>0.
\end{equation}
This orientation-preserving change of coordinate preserves Smith exponents
and multiplies every order-$n$ leading coefficient by the positive scalar
$c_\eta^n$.
\end{enumerate}
\end{proposition}

In part~(1), completed localization and proper base change give the stated
costalk identification.  Transposition and Verdier duality give $\Phi_*$.
Both isomorphisms are constructed over the universal completed coefficient
ring, before a deformation direction is chosen.  In part~(2), the linear
contragredient pairing corresponds to the geometric intersection pairing.
Part~(3) records the orientation needed later to compare leading terms, not
merely Smith exponents.

\begin{proof}
\smallskip
\noindent\emph{Step 1: completion over the connected stabilizer.}
Lemma~\ref{lem:two-induction-square} makes Solleveld's parabolic-induction map
an isomorphism after completion at $(\sigma,r)$.  Put
\[
 C_{Q,y}=Z_{Z_Q(\sigma)\times\mathbb C^\times}(y),\qquad
 C_{Q,y}^{\circ}=Z_{Z_Q(\sigma)\times\mathbb C^\times}(y)^{\circ}.
\]
The component group of $C_{Q,y}$ is canonically $A_Q$, hence also the enhancement group
of Proposition~\ref{prop:two-component-groups}.  That proposition also
identifies $C_{Q,y}$ with a Levi factor of the full scaling stabilizer
$\widetilde H_y$.  Its restriction isomorphisms identify the target
coefficient ring and fibre homology with those for
$\widetilde H_y^\circ$, naturally for the Hecke action.  Thus the
comparison may be proved with $C_{Q,y}$-equivariance throughout.
The use of a maximal torus
$S\subset C_{Q,y}^{\circ}$ is only a faithfully flat change of equivariance.
Indeed, Solleveld's change-of-groups isomorphism
\cite[equation~(B.1)]{Solleveld}
identifies the $S$-equivariant induction map with base change of the map
defined by the same proper correspondence in
$C_{Q,y}^{\circ}$-equivariant homology.  Since
$H_S^*(\mathrm{pt})$ is finite free, hence faithfully flat, over
$H_{C_{Q,y}^{\circ}}^*(\mathrm{pt})$, Lemma~B.3 and
Theorem~B.2(a) of~\cite{Solleveld} show that this connected-equivariant map
becomes an isomorphism after completion at $(\sigma,r)$.

\smallskip
\noindent\emph{Step 2: the enhancement and the universal central parameter.}
The full reductive group $C_{Q,y}$, including its components, acts on the
correspondence.  Before the final base
change its action on the coefficient ring is semilinear, so one must not
apply an $A_Q$-projector directly to an arbitrarily chosen
$S$-equivariant model.  The point $(\sigma,r)$ is fixed and its completion
ideal is stable.  Conjugation by $Q$ fixes $\mathfrak z(\mathfrak q)$
pointwise; after $\kappa_\delta$ specializes the remaining coordinates and
retains only $\widehat S_{Q,\nu}$, the component action is therefore
$\widehat S_{Q,\nu}$-linear.  Since $\mathbb C[A_Q]$ is semisimple, applying
$\operatorname{Hom}_{A_Q}(\rho_Q,-)$ is exact and commutes with completion.
Now use Lemma~\ref{lem:levi-coefficient-splitting} to identify the source,
before completion, with
$\Hh_G\otimes_{\Hh_Q}(\delta\otimes S_Q)$.  The central generator
$\xi\in\mathfrak z(\mathfrak q)^*$ acts geometrically by its equivariant
first Chern class, so pullback along $f\mapsto f(\nu+t\eta)$ changes its action
to
\[
 \xi(\nu)+t\langle\xi,\eta\rangle.
\]
This is the universal unramified-twist family and proves
\eqref{eq:universal-coefficient-family}.

\smallskip
\noindent\emph{Step 3: orientation of the one-parameter coordinate.}
For the positive-coordinate assertion, let $\eta_0$ be the primitive
cocharacter generating the positive ray that defines $A_\eta$.  The chosen
integral cocharacter has the form $m\eta_0$ with $m>0$.  The positive normal
weights identify the ample orbifold class on the projectivized punctured
slice with the positive character ray, and the preceding coefficient formula
gives
\[
 z_\eta\longmapsto
 t\langle\xi_+,m\eta_0\rangle=c_\eta t,
 \qquad c_\eta>0.
\]
A different stack normalization multiplies the positive generator by a
positive rational number.  For any coefficient matrix $C$, pulling back a
leading term $z_\eta^nC$ therefore
gives $t^n(c_\eta^nC)$.  The positive factor $c_\eta^n$ does not change the
Smith exponent.  A point slice needs no equivariant coordinate: its
normal map is the identity.

\smallskip
\noindent\emph{Step 4: identification of the costalk and stalk lattices.}
Let
\[
 j_y:\mathcal P_y^{\exp(\mathbb C(\sigma,r))}
       \hookrightarrow\mathcal P_y,
 \qquad
 R_S=\widehat{H_S^\bullet(\mathrm{pt})}_{(\sigma,r)}.
\]
Proposition~A.2 of~\cite{Solleveld} is stated for an affine $S$-variety.
The same completed-localization assertion holds for the present Springer
fibre: intersecting with an $S$-stable Bruhat stratification of the ambient
flag variety gives a finite stratification by affine $S$-varieties, so one
applies that proposition to the strata and then uses the localization long
exact sequences in equivariant Borel--Moore homology.  Consequently, since
$(\sigma,r)\in\operatorname{Lie}S$,
completed localization gives an isomorphism
\[
 R_S\otimes_{H_S^\bullet(\mathrm{pt})}
 H_\bullet^S\!\left(
   \mathcal P_y^{\exp(\mathbb C(\sigma,r))},j_y^*\dot{\mathcal L}
 \right)
 \xrightarrow{\sim}
 R_S\otimes_{H_S^\bullet(\mathrm{pt})}
 H_\bullet^S(\mathcal P_y,\dot{\mathcal L}).
\]
Equation~(3.16) in the proof of \cite[Proposition~3.5(a)]{Solleveld} is an
identity in the equivariant derived category.  Restricting it to $S$, taking
equivariant hypercohomology, and retaining the shifts in the computation
immediately following equations~(3.16)--(3.19) identifies the source with
the completed equivariant point costalk of $K_{\sigma,r}$.
The comparison descends from $S$ to $H_Q$ by the faithfully flat
change of coefficients in Step~1.  Restriction from
$\widetilde H_y^\circ$ to $H_Q$ then identifies it with the
connected-stabilizer model of \eqref{eq:normal-arc-lattices}.
These identifications commute with the central coefficient
homomorphism and the component projector by Proposition
\ref{prop:two-component-groups}.  The normal costalk differs by the common shift in
\eqref{eq:normal-point-objects}; this regrading leaves the ungraded
Hecke module and its integral lattice unchanged.  Theorem~2.5
of~\cite{Solleveld} is used here
only to identify the completed convolution action with the completed Hecke
action.

Apply the same argument to the dual coefficient system and use
\cite[Proposition~3.5(b)]{Solleveld}.  Because the completed costalk is a
finite free induced family, dualization is exact.  The inverse transpose of
the completed induction isomorphism is therefore the asserted stalk map
$\Phi_*$.  Transposition gives the standard Hecke anti-involution, and the
projection formula identifies the induced pairing with the Levi pairing.

\smallskip
\noindent\emph{Step 5: comparison of the two primitive maps.}
Pull both completed lattices back to the chosen arc.  The geometric
map is \eqref{eq:normal-fibre-map}, Hecke-linear by naturality of orbit
restriction (Lemma~\ref{lem:residual-grading}).  Its generic fibre is
invertible: the normal weights lie outside the Levi root system, so
relative regularity makes their values on $\eta$ nonzero.  Localization
on the slice leaves only the vertex.
Lemma~\ref{lem:normal-purity-base-change} applies also to this possibly
non-dominant cocharacter: purity was established using a positive
cocharacter, and freeness over the full stabilizer coefficient ring
is independent of the chosen arc.

The normal map is already primitive.  The IC summand carrying the
head restricts to a point-supported object on the slice; its map is
the identity, also after enhancement projection.  Normalize the
algebraic map primitively and use $\Phi_!$ and $\Phi_*$
to regard both as maps between the same lattices.

\smallskip
\noindent\emph{Step 5a: generic proportionality.}
Let
$\overline\K$ be an algebraic closure of $\K$, let $W_Q$ be the Weyl group of
$Q$, and let $W^Q$ be the set of minimal representatives for $W/W_Q$.
For a field extension $E/\mathbb C$, write
$\Hh_{G,E}=\Hh_G\otimes_{\mathbb C}E$, and use the analogous notation for
$Q$ and for modules.  Let $X_\K$ be the generic fibre of the induced family,
and let $\Omega_\delta$ be the finite set of generalized
$\operatorname{Sym}(\mathfrak t^*)$-weights of $\delta$.  By the PBW theorem,
the generalized weights of the induced module lie in
\[
 \bigcup_{w\in W^Q}w\bigl(\Omega_\delta+\nu+t\eta\bigr).
\]
These subsets are pairwise disjoint over $\overline\K$.  Indeed, equality
between a weight in the $w_1$-subset and one in the $w_2$-subset would, after
comparing the coefficient of $t$, give $w_1\eta=w_2\eta$.  Relative
regularity says $\operatorname{Stab}_W(\eta)=W_Q$, so
$w_2^{-1}w_1\in W_Q$.  Hence $w_1W_Q=w_2W_Q$, and uniqueness of the minimal
representative of a coset gives $w_1=w_2$.

Every endomorphism of the induced $\Hh_{G,\overline\K}$-module consequently
preserves the inducing generalized-weight block
$1\otimes\delta_{\overline\K}$.  The identity-coset block is stable under
$\Hh_{Q,\overline\K}$: its polynomial weights are
$\Omega_\delta+\nu+t\eta$, and the finite Weyl group $W_Q$ permutes these
weights without leaving that block.  The restriction of an endomorphism to
this block is therefore $\Hh_{Q,\overline\K}$-linear and hence scalar by
Schur's lemma.  Here no descent issue is hidden: the finite-dimensional
$\Hh_Q$-module $\delta$ is irreducible over the algebraically closed field
$\mathbb C$, so Burnside's theorem identifies its image algebra with
$\operatorname{End}_{\mathbb C}(\delta)$.  Thus every scalar extension of
$\delta$ is absolutely irreducible and has scalar endomorphism ring.  Since this
block generates the induced module, the endomorphism itself is scalar.
Descent gives
\[
 \operatorname{End}_{\Hh_{G,\K}}(X_\K)=\K.
\]
For the Weyl element defining the standard-to-contragredient operator,
\cite[Proposition~4.3]{SolleveldInduced} shows that the normalized
intertwiner is rational in the twist parameter and is regular and invertible
on a dense Zariski-open subset.  In the product formula in its proof, the
possible poles of the operator and its inverse lie on finitely many affine
root hyperplanes associated with roots outside $R_Q$.  Their linear parts
pair nontrivially with $\eta$ by relative regularity, so every defining
affine-linear function restricts nontrivially to the line $\nu+t\eta$.  Thus
$I_{\rm alg}\otimes_\A\K$ is an isomorphism.  Composition
with its inverse identifies the generic Hom-space with the scalar
endomorphism ring above.

It follows that the transported geometric map
$I_{\rm geo}$ and normalized algebraic map $I_{\rm alg}$ are proportional:
\[
 I_{\rm geo}=aI_{\rm alg}\qquad(a\in\K^\times).
\]
\smallskip
\noindent\emph{Step 5b: comparison of contents.}
Use $\Phi_!$ and $\Phi_*$ to regard both maps as nonzero maps between the
same finite free $\A$-lattices $M$ and $N$.  For
\[
 0\ne T\in\operatorname{Hom}_{\K}
   (M\otimes_\A\K,N\otimes_\A\K),
\]
define its content to be the fractional $\A$-ideal
\[
 \operatorname{cont}(T)=
 \big\langle\lambda(Tm):m\in M,\ \lambda\in N^\vee\big\rangle_\A
 \subset\K.
\]
Here each $\lambda\in N^\vee$ is extended $\K$-linearly to
$N\otimes_\A\K$.  Equivalently, $\operatorname{cont}(T)$ is generated by
the entries of the matrix of $T$ in any $\A$-bases.  The intrinsic definition
shows that it is basis independent, and
\[
 \operatorname{cont}(bT)=b\operatorname{cont}(T)
 \qquad(b\in\K^\times).
\]
If $T:M\to N$ is integral, then $\operatorname{cont}(T)\subset\A$, and
$T$ is primitive precisely when $\operatorname{cont}(T)=\A$: equivalently,
at least one matrix entry is a unit, or $T(M)\not\subset tN$.

Both $I_{\rm geo}$ and $I_{\rm alg}$ are primitive integral maps, so their
contents are $\A$.  By Step~5a,
$I_{\rm geo}=aI_{\rm alg}$ with $a\in\K^\times$.  Hence
\[
 \A=\operatorname{cont}(I_{\rm geo})
 =a\operatorname{cont}(I_{\rm alg})=a\A.
\]
Since $\A$ is a discrete valuation ring, this equality is equivalent to
$v_t(a)=0$, and therefore $a\in\A^\times$.  Thus the two integral maps
differ by a unit.  Such a unit changes neither their Smith exponents nor the
induced filtration.
\end{proof}

\subsection{Proof of the main theorem}

\begin{theorem}[Jantzen theorem]
\label{thm:equal-parameter-jantzen}
Let $\Hh=\Hh(R,1)$ be an equal-parameter graded affine Hecke algebra, and use
the standard principal Springer parametrization fixed in the Introduction.
Let $X(M,\delta,\nu)$ be a Langlands standard module with real central
character, and deform it along a rational cocharacter $\eta$ in the
interior of its dominant Langlands chamber, positively rescaled to be
integral.  Equivalently, $\eta$ is dominant and regular relative to $M$ in
the sense fixed in the Introduction.  Then every
successive quotient of its Jantzen filtration is semisimple.  For a simple
module indexed by an enhanced orbit $(\mathcal O_\delta,\mathcal L_\delta)$,
write $\gamma$ for the enhanced orbit indexing the standard module.  Then
its layer multiplicities are obtained from the Lefschetz strings in the
corresponding local IC costalk.  With the head in layer zero,
\[
 \sum_{n\geq 0}[\operatorname{gr}_J^nX_\gamma:L_\delta]v^n
 =v^{d_\delta-d_\gamma}\widehat P_{\gamma\delta}(v^{-1}),
\]
where $\widehat P_{\gamma\delta}$ is defined in
\eqref{eq:full-ic-polynomial}.  Thus the statement retains the complete
cohomological grading and the dual-label convention.
\end{theorem}

\begin{proof}
\smallskip
\noindent\emph{Step 1: pass to the geometric convention and fix the chamber.}
Let $M$ be the Langlands Levi.  By the
analytic/geometric transport of
Lemma~\ref{lem:analytic-geometric-transport}, it is enough to prove
the assertion for the geometric data
\[
 (\delta_{\rm g},\nu_{\rm g},\eta_{\rm g})
 =\bigl((\operatorname{IM}_M)^*\delta,-\nu,-\eta\bigr)
\]
and then apply $\operatorname{IM}^*$.  Let
$(y,\sigma,r,\rho)$ be the corresponding enhanced geometric parameter.
Choose the homogeneous triple $(y,h,f)$ from the graded Slodowy construction
and put $\sigma_0=\sigma-rh$.  The direction $\eta_{\rm g}$ is positive on the
opposite attractive nilradical used in the fixed-slice construction, while
the triple-centralizing element $\operatorname{Re}\sigma_0$ has the sign
required in Lemma~\ref{lem:two-induction-square}.  For the remainder of the proof,
rename the geometric triple $(\delta_{\rm g},\nu_{\rm g},\eta_{\rm g})$ as
$(\delta,\nu,\eta)$.
Before fixing representatives, conjugate the entire enhanced parameter so
that $\sigma_0$ lies in the chosen Cartan and
$\operatorname{Re}\sigma_0$ is negative on the nilradical of the Langlands
parabolic.  Solleveld's Proposition~B.4(b), together with the remark
immediately following it, shows that every Langlands parameter admits this
normalization.  Accordingly, put
$Q=M=Z_G(\operatorname{Re}\sigma_0)$
\cite[Proposition~B.4(b) and the following remark]{Solleveld}.

\smallskip
\noindent\emph{Step 2: identify the algebraic and geometric lattices.}
Apply the constructions of this section to the principal cuspidal datum
$(T,\{0\},\mathbf1)$, which is self-dual.  All subsequent constructions are
taken in this fixed principal realization.
This geometric standard is induced from the appropriate opposite parabolic
and is realized by the $\rho$-isotypic costalk; its dual-labelled
contragredient target is the $\rho$-isotypic stalk, as in
Proposition~\ref{prop:component-descent}.
Proposition~\ref{prop:general-coefficient-arc}
identifies these two algebraic lattices with the completed normal
costalk and stalk families, using \eqref{eq:normal-point-objects}.
They are finite free over the one-parameter coefficient ring.  The
comparison of their maps will be made after the purity check in Step~3.  Proposition
\ref{prop:two-component-groups} identifies the enhancement in the Levi datum
with the final-orbit enhancement, and Proposition
\ref{prop:component-descent} applies the common isotypic projector.

\smallskip
\noindent\emph{Step 3: verify purity of the localized complex.}
Purity must be checked for the localized complex itself because the fixed
localization diagram is not Cartesian.  For the principal datum the global
incidence variety $\dot{\mathfrak g}$ is a vector bundle over $G/B$, hence is
smooth, and its coefficient system is constant.  Its torus fixed locus
$\dot{\mathfrak g}^{\sigma,r}$, where the superscript denotes the fixed locus
of $\exp(\mathbb C(\sigma,r))$, is smooth.  Solleveld's definition and
fixed-point description, equations~(2.11)--(2.14) of~\cite{Solleveld}, use
the projection
\[
 \operatorname{pr}_1:\dot{\mathfrak g}^{\sigma,r}
 \longrightarrow\mathfrak g_{\sigma,2r}
\]
and give
\[
 K_{\sigma,r}=(\operatorname{pr}_1)_!
 \IC\bigl(\dot{\mathfrak g}^{\sigma,r},\mathbf1\bigr),
\]
with componentwise perverse shifts and the fixed Tate twists in the
normalization of the geometric realization.  Each fixed component is smooth,
so its IC
complex is its shifted constant complex and is pure.  After spreading out,
proper direct-image purity and the decomposition theorem
\cite[Stabilit\'es~5.1.14 and Th\'eor\`eme~5.4.5]{BBD} make every perverse summand of
$K_{\sigma,r}$ pure
with its prescribed shift.  Equivalently, the decomposition theorem writes
$K_{\sigma,r}$ as a direct sum of shifts of pure perverse IC complexes.
The cohomological shifts and Tate twists entering the Grojnowski regrading
change the numerical degrees and weights by the prescribed amounts, but they
preserve purity.  Consequently, after the normalization fixed in
Section~\ref{sec:local-smith}, every perverse IC summand satisfies the purity
hypothesis of Theorem~\ref{thm:general-local}.  That theorem may therefore be
applied separately to every summand of $K_{\sigma,r}$.

\smallskip
\noindent\emph{Step 4: compare the normal Smith filtration with the
Jantzen filtration.}
The direction $\eta$ is positive on the attractive parabolic opposite
to the Langlands parabolic.  Proposition
\ref{prop:general-graded-slodowy} supplies the positively contracted
slice $V_{\sigma,r}(y)$.

For each perverse IC summand, take the normal orbit-restriction map
\eqref{eq:normal-fibre-map}.  Transverse base change identifies it with
the vertex map on $V_{\sigma,r}(y)$.
Lemma~\ref{lem:normal-purity-base-change} justifies ordinary base change
from the connected stabilizer to the contracting direction.  For a
positive-dimensional slice, its completed equivariant coordinate is
$z_\eta=c_\eta t$, with $c_\eta>0$; on a point slice the normal map is
the identity.  The local Smith theorem
\ref{thm:general-local} therefore gives a convolution-stable filtration
$F^\bullet$ with semisimple layers.
The coefficient homomorphism is component-invariant by
Proposition~\ref{prop:two-component-groups}; apply the exact
$\rho$-isotypic projector after the coefficient base change.

It remains to identify the map, not merely its source and target.
The normal map is generically invertible and primitive: on the
head-labelled, point-supported slice summand it is the identity.
It is Hecke-linear by functoriality of orbit restriction.
Proposition~\ref{prop:general-coefficient-arc}(2), now applicable by
Step~3, compares it with the primitive algebraic intertwiner.
The scalar generic endomorphism ring makes the two maps proportional; comparing
their contents over $\A$ makes the scalar a unit.
Consequently the vertical identifications in
\eqref{eq:master-comparison} satisfy
\[
 \Phi_!\!\left(I_t^{-1}
       \bigl(t^nX_{\gamma^\dagger,\A}^{\vee}\bigr)\right)
 =(c_\gamma^N)^{-1}
       \bigl(t^n\mathsf L_{*,\gamma}^{\eta}\bigr)
 \qquad(n\geq0).
\]
Reducing modulo $t$ identifies $J^\bullet$ with $F^\bullet$.
No ambient point self-intersection factor enters this argument:
the normal arrow is defined along the orbit before taking its fibre.

\smallskip
\noindent\emph{Step 5: compute the absolute layer indices and return to the
analytic convention.}
Corollary~\ref{cor:absolute-layer-normalization} identifies layer $n$ with
\[
 \operatorname{Hom}_{A_\gamma}\!\left(
  \rho_\gamma,
  \mathcal H^{n+d_\gamma}i_\gamma^!\IC_\delta
 \right).
\]
Verdier duality turns this into
\[
 \operatorname{Hom}_{A_\gamma}\!\left(
  \rho_\gamma^\vee,
  \mathcal H^{-n-d_\gamma}i_\gamma^*\IC_{\delta^\dagger}
 \right).
\]
Writing $-n-d_\gamma=-d_\delta+j$ gives
$n=d_\delta-d_\gamma-j$.  If $m_j$ denotes the dimension of the displayed
dual stalk coefficient in degree $-d_\delta+j$, then the perverse support
bound $\mathcal H^ki_\gamma^*\IC_{\delta^\dagger}=0$ for
$k>-d_\gamma$ implies
$m_j=0$ when $j>d_\delta-d_\gamma$.  Thus every nonzero term below has
$n\geq0$, and
\[
 \sum_{n\geq0}[\operatorname{gr}_J^nX_\gamma:L_\delta]v^n
 =\sum_jm_jv^{d_\delta-d_\gamma-j}
 =v^{d_\delta-d_\gamma}\widehat P_{\gamma\delta}(v^{-1}).
\]
This proves~\eqref{eq:layer-polynomial}.  A term indexed by $j$
contributes $v^{d_\delta-d_\gamma-j}$, and normalization on the generic
Langlands quotient fixes the common shift.  Finally, under the signed
base-ring identification above,
$\operatorname{IM}^*$ is exact and $\mathbb C[[t]]$-linear, so it transports
these layers, their semisimplicity, and their enhanced-orbit multiplicities
back to the analytic standard module $X(M,\delta,\nu)$.
\end{proof}

\section{Dominance and generalizations}
\label{sec:dominance}

\subsection{Unequal parameters arising from generalized Springer theory}

There are two different possible meanings of an unequal-parameter
generalization.  An arbitrary $W$-invariant function
$\alpha\mapsto c_\alpha$ in the presentation of the Introduction need not
have a known sheaf-theoretic realization.  The discussion here concerns the
\emph{geometric} graded affine Hecke algebras
\[
 \mathbb H(G,M_{\rm cusp},q\mathcal E,r)
\]
attached to a generalized-Springer cuspidal datum.  Their root parameters
$c_\alpha r$ are determined by the datum and may be unequal; the algebra may
also carry the finite extension and cocycle described in
\cite[Section~2]{AMS}.  Standard modules, simple modules, and their
Kazhdan--Lusztig multiplicities are realized by the corresponding equivariant
constructible complexes; see \cite[Sections~2--5]{Solleveld}.

Most of the proof was formulated in this geometric generality.  None of the
following ingredients uses equality of the parameters:
\begin{enumerate}[label=\textup{(G\arabic*)},leftmargin=2.8em]
\item the contracting graded Slodowy slice and the completed fixed-slice
comparison, Propositions~\ref{prop:general-graded-slodowy} and
\ref{prop:general-fixed-slice};
\item the cone calculation with finite-stabilizer coefficients, hard
Lefschetz, and the Smith--primitive comparison of
Section~\ref{sec:local-smith};
\item the residual grading on the specialized convolution algebra and the
identity
$\operatorname{rad}(\mathcal E_\chi^{\rm res})
=(\mathcal E_\chi^{\rm res})_{>0}$;
\item exact projection to every component-group enhancement, Propositions
\ref{prop:two-component-groups} and~\ref{prop:component-descent}; and
\item the completed parabolic-induction comparison, the generic
proportionality argument, and the positive coordinate
$z_\eta\mapsto c_\eta t$ in Proposition
\ref{prop:general-coefficient-arc}.
\end{enumerate}
Thus no new local Smith--Lefschetz calculation or semisimplicity argument is
required for unequal parameters.  The remaining issue is to make
the global algebraic standard family and its duality agree, in every
geometric block, with the two completed lattices to which the local theorem
applies.

The following proposition records the precise conditional reduction.

Let $\mathbb H(G,M_{\rm cusp},q\mathcal E,r)$ be a geometric graded affine
Hecke algebra, let $X_\gamma$ be a Langlands standard module with real central
character, and deform its central parameter along a rational cocharacter
$\eta$ in the interior of the dominant Langlands chamber (equivalently, a
dominant direction regular relative to the Langlands Levi).  Let
$K_{\sigma,r}=(\pi_{\sigma,r})_!F_{\sigma,r}$ be the corresponding localized
generalized-Springer complex.  Use the normal lattices \eqref{eq:normal-arc-lattices}, with the
connected point-stabilizer coefficient ring, and then take the
$\rho_\gamma$-isotypic part.  In particular the notation
$\mathsf L_{!,*}^{\eta}$ below means
\[
 \operatorname{Hom}_{A_\gamma}\!\left(\rho_\gamma,\,
 \mathbb C[[t]]\otimes_{\widehat R_\gamma,\psi_\eta}
 H_{C_{\chi,\gamma}^\circ}^{\bullet}
       (\mathsf N_\gamma^{!,*}K_{\sigma,r})^\wedge_\chi\right),
\]
where $\widehat R_\gamma
=H_{C_{\chi,\gamma}^\circ}^{\bullet}(\mathrm{pt})^\wedge_\chi$.
The projector is applied after the invariant coefficient base change.

\paragraph{Block hypotheses.}
The reduction uses the following four inputs.
\begin{enumerate}[leftmargin=3.2em]
\item[\textup{(U1)}] For the Langlands Levi $Q$, the universal central-twist lattice admits
the tensor splitting of Lemma~\ref{lem:levi-coefficient-splitting}, or a
crossed-product analogue with the same consequence.  In particular,
Proposition~\ref{prop:general-coefficient-arc} identifies the completed
algebraic standard lattice with the appropriate isotypic equivariant
costalk.
\item[\textup{(U2)}] Verdier duality for $q\mathcal E$ is identified with the algebraic
transpose anti-involution in the block under consideration.  Equivalently,
the completed stalk is the integral contragredient target of $X_\gamma$, and
the primitive algebraic intertwiner corresponds, up to an element of
$\mathbb C[[t]]^\times$, to the normal orbit-restriction map.
\item[\textup{(U3a)}] The cuspidal coefficient $q\mathcal E$ admits a pure
polarizable mixed realization, and the localized generalized-Springer
complex is formed in that category with the componentwise perverse shifts and
Tate twists prescribed by the geometric Hecke normalization.  Every
perverse IC summand used below is therefore pure of its prescribed weight
and polarizable, and the component-group isotypic projectors preserve these
properties.
\item[\textup{(U3b)}] The two modules $\mathsf L_!^\eta$ and
$\mathsf L_*^\eta$ defined above are finite-rank free
$\mathbb C[[t]]$-modules.
\end{enumerate}

Condition (U2) may identify a self-dual block with itself, or it may pair a
block with its linear-dual block.  Condition (U3b) says precisely that the
two arcwise lattices have no $t$-torsion, so their canonical map has an
ordinary Smith normal form.  This follows from (U1) and (U2) through
Proposition~\ref{prop:general-coefficient-arc}; it is listed separately to
make the freeness input to the local theorem explicit.

\begin{proposition}[Reduction of the geometric unequal-parameter case]
\label{prop:unequal-parameter-reduction}
Assume \textup{(U1)}--\textup{(U3b)}.  Then every Jantzen layer of
$X_\gamma$ is semisimple.  If $L_\delta$ is the
simple module indexed by the enhanced orbit $\delta$, then, with the
Langlands quotient in layer zero,
\[
 [\operatorname{gr}_J^nX_\gamma:L_\delta]
 =\dim\operatorname{Hom}_{A_\gamma}\!\left(
   \rho_\gamma,
   \mathcal H^{n+d_\gamma}i_\gamma^!\IC_\delta
  \right).
\]
Equivalently, the layer polynomial is the fully graded local IC polynomial
appearing in Theorem~\ref{thm:equal-parameter-jantzen}.
\end{proposition}

\begin{proof}
By (U1) and (U2), the algebraic Jantzen filtration is the Smith filtration of
the completed costalk-to-stalk map: multiplying that map by a unit of
$\mathbb C[[t]]$ does not change either filtration.  The positive-coordinate
part of Proposition~\ref{prop:general-coefficient-arc} identifies the
equivariant variable with a positive nonzero multiple of $t$, so the Smith
exponents are unchanged.  Assumptions (U3a) and (U3b) supply the remaining
hypotheses of Theorem~\ref{thm:general-local}: (U3a) supplies purity and polarizability
for hard Lefschetz, while (U3b) supplies the free lattices required by Smith
theory.  That theorem gives semisimplicity and the displayed multiplicity
formula, including the component-group projector and the absolute layer
normalization.
\end{proof}

Three points remain before this reduction gives an unconditional theorem for
every geometric unequal-parameter algebra.

First, Lemma~\ref{lem:levi-coefficient-splitting} is presently stated with
an untwisted tensor decomposition
\[
 \Hh_Q\simeq
 \Hh_{Q,{\rm der}}\otimes
 \operatorname{Sym}(\mathfrak z(\mathfrak q)^*).
\]
For connected, untwisted generalized-Springer algebras, the Bernstein
relations and the K\"unneth argument in its proof give precisely this
splitting.  In a finite crossed product or a cocycle-twisted block, however,
the component group can act on the central coefficient ring.  One must
replace the displayed tensor product by the corresponding semilinear
crossed-product statement and prove that completion and the enhancement
projector still recover the algebraic universal twist family.  This is a
coefficient-family problem; it does not affect the local theorem.

Second, Proposition~\ref{prop:general-coefficient-arc} naturally compares
the costalk block for $q\mathcal E$ with the stalk block for the dual
coefficient datum $q\mathcal E^\vee$.  When the datum and its cocycle are
self-dual, this is the usual standard-to-contragredient arrow in one algebra.
Without a chosen identification with the dual block, the construction still
produces a paired Smith filtration, but it has not yet been identified in
the paper with the customary internal Jantzen filtration of a single
unequal-parameter algebra.  This dual-block identification is the main
representation-theoretic compatibility that remains.

Third, the purity check in Step~3 of the proof of
Theorem~\ref{thm:equal-parameter-jantzen} was written only for the principal coefficient
system.  For a generalized-Springer datum of geometric origin, the cuspidal
local system has finite monodromy and admits a pure polarizable realization;
the expected extension of that argument is to express
$K_{\sigma,r}$ as the proper direct image of this coefficient object on the
smooth components of the fixed incidence variety, with the prescribed
componentwise shifts and Tate twists.  This should establish (U3a) by proper
direct-image purity and the decomposition theorem.  Nevertheless, that
mixed normalization and its compatibility with
the twisted finite action must be written down before purity can be cited
uniformly.  Condition (U3b) is not an additional obstruction: once (U1) and
(U2) identify the two geometric modules with the completed standard and
contragredient lattices, Proposition~\ref{prop:general-coefficient-arc}
makes them finite free over $\mathbb C[[t]]$.

No other local input is needed for this extension.  The enhancement groups
may be nontrivial, the parameters $c_\alpha$ may be unequal, and the
positive-degree residual ideal still acts trivially on every Jantzen layer.
An arbitrary numerical unequal-parameter function without a
generalized-Springer realization is not covered: in that setting the
geometric standard module and the costalk-to-stalk map are unavailable.

\subsection{Dominant and non-dominant directions}

Theorem~\ref{thm:equal-parameter-jantzen} is stated for the standard principal Springer
realization fixed in the Introduction.  Nontrivial component-group
representations are retained by exact isotypic projection on the fixed
incidence diagram.  The polynomial
$\widehat P_{\gamma\delta}(u)$, defined in
\eqref{eq:full-ic-polynomial}, retains every local cohomological degree and
the dual enhanced labels.

Dominance is used at two specific points.  First, it makes the deformation
cocharacter strictly positive on the normal directions to the fixed stratum;
this is Proposition~\ref{prop:general-graded-slodowy}.  Second,
the same positivity identifies the equivariant parameter with the first
Chern class of an ample line bundle on the projectivized punctured
slice; see \eqref{eq:positive-coefficient-arc}.  These are precisely the
positivity inputs needed for hard Lefschetz in
Theorem~\ref{thm:general-local}.  Outside this chamber the
filtration can change, see Appendix \ref{app:counterexample}.  

\appendix
\renewcommand{\theequation}{A.\arabic{equation}}
\setcounter{equation}{0}
\section{A regular non-dominant deformation of a standard module}
\label{app:counterexample}

The example concerns a single standard module for
$\Hh(\mathrm{GL}_{16})$.  Two regular real deformations of its inducing
data give different Jantzen filtrations; one direction is dominant and the
other is not.  The geometric input is Williamson's rank-two intersection
matrix for the Schubert interval $[1234,4231]$.

\subsection{Standard modules and deformation families}

The quiver calculation is naturally written in the geometric
standard-module convention of Section~\ref{sec:fixed-slice-comparison}.  It
is important here to distinguish the geometric and analytic inducing
modules.  The geometric costalk standard is induced from anti-discrete-series
modules, on which the symmetric group acts trivially.  The
Iwahori--Matsumoto involution transports it to the analytic Langlands
standard induced from discrete-series segment modules, on which the
symmetric group acts by sign.  It also sends a geometric coefficient arc
$\mathbf c(t)$ to the sign-reversed analytic arc $-\mathbf c(t)$, without
changing the parameter $t$.  Replacing $t$ by $-t$ changes neither the Smith
exponents nor the Jantzen layers.

Write $\Hh_d=\Hh(\mathrm{GL}_d)$.  For a segment
$\Delta=[a,b]$, let $Z_{\rm an}(\Delta)$ be the corresponding analytic
essentially discrete-series $\Hh_{b-a+1}$-module.  In the present
normalization it is one-dimensional: the symmetric group acts by sign and
the polynomial weights are $(a,a+1,\ldots,b)$.  Put
\[
 \Delta_i=[i-1,i+2],\qquad
 E_i^{\rm an}=Z_{\rm an}(\Delta_i)\qquad(1\leq i\leq4),
\]
and define the geometric anti-discrete-series factor by
\[
 E_i^{\rm g}=(\operatorname{IM}_4)^*E_i^{\rm an}.
\]
The symmetric group acts trivially on $E_i^{\rm g}$, and its polynomial
weights, in the same coordinate order, are
\[
 -(i-1),\ -i,\ -(i+1),\ -(i+2).
\]
Label the polynomial weight $-k$ by the quiver vertex $k$.
Since $[\sigma,y]=y$, an arrow goes from $k$ to $k-1$.
Thus the geometric quiver is $0\leftarrow1\leftarrow\cdots\leftarrow6$,
and $E_i^{\rm g}$ is the interval representation $\Delta_i$.
For a formal scalar $c$, set
\[
 E_i^{\rm an}(c)=Z_{\rm an}([i-1+c,i+2+c]),
 \qquad
 E_i^{\rm g}(c)=(\operatorname{IM}_4)^*E_i^{\rm an}(-c).
\]
Thus $c$ is the geometric central twist, and applying
$\operatorname{IM}_4^*$ gives the analytic twist $-c$.

Define
\begin{align*}
 \mathfrak n&=\Delta_1+\Delta_2+\Delta_3+\Delta_4
 =[0,3]+[1,4]+[2,5]+[3,6],\\
 \mathfrak m&=[0,6]+[1,4]+[2,5]+[3,3].
\end{align*}
Write $\operatorname{Ind}_{\rm g}$ for induction in the geometric costalk
convention, using the opposite parabolic fixed in
Section~\ref{sec:fixed-slice-comparison}.  The geometric standard module
attached to $\mathfrak n$ is
\begin{equation}\label{eq:appendix-standard-module}
 X_{\rm g}(\mathfrak n)
 =\operatorname{Ind}_{\rm g}
   (E_1^{\rm g}\boxtimes E_2^{\rm g}\boxtimes
    E_3^{\rm g}\boxtimes E_4^{\rm g}).
\end{equation}
Its unique irreducible quotient is $L_{\rm g}(\mathfrak n)$.  Put
\[
 X_{\rm an}(\mathfrak n)=\operatorname{IM}^*X_{\rm g}(\mathfrak n),
 \qquad
 L_{\rm an}(\mathfrak n)=\operatorname{IM}^*L_{\rm g}(\mathfrak n).
\]
By Lemma~\ref{lem:analytic-geometric-transport}, this is the analytic
Langlands standard module
\[
 X_{\rm an}(\mathfrak n)
 \simeq\operatorname{Ind}
 \bigl(E_4^{\rm an}\boxtimes E_3^{\rm an}\boxtimes
       E_2^{\rm an}\boxtimes E_1^{\rm an}\bigr),
\]
whose segment centres are in decreasing order.  Its unique irreducible
quotient is $L_{\rm an}(\mathfrak n)$.  Let
$L_{\rm g}(\mathfrak m)$ be the geometric simple module labelled by the
second multisegment, and put
$L_{\rm an}(\mathfrak m)=\operatorname{IM}^*L_{\rm g}(\mathfrak m)$.

Let
\[
 \mathcal R=\mathbb C[[c_1,c_2,c_3,c_4]],\qquad
 \alpha_1=c_3-c_4,\quad \alpha_2=c_2-c_3,\quad
 \alpha_3=c_1-c_2.
\]
Over $\mathcal R$, form the geometric standard and opposite families; the
Hecke algebras and inducing modules in this display are scalar-extended to
$\mathcal R$:
\begin{align*}
 \mathcal X^+_{{\rm g},\mathbf c}
  &=\operatorname{Ind}_{\rm g}
    \bigl(E_1^{\rm g}(c_1)\boxtimes E_2^{\rm g}(c_2)\boxtimes
          E_3^{\rm g}(c_3)\boxtimes E_4^{\rm g}(c_4)\bigr),\\
 \mathcal X^-_{{\rm g},\mathbf c}
  &=\operatorname{Ind}_{\rm g}
    \bigl(E_4^{\rm g}(c_4)\boxtimes E_3^{\rm g}(c_3)\boxtimes
          E_2^{\rm g}(c_2)\boxtimes E_1^{\rm g}(c_1)\bigr).
\end{align*}
On each regular one-parameter line in this coefficient space, take
the primitive integral standard-to-opposite intertwiner.  We use these
DVR maps, not a choice of universal primitive matrix over $\mathcal R$.
Lemma~\ref{lem:analytic-geometric-transport} identifies the transform of
its source with the analytic standard family
\begin{equation}\label{eq:appendix-analytic-family}
 \operatorname{IM}^*\mathcal X^+_{{\rm g},\mathbf c}
 \simeq
 \operatorname{Ind}\bigl(
 E_4^{\rm an}(-c_4)\boxtimes E_3^{\rm an}(-c_3)\boxtimes
 E_2^{\rm an}(-c_2)\boxtimes E_1^{\rm an}(-c_1)\bigr).
\end{equation}
In particular, the analytic factors remain essentially discrete series with
sign symmetric-group action, and they occur in decreasing segment order at
$\mathbf c=0$.
The primitive integral maps along
\begin{align}
 \ell_{\rm dom}^*:(c_1,c_2,c_3,c_4)&\longmapsto(3t,2t,t,0),
 \label{eq:appendix-dominant-line}\\
 \ell_{\rm nd}^*:(c_1,c_2,c_3,c_4)&\longmapsto(2t,t,0,3t)
 \label{eq:appendix-nondominant-line}
\end{align}
are the two geometric Jantzen maps.  The source family specializes at
$t=0$ to $X_{\rm g}(\mathfrak n)$ along both lines; only the rates at which
its four inducing factors are translated differ.  Its Iwahori--Matsumoto
transform gives the corresponding analytic Jantzen maps along the
sign-reversed arcs.  In the geometric relative $A_3$ root coordinates the
two tangent directions are $(1,1,1)$ and $(-3,1,1)$.  In the decreasing
analytic factor order $(E_4^{\rm an},E_3^{\rm an},E_2^{\rm an},E_1^{\rm an})$,
formula~\eqref{eq:appendix-analytic-family} gives the directions
$(1,1,1)$ and $(-3,1,1)$, respectively.  Thus the first analytic direction
is dominant and the second is regular and non-dominant.


\subsection{The category \texorpdfstring{$\mathcal O$}{O} dictionary}

The Arakawa--Suzuki functor gives a representation-theoretic
interpretation of the two orbit labels and their deformations.
The actual local intersection map will be identified independently by
the explicit flag model below.
Work in $\mathcal O(\mathfrak{gl}_4)$, let
$V=\mathbb C^4$, and write $\rho$ for the half-sum of the positive roots.
Choose weights $\Lambda$ and $\mu_0$ by
\begin{equation}\label{eq:appendix-AS-weights}
 \Lambda+\rho=(7,6,5,4),
 \qquad
 \mu_0+\rho=(3,2,1,0).
\end{equation}
Then $\Lambda-\mu_0=(4,4,4,4)$, so
\[
 \sum_{i=1}^4(\Lambda_i-\mu_{0,i})=16.
\]
The choice $4$ makes all segments below nonempty.  More generally,
$\Lambda-\mu_0=(k,k,k,k)$ gives tensor degree $4k$ and
$\Lambda-\mu_w=(k+3,k,k,k-3)$.  The tensor-weight condition already
holds for $k=3$, when the final segment has length zero; we use $k=4$
to avoid that additional convention.

After restriction to $\mathfrak{sl}_4$, both $\Lambda$ and $\mu_0$ are the
zero weight.  Thus the category $\mathcal O$ calculation lies in the regular
integral principal block. 
 For a weight $\nu$, let $M(\nu)$ denote the Verma module of highest weight
$\nu$, and let $L(\nu)$ denote its simple quotient.  The Arakawa--Suzuki
functor attached to $\Lambda$ is
\[
 F_\Lambda(X)=
 \operatorname{Hom}_{\mathfrak{gl}_4}
 \bigl(M(\Lambda),X\otimes V^{\otimes16}\bigr),
\]
with the $\Hh(\mathrm{GL}_{16})$-action defined by the permutation and tensor
Casimir operators.  Since $\Lambda+\rho$ is dominant, this is an exact
functor.  On a Verma module it satisfies
\begin{equation}\label{eq:appendix-AS-standard}
 F_\Lambda(M(\mu))\simeq M(\Lambda,\mu),
\end{equation}
where $M(\Lambda,\mu)$ is the corresponding degenerate affine Hecke
standard module.  Moreover, because $\Lambda+\rho$ is regular, the
nonvanishing condition for the simple objects in the present regular block
implies
\begin{equation}\label{eq:appendix-AS-simple}
 F_\Lambda(L(\mu))\simeq L(\Lambda,\mu),
\end{equation}
the unique irreducible quotient of $M(\Lambda,\mu)$; see
\cite[Sections~2.2 and~3.1--3.2]{Suzuki} and the original construction
\cite{ArakawaSuzuki}.

The segment attached to the $i$-th coordinates of $(\Lambda,\mu)$ is
\begin{equation}\label{eq:appendix-AS-segment}
 \Delta_i(\Lambda,\mu)
 =[(\mu+\rho)_i,(\Lambda+\rho)_i-1].
\end{equation}
Its length is $\Lambda_i-\mu_i$.  For the pair in
\eqref{eq:appendix-AS-weights}, the four segments, in the
Arakawa--Suzuki row order, are
\[
 [3,6],\quad[2,5],\quad[1,4],\quad[0,3].
\]
Consider the multisegment:
\[
 \mathfrak n=[0,3]+[1,4]+[2,5]+[3,6].
\]

Now let $w=4231\in S_4$ and put $\mu_w=w\circ\mu_0$, where
$w\circ\mu_0=w(\mu_0+\rho)-\rho$.  Then
\[
 \mu_w+\rho=(0,2,1,3),
 \qquad
 \Lambda-\mu_w=(7,4,4,1).
\]
Formula~\eqref{eq:appendix-AS-segment} gives
\[
 [0,6],\quad[2,5],\quad[1,4],\quad[3,3],
\]
and the second multisegment defined above:
\[
 \mathfrak m=[0,6]+[1,4]+[2,5]+[3,3].
\]
Both shifted weights in \eqref{eq:appendix-AS-weights} are regular, so the
stabilizer corrections in the general Arakawa--Suzuki parametrization are
absent.  Consequently the permutation labels and the multisegment labels
match as follows:
\begin{equation}\label{eq:appendix-AS-label-dictionary}
 \begin{array}{c@{\quad\longmapsto\quad}c}
  M(\mu_0)\text{, labelled by }1234&X_{\rm g}(\mathfrak n),\\[2pt]
  L(\mu_w)\text{, labelled by }4231&L_{\rm g}(\mathfrak m).
 \end{array}
\end{equation}

Recall that  Suzuki's one-dimensional segment
module has trivial symmetric-group action and polynomial weights
$a,a+1,\ldots,b$.  His cross relation for the graded Hecke algebra is identified with the one used in
this paper by negating the polynomial generators.  The segment module then
has trivial symmetric-group action and weights
$-a,-a-1,\ldots,-b$, which is precisely the geometric
anti-discrete-series convention defined above; the quiver vertex $k$ is
labelled by
the weight $-k$.  Moreover, $\operatorname{Ind}_{\rm g}$ uses the opposite
parabolic.  Thus its displayed factor order is the reverse of the
Arakawa--Suzuki row order.  With these two convention changes,
\eqref{eq:appendix-AS-standard} and
\eqref{eq:appendix-AS-simple} give exactly
\eqref{eq:appendix-AS-label-dictionary}, not merely modules related by an
additional sign twist.

The functor also transports the contravariant form and its divisibility
filtration.  In Suzuki's integral construction one deforms
\[
 \Lambda(t)=\Lambda+t\delta,
 \qquad
 \mu(t)=\mu+t\delta.
\]
The segment lengths remain fixed, while the $i$-th segment is translated by
$t\delta_i$.  In the geometric coordinates defined above, the
Arakawa--Suzuki row order is $E_4,E_3,E_2,E_1$ and the polynomial-generator
sign change gives
\begin{equation}\label{eq:appendix-AS-direction-dictionary}
 \delta=(-c_4,-c_3,-c_2,-c_1).
\end{equation}
Thus the two geometric lines defined above correspond in category
$\mathcal O$ to
\begin{align*}
 (c_1,c_2,c_3,c_4)&=(3,2,1,0)t
 &\Longleftrightarrow\quad
 \delta_{\rm dom}&=(0,-1,-2,-3),\\
 (c_1,c_2,c_3,c_4)&=(2,1,0,3)t
 &\Longleftrightarrow\quad
 \delta_{\rm nd}&=(-3,0,-1,-2).
\end{align*}
Their simple-root coordinates are $(1,1,1)$ and $(-3,1,1)$,
respectively.  The first is dominant, up to a central translation, and the
second is regular and non-dominant.

Suzuki's directionwise comparison theorem
\cite[Theorem~4.3.5]{Suzuki} now gives, for either fixed direction,
\begin{equation}\label{eq:appendix-AS-Jantzen}
 [\operatorname{gr}_J^jM(\mu_0):L(\mu_w)]
 =
 [\operatorname{gr}_J^jX_{\rm g}(\mathfrak n):
    L_{\rm g}(\mathfrak m)].
\end{equation}
Applying the Iwahori--Matsumoto equivalence described above replaces the
geometric objects by $X_{\rm an}(\mathfrak n)$ and
$L_{\rm an}(\mathfrak m)$ without changing $t$ or the Smith exponents.
The rank-two calculation below will therefore also describe the
two-dimensional composition-multiplicity space for this category
$\mathcal O$ pair.  We do not infer the identification of the local
matrix from a coincidence of Kazhdan--Lusztig polynomials: the next
subsection identifies the equivariant normal slice explicitly.

Only the directionwise identity \eqref{eq:appendix-AS-Jantzen} is used
in this dictionary.  The later direction-independence assertion in
\cite[Proposition~5.3.2]{Suzuki} is not an input.

\subsection{The flag slice and the two integral maps}

Write the backward-quiver maps as $a_j:V_j\to V_{j-1}$, with
\[
 (\dim V_0,\ldots,\dim V_6)=(1,2,3,4,3,2,1).
\]
Both $\mathfrak n$ and $\mathfrak m$ belong to the open set on which
$a_1,a_2,a_3$ are surjective and $a_4,a_5,a_6$ are injective.
Put $E=V_3$.  Such a representation determines two complete flags:
\[
 \begin{split}
 F_1&=\ker a_3,\quad
 F_2=\ker(a_2a_3),\quad
 F_3=\ker(a_1a_2a_3),\\
 F'_1&=\operatorname{im}(a_4a_5a_6),\quad
 F'_2=\operatorname{im}(a_4a_5),\quad
 F'_3=\operatorname{im}a_4.
 \end{split}
\]
Conversely, these flags recover the representation up to choices of
isomorphisms
\[
 V_2\simeq E/F_1,\quad V_1\simeq E/F_2,\quad V_0\simeq E/F_3,
 \qquad
 V_6\simeq F'_1,\quad V_5\simeq F'_2,\quad V_4\simeq F'_3.
\]
Thus the open quiver variety is a principal bundle over
$\mathrm{Fl}(E)\times\mathrm{Fl}(E)$ for
$\prod_{j\ne3}\mathrm{GL}(V_j)$.  These bundle directions lie in the
base-change orbits and do not occur in a normal slice.

At $\mathfrak n$, let $e_i$ be the central-vertex vector belonging to
$\Delta_i$.  Both flags are
\[
 \langle e_4\rangle
 \subset\langle e_4,e_3\rangle
 \subset\langle e_4,e_3,e_2\rangle.
\]
At $\mathfrak m$, index the four central vectors by the segment starts
$0,1,2,3$, again calling them $e_1,e_2,e_3,e_4$.
The first flag is unchanged and the second is
\[
 \langle e_1\rangle
 \subset\langle e_1,e_3\rangle
 \subset\langle e_1,e_3,e_2\rangle.
\]
Relative to the ordered basis $(e_4,e_3,e_2,e_1)$ its permutation is
$4231$.  Fixing the first flag therefore identifies the normal
singularity of $\overline{\mathcal O_{\mathfrak m}}$ along
$\mathcal O_{\mathfrak n}$ with the identity chart in the Schubert
variety $X_{4231}\subset\mathrm{Fl}_4$.
This is an explicit slice construction, not a cancellation inferred
only from orbit dimensions or ordinary multiplicities.

The stabilizer torus scales the $e_i$.  Its coefficient arc is
$\lambda_i\mapsto c_i$.  The lower-triangular chart in the ordered
basis $(e_4,e_3,e_2,e_1)$ has weights
$c_i-c_j$ for $i<j$, and its consecutive weights are precisely
\[
 \alpha_1=c_3-c_4,\qquad
 \alpha_2=c_2-c_3,\qquad
 \alpha_3=c_1-c_2.
\]
The torus action and its two coefficient lines are therefore determined
by the flag construction.  In particular $\ell_{\rm dom}$ contracts
this normal chart.

Williamson's local intersection calculation
\cite[Section~8.3]{WilliamsonLocal}, in homogeneous integral bases, is
\begin{equation}\label{eq:appendix-matrix}
 B_{4231}=\begin{pmatrix}
 \alpha_1\alpha_2\alpha_3^2(\alpha_1+\alpha_2)&
 \alpha_1\alpha_2\alpha_3(\alpha_1+\alpha_2)\\
 \alpha_1\alpha_2\alpha_3(\alpha_1+\alpha_2)&
 -\alpha_1\alpha_3(\alpha_1+2\alpha_2+\alpha_3)
 \end{pmatrix}.
\end{equation}
If positive roots are taken as the cotangent rather than the tangent
characters, all three $\alpha_i$ change sign.  This replaces $B$ by
$-DBD$, where $D=\operatorname{diag}(1,-1)$, and so makes no difference
to integral equivalence or Smith exponents.  Its determinant is
\begin{equation}\label{eq:appendix-det}
 \det B_{4231}=-\alpha_1^2\alpha_2\alpha_3^2
 (\alpha_1+\alpha_2)(\alpha_2+\alpha_3)
 (\alpha_1+\alpha_2+\alpha_3).
\end{equation}

Here a \emph{multiplicity map} can be defined without assuming
semisimplicity of the non-dominant layers.  Choose a primitive idempotent
in the matrix factor $\operatorname{End}(V_{\mathfrak m})$ of
$(\mathcal E_\chi^{\rm res})_0$ and lift it using the IC decomposition
of the equivariant Springer complex.  Applying it to both normal
lattices and their map isolates the rank-two IC coefficient map.
On the special fibre this idempotent has rank one on
$L_{\rm g}(\mathfrak m)$ and rank zero on the other simple modules.
Since its image functor is exact, its dimensions on any layer compute
that layer's composition multiplicity of $L_{\rm g}(\mathfrak m)$.

\begin{lemma}[Directionwise comparison]\label{lem:appendix-transport}
On each of the two lines
\eqref{eq:appendix-dominant-line}--\eqref{eq:appendix-nondominant-line},
the $L_{\rm g}(\mathfrak m)$ multiplicity map of the primitive
Jantzen intertwiner is integrally equivalent over $\mathbb C[[t]]$
to the corresponding specialization of \eqref{eq:appendix-matrix}.
\end{lemma}

\begin{proof}
The flag construction identifies the equivariant normal IC map.
Indeed, pullback along the displayed principal bundle and variation
of the first flag add only orbit directions.  Normal restriction
\eqref{eq:normal-slice-arrow} removes those directions by orbit
adjunction, without dividing any Euler class.  It leaves the vertex
map on the Schubert chart, with the torus characters just computed.

Now restrict to either coefficient line.  The normal IC summand is
pure; vertex freeness from Lemma
\ref{lem:normal-purity-base-change} makes this an ordinary coefficient
base change.  All six root values are nonzero on either line, so the
normal map is generically invertible, also on the non-dominant line.
Proposition~\ref{prop:general-coefficient-arc}(2) compares the full
primitive normal map with the algebraic intertwiner by a DVR unit.
Taking the $\mathfrak m$ multiplicity component gives the assertion.

Primitivity is imposed on the full intertwiner, not separately on
this rank-two component.  Its positive Smith exponents must not be
removed by dividing the component by a power of $t$.  No universal
primitive-matrix equivalence over
$\mathbb C[[c_1,c_2,c_3,c_4]]$ is asserted.
\end{proof}

\begin{lemma}[Smith exponents on the two lines]\label{lem:appendix-smith}
Over $\C[[t]]$, the Smith exponents of
$B_{4231}$ are $(3,5)$ along $\ell_{\rm dom}$ and $(4,4)$ along
$\ell_{\rm nd}$.
\end{lemma}

\begin{proof}
On the two lines the matrices are, respectively,
\[
 \begin{pmatrix}2t^5&2t^4\\2t^4&-4t^3\end{pmatrix},
 \qquad
 \begin{pmatrix}6t^5&6t^4\\6t^4&0\end{pmatrix}.
\]
Let $a_1\leq a_2$ be the Smith exponents.  For a matrix
$C=(C_{ij})\in M_2(\mathbb C[[t]])$,
\[
 a_1=\min_{i,j}v_t(C_{ij}),\qquad
 a_1+a_2=v_t(\det C).
\]
Along $\ell_{\rm dom}$ these two quantities are $3$ and $8$, so
$(a_1,a_2)=(3,5)$.  Along $\ell_{\rm nd}$ they are $4$ and $8$, so
$(a_1,a_2)=(4,4)$.  In the order
\[
 \alpha_1,\ \alpha_2,\ \alpha_3,\ \alpha_1+\alpha_2,\
 \alpha_2+\alpha_3,\ \alpha_1+\alpha_2+\alpha_3,
\]
the coefficients of $t$ in the six positive-root values along
$\ell_{\rm nd}$ are $-3,1,1,-2,2,-1$.  They are all nonzero, so this line is
regular.
\end{proof}

\begin{remark}[Consistency with the IC prediction]
The base-change group has dimension $44$, while
$\dim\operatorname{End}(\mathfrak n)=10$ and
$\dim\operatorname{End}(\mathfrak m)=5$.  Hence
\[
 \dim\mathcal O_{\mathfrak n}=34,\qquad
 \dim\mathcal O_{\mathfrak m}=39.
\]
Williamson's calculation gives
\[
 \widehat P_{\mathfrak n,\mathfrak m}(u)=1+u^2,
\]
so $[X_{\rm g}(\mathfrak n):L_{\rm g}(\mathfrak m)]=2$.  Since the
orbit-dimension difference is $5$, Theorem
\ref{thm:equal-parameter-jantzen}, together with Lemma
\ref{lem:analytic-geometric-transport}, places the two copies in dominant
Jantzen levels $5$ and $3$.  This agrees with
Lemma~\ref{lem:appendix-smith} and fixes the grading normalization.
\end{remark}

\begin{theorem}[Failure of direction independence]
\label{thm:direction-counterexample}
The Jantzen filtration of the analytic Langlands standard module
$X_{\rm an}(\mathfrak n)$ for $\Hh(\mathrm{GL}_{16})$ depends on the regular
real deformation direction, even up to an overall translation of filtration
indices.
\end{theorem}

\begin{proof}
Consider the two geometric specializations $\ell_{\rm dom}$ and
$\ell_{\rm nd}$ from
\eqref{eq:appendix-dominant-line}--\eqref{eq:appendix-nondominant-line}.
The first lies in the contracting chamber.  After the sign change of
Lemma~\ref{lem:analytic-geometric-transport}, it gives the dominant analytic
Langlands deformation.  For the second geometric line, the six positive-root
values have nonzero coefficients
$-3,1,1,-2,2,-1$, as recorded in the proof of
Lemma~\ref{lem:appendix-smith}.  In decreasing analytic factor order, the
corresponding coefficients are $-3,1,1,-2,2,-1$.  Thus the second analytic
line is regular and non-dominant.

The directionwise comparison in
Lemma~\ref{lem:appendix-transport} applies to both regular lines.  It therefore
identifies both specialized Jantzen blocks, up to integral changes of basis
and a unit, with the corresponding specialization of $B_{4231}$.  These
operations preserve Smith exponents.  Lemma
\ref{lem:analytic-geometric-transport} then transports both filtrations to
the analytic standard module, again without changing their Smith exponents.
Lemma~\ref{lem:appendix-smith} gives
multisets $\{3,5\}$ and $\{4,4\}$ for the two copies of
$L_{\rm an}(\mathfrak m)$.  The gap is $2$ in the first filtration and $0$ in the
second.  A common translation cannot change this gap, so the filtrations are
not equal up to translation.  Equivalently, along the dominant deformation
one copy of $L_{\rm an}(\mathfrak m)$ occurs in each of the third and fifth
Jantzen layers, whereas along the regular non-dominant deformation the
entire two-dimensional $L_{\rm an}(\mathfrak m)$-multiplicity space has
Smith exponent four.
The latter statement concerns the multiplicity space and does not assume
that the fourth non-dominant layer is semisimple.
\end{proof}

\renewcommand{\theequation}{B.\arabic{equation}}
\setcounter{equation}{0}
\section{The \texorpdfstring{$A_1$}{A1} model: one nonzero Jantzen layer}
\label{app:rank-one}

This appendix illustrates the general proof in the smallest nontrivial
principal block.  The calculation contains the primitive integral
intertwiner, the contracting fixed Springer slice, the Lefschetz
string--Smith exponent correspondence, and the residual-radical argument
for semisimplicity.  It is a consistency check, not a second type-$A$ proof.
The principal datum uses the trivial enhancement on both orbits.
For the group $\mathrm{SL}_2$ used below, the nonzero-orbit stabilizer
in the diagonal torus is $\{\pm I\}$, so its component group is
$\mu_2$, not the trivial group.  Its trivial character is the only one
occurring in this calculation.

\subsection{The algebraic family and its Smith form}

Let $R=\{\pm\alpha\}$, put
\[
 \xi=\frac{\alpha}{2},\qquad s=N_{s_\alpha},
\]
and normalize the parameter as in the Introduction.  Then
$s(\xi)=-\xi$ and $\langle\xi,\alpha^\vee\rangle=1$, so
\[
 \Hh(A_1,1)
 =\mathbb C\langle s,\xi\rangle/
   (s^2-1,\ \xi s+s\xi-1).
\]
The transpose anti-involution is $s^*=s$ and $\xi^*=\xi$.

For $\lambda\in\mathbb C$, let $\mathbb C_\lambda$ be the
$\mathbb C[\xi]$-module on which $\xi$ acts by $\lambda$, and put
\[
 X_\lambda=\Hh(A_1,1)\otimes_{\mathbb C[\xi]}\mathbb C_\lambda.
\]
With $v=1\otimes1$ and $w=s\otimes1$, the action is
\[
 sv=w,\qquad sw=v,\qquad
 \xi v=\lambda v,\qquad \xi w=v-\lambda w.
\]
Thus, in the ordered basis $(v,w)$,
\begin{equation}\label{eq:A1-action}
 [s]=\begin{pmatrix}0&1\\1&0\end{pmatrix},
 \qquad
 [\xi]=\begin{pmatrix}\lambda&1\\0&-\lambda\end{pmatrix}.
\end{equation}

Consider the dominant analytic arc
\[
 \lambda(t)=\frac12+t.
\]
The standard lattice of \eqref{eq:standard-lattice-definition} is, in this
case,
\[
 X_\A=
 (\A\otimes_{\mathbb C}\Hh(A_1,1))
 \otimes_{\A[\xi]}\A_{\lambda(t)},
\]
where $\xi$ acts on the rank-one module $\A_{\lambda(t)}$ by
$\lambda(t)$.  It is free over $\A$ with basis $(v,w)$.

The primitive standard-to-contragredient map is represented in $(v,w)$ and
the corresponding dual basis by
\begin{equation}\label{eq:A1-intertwiner}
 G_t=
 \begin{pmatrix}2\lambda(t)&1\\1&2\lambda(t)\end{pmatrix}
 =\begin{pmatrix}1+2t&1\\1&1+2t\end{pmatrix}.
\end{equation}
Indeed, if $S$ and $\Xi$ denote the two matrices in
\eqref{eq:A1-action}, then
\[
 G_tS=S^{\mathsf T}G_t,
 \qquad
 G_t\Xi=\Xi^{\mathsf T}G_t.
\]
These are precisely the two $\Hh$-linearity identities for the chosen
transpose anti-involution.  The generic Hom-space is one-dimensional, so
they determine the standard-to-contragredient map up to a scalar.  The
displayed integral normalization is primitive because its entries generate
the unit ideal of $\A$, while
\[
 \det G_t=4t(1+t),
\]
so it becomes an isomorphism over $\K$ but not over $\A$.

\begin{samepage}
Put
\[
 P=\begin{pmatrix}1&1\\1&-1\end{pmatrix}\in\operatorname{GL}_2(\A).
\]
Since $2$ is a unit in $\A$, the following independent source and target
basis changes give a Smith equivalence:
\[
 P^{-1}G_tP
 =\begin{pmatrix}2+2t&0\\0&2t\end{pmatrix}.
\]
After multiplying the two target basis vectors by units, the Smith normal
form is therefore
\begin{equation}\label{eq:A1-smith}
 \operatorname{diag}(1,t).
\end{equation}
In particular, the two Smith exponents are $0$ and $1$.
\end{samepage}

Set $v_+=v+w$ and $v_-=v-w$.  At $t=0$, the line
$\mathbb Cv_-$ is a submodule, because
\[
 sv_-=-v_-,\qquad \xi v_-=-\frac12v_-.
\]
For $\varepsilon\in\{+1,-1\}$, let $L_\varepsilon$ be the one-dimensional
module on which
\[
 s\longmapsto\varepsilon,
 \qquad
 \xi\longmapsto\frac{\varepsilon}{2}.
\]
The quotient of $X_{1/2}$ by $\mathbb Cv_-$ is $L_+$, and one obtains the
nonsplit exact sequence
\[
 0\longrightarrow L_-
 \longrightarrow X_{1/2}
 \longrightarrow L_+\longrightarrow0.
\]
The definition \eqref{eq:jantzen} and \eqref{eq:A1-smith} give
\begin{equation}\label{eq:A1-layers}
 J^0X_{1/2}=X_{1/2},\qquad
 J^1X_{1/2}=\mathbb Cv_-,\qquad J^2X_{1/2}=0,
\end{equation}
and hence
\[
 \operatorname{gr}_J^0X_{1/2}\simeq L_+,
 \qquad
 \operatorname{gr}_J^1X_{1/2}\simeq L_-.
\]

\subsection{The fixed Springer line and the Lefschetz string}

The same exponents arise geometrically.  Take
$G=\operatorname{SL}_2(\mathbb C)$, let $(e,h,f)$ be the standard
$\mathfrak{sl}_2$-triple with $h=\alpha^\vee$, and keep $r=\tfrac12$.
The analytic central point and dominant direction used above are
\[
 \sigma_{\rm an}=\frac12h,\qquad \eta_{\rm an}=h.
\]
Under the analytic/geometric convention of Section~\ref{sec:fixed-slice-comparison}, the
corresponding geometric data are
\[
 \sigma_{\rm g}=-\frac12h,\qquad \eta_{\rm g}=-h.
\]
The parameter $t$ is unchanged.  Since $2r=1$, the fixed graded piece is
\[
 \mathfrak g_{\sigma_{\rm g},1}=\mathbb Cf\simeq\mathbb A^1,
 \qquad G_{\sigma_{\rm g}}=T.
\]
The cocharacter $\eta_{\rm g}$ acts on $f$ with weight $2$, and hence
contracts this line to the origin.

Let $B^+$ and $B^-$ be the two $T$-fixed Borel subgroups, with
$f\in\operatorname{Lie}B^-$.  The fixed principal incidence variety over
$\mathbb Cf$ is the disjoint union
\[
 \widetilde X_{\sigma_{\rm g},r}
 =\bigl(\mathbb A^1\times\{B^-\}\bigr)
  \sqcup\bigl(\{0\}\times\{B^+\}\bigr).
\]
The first component maps isomorphically to the line, while the second maps
to its origin.  Consequently, with the componentwise perverse shifts and
the fixed Tate normalizations used in the paper,
\begin{equation}\label{eq:A1-Springer-complex}
 K_{\sigma_{\rm g},r}
 \simeq \mathbb C_{\{0\}}\oplus\mathbb C_{\mathbb A^1}[1].
\end{equation}

Let $i:\{0\}\hookrightarrow\mathbb A^1$.  On the point-supported summand
of \eqref{eq:A1-Springer-complex}, the canonical map $i^!\to i^*$ is the
identity.  On the smooth-line summand it is the equivariant Gysin map, hence
multiplication by the Euler class of the normal line.  If
$z\in H^2_{A_{\eta_{\rm g}}}(\mathrm{pt})$ is the positive generator for
the contracting torus, there are homogeneous generators, ordered as in
\eqref{eq:A1-Springer-complex}, in which
the full canonical map is
\begin{equation}\label{eq:A1-geometric-map}
 c=\operatorname{diag}(1,z).
\end{equation}
Its torsion cokernel is $\mathbb C[z]/(z)$.  Equivalently, the quotient of
the punctured line by the effective torus is a point, and Lemma
\ref{lem:cone-calculation} identifies this cokernel, with its torsor shift,
with the cohomology of that point.  Hard Lefschetz therefore supplies one
primitive string of length $1$.  (Before effectivizing, the quotient stack
is $B\mu_2$; exact coarse pushforward gives the same calculation.)

Along the formal arc, Proposition~\ref{prop:general-coefficient-arc} gives
\[
 z\longmapsto c_\eta t,
 \qquad c_\eta>0.
\]
For the original coroot cocharacter the normal weight is $2$, so one may
take $c_\eta=2$.  Passing to the effective quotient merely rescales the
positive generator; in either convention $c_\eta>0$.  Thus, in the coroot
coordinate, \eqref{eq:A1-geometric-map} pulls back to
\[
 c^\eta\sim\operatorname{diag}(1,2t),
\]
where $\sim$ denotes equivalence under invertible source and target
$\A$-basis changes.  This map has the same Smith exponents $(0,1)$ as
\eqref{eq:A1-intertwiner}.

\subsection{The integral comparison and the residual radical}

Write $I_{\rm alg}$ for the map represented by
\eqref{eq:A1-intertwiner}, and $I_{\rm geo}$ for the pullback of
\eqref{eq:A1-geometric-map}.  After the fixed comparison isomorphisms
$\Phi_!$ and $\Phi_*$, their generic fibres are proportional.  Both maps
are primitive, so their contents are $\A$.  If
$I_{\rm geo}=aI_{\rm alg}$ with $a\in\K^\times$, then
\[
 \A=\operatorname{cont}(I_{\rm geo})
 =a\operatorname{cont}(I_{\rm alg})=a\A.
\]
Hence $a\in\A^\times$, exactly as in Proposition
\ref{prop:general-coefficient-arc}.  The two raw matrices need not be scalar
multiples: they also record the invertible basis changes $\Phi_!$ and
$\Phi_*$.

Second, semisimplicity is a statement about the residual special fibre, not
only about the Smith normal form.  The one-variable algebra acting before
reduction is
\[
 \mathcal B_\eta=
 \frac{\A\langle s,\xi\rangle}
 {s^2-1,\ \xi s+s\xi-1,\
  \xi^2-(\frac12+t)^2}.
\]
It is free of rank $4$ over $\A$ and specializes to the exact central fibre
\[
 \mathcal E=
 \Hh(A_1,1)/(\xi^2-\tfrac14).
\]
Put
\[
 n=\xi-\frac12s.
\]
In $\mathcal B_\eta$ one has $n^2=t(1+t)$, whereas in $\mathcal E$ one has
\[
 n^2=0,\qquad sn=-ns.
\]
Thus $\mathcal E$ has the residual nonnegative grading
\[
 \mathcal E_0=\operatorname{span}_{\mathbb C}\{1,s\}
   \simeq\mathbb Ce_+\oplus\mathbb Ce_-,
 \qquad
 \mathcal E_1=\mathbb Cn\oplus\mathbb Csn,
 \qquad \mathcal E_m=0\quad(m\geq2),
\]
where $e_\pm=(1\pm s)/2$.  Its positive-degree ideal is square-zero and
\begin{equation}\label{eq:A1-radical}
 \operatorname{rad}\mathcal E
 =\mathcal E_{>0}=\mathbb Cn\oplus\mathbb Csn.
\end{equation}
Under the Springer decomposition \eqref{eq:A1-Springer-complex}, the
idempotents $e_+$ and $e_-$ are the degree-zero projectors onto the two IC
summands, while $n$ and $sn$ are the two off-diagonal residual degree-one
Ext classes.  This is the residual grading of
Lemma~\ref{lem:residual-grading}; the original polynomial grading is not
preserved by specialization.

On the special fibre $X_{1/2}$ one computes
\[
 nv_+=v_-,\qquad nv_-=0.
\]
Consequently
\[
 (\operatorname{rad}\mathcal E)X_{1/2}
 =\mathbb Cv_-=J^1X_{1/2}.
\]
The radical sends $J^0$ into $J^1$ and kills $J^1$.  It therefore acts
trivially on each of the two quotients in
\eqref{eq:A1-layers}; each quotient is a module for the semisimple algebra
$\mathcal E_0$.  This is the rank-one form of the argument in Proposition
\ref{prop:ext-equivariance}.  Although $nv_+=v_-$ on the total graded
$\mathcal E$-module
$\operatorname{gr}_JX_{1/2}$, the degree-one element $n$ maps
$\operatorname{gr}_J^0X_{1/2}$ to $\operatorname{gr}_J^1X_{1/2}$, whereas
its induced endomorphism of either individual layer is zero.

\subsection{The local IC formula}

Finally, let $\mathcal O_0=\{0\}$ and
$\mathcal O_1=\mathbb C^\times f$, and put
$\gamma_i=(\mathcal O_i,\mathbf1)$.  Both enhancements in the principal datum are trivial, and
\[
 d_0=0,\qquad d_1=1,
 \qquad
 \IC_0=\mathbb C_{\{0\}},\qquad
 \IC_1=\mathbb C_{\mathbb A^1}[1].
\]
At the origin,
\[
 i^!\IC_0=\mathbb C,
 \qquad
 i^!\IC_1=\mathbb C[-1](-1),
 \qquad
 i^*\IC_1=\mathbb C[1].
\]
The first two costalks place the corresponding constituents in Smith layers
$0$ and $1$, respectively.  The stalk definition
\eqref{eq:full-ic-polynomial} gives
\[
 \widehat P_{00}(u)=1,
 \qquad
 \widehat P_{01}(u)=1.
\]
Under the principal analytic convention,
$X_{\gamma_0}=X_{1/2}$, $L_{\gamma_0}=L_+$, and
$L_{\gamma_1}=L_-$.  Abbreviate these modules by $X_0,L_0,L_1$.
Hence
Theorem~\ref{thm:equal-parameter-jantzen} reads
\[
 \sum_{n\geq0}[\operatorname{gr}_J^nX_0:L_0]v^n=1,
 \qquad
 \sum_{n\geq0}[\operatorname{gr}_J^nX_0:L_1]v^n
 =v^{d_1-d_0}\widehat P_{01}(v^{-1})=v.
\]
This is the algebraic filtration in
\eqref{eq:A1-layers}.  The example therefore verifies both the
orbit-dimension translation and the distinction between the exponent-zero
head block and the positive Lefschetz string.

\subsection{Why the map is taken normally to the orbit}

At a nonzero point of the Springer line the orbit is open and its
transverse slice is a point.  Normal restriction of its IC complex
therefore gives the identity map, hence Smith exponent zero.
The ordinary ambient point map would instead contain the Euler class
of the orbit-tangent line.  The point stabilizer acts trivially on
that line, so this Euler class is zero.  Thus even in rank one the
ambient point map cannot replace the normal map.

For a positive-dimensional normal slice, take
\[
 G=\mathrm{SL}_3,\quad r=\tfrac12,\quad
 \sigma=\operatorname{diag}(1,0,-1),\quad y=E_{12}.
\]
Then $X_\chi=\mathbb CE_{12}\oplus\mathbb CE_{23}$.
The triple with $f=E_{21}$ and $h=\operatorname{diag}(1,-1,0)$ has
slice $V=y+\mathbb CE_{23}$, and
\[
 \eta(a)=\operatorname{diag}(a,a,a^{-2})
\]
has weight zero on the orbit-tangent line $\mathbb CE_{12}$ and
weight $3$ on the slice direction.  For the full-support object
$\mathbb C_{X_\chi}[2]$, the ambient point map has Euler factor
$(0z)(3z)=0$.  The normal map has factor $3z$, giving Smith exponent
one after $z=t$.  This is precisely the distinction implemented by
\eqref{eq:normal-fibre-map}--\eqref{eq:normal-slice-arrow}.

\medskip
\noindent\textbf{Disclosure of computational assistance.}
The author used ChatGPT for proof exploration and for assistance with
drafting and checking references.  The author is responsible for verifying the
arguments and citations and for the contents of the paper.

In a LinkedIn post on 5 August 2026, the author shared an initial version of the paper along with the transcript of the AI interaction. That version had 22 pages, including an independent proof for type $A$. The current version of the paper has been substantially expanded, clarified and corrected, and does not include the separate quiver-variety proof for $GL_N$.


\end{document}